\documentclass[12pt]{amsart}
\usepackage{amssymb}
\usepackage{graphicx}
\usepackage[all]{xy}
\usepackage[margin=1in]{geometry}
\usepackage[hidelinks]{hyperref}

\newtheorem{theorem}{Theorem}[section]
\newtheorem{corollary}[theorem]{Corollary}
\newtheorem{lemma}[theorem]{Lemma}
\newtheorem{proposition}[theorem]{Proposition}

\theoremstyle{definition}
\newtheorem{definition}[theorem]{Definition}
\newtheorem{remark}[theorem]{Remark}
\newtheorem{example}[theorem]{Example}

\numberwithin{equation}{section}

\DeclareMathOperator{\Alex}{Alex}
\DeclareMathOperator{\Alg}{Alg}
\DeclareMathOperator{\Cat}{Cat}
\DeclareMathOperator{\Hom}{Hom}
\DeclareMathOperator{\id}{id}
\DeclareMathOperator{\im}{im}
\DeclareMathOperator{\Ob}{Ob}
\DeclareMathOperator{\op}{op}
\DeclareMathOperator{\Path}{Path}
\DeclareMathOperator{\Quiv}{Quiv}
\DeclareMathOperator{\un}{un}

\newcommand{\A}{\mathcal{A}}
\newcommand{\B}{\mathcal{B}}
\newcommand{\Bcl}[1]{B^{\mathcal{CL}}_{#1}(\Sc)}
\newcommand{\Cc}{\mathcal{C}}
\newcommand{\CL}[1]{\mathcal{CL}_{#1}}
\newcommand{\de}{\partial}
\newcommand{\F}{\mathbb{F}}
\newcommand{\gammahole}[2]{\gamma_{#1,#2}}
\newcommand{\gloneone}{\mathcal{U}_q(\mathfrak{gl}(1|1))}
\newcommand{\Ib}{\mathbf{I}}
\newcommand{\IdemRing}{\mathbb{I}}
\newcommand{\Jb}{\mathbf{J}}
\newcommand{\lda}{\lambda}
\newcommand{\m}{\mathfrak{m}}
\newcommand{\mc}{\mathcal}
\newcommand{\mi}{\mu}
\newcommand{\ring}{\Bbbk}
\newcommand{\sac}[3]{\mathcal{A}(#1, #2, #3)}
\newcommand{\Sc}{\mathcal{S}}
\newcommand{\set}[1]{\left\{#1\right\}}
\newcommand{\sltwo}{\mathcal{U}_q(\mathfrak{sl}(2))}
\newcommand{\sm}{\setminus}
\newcommand{\td}{\widetilde}
\newcommand{\x}{\mathbf{x}}
\newcommand{\y}{\mathbf{y}}
\newcommand{\z}{\mathbf{z}}
\newcommand{\Z}{\mathbb{Z}}

\renewcommand{\emptyset}{\varnothing}

\renewcommand{\tilde}{\widetilde}
\renewcommand{\theta}{\vartheta}

\begin{document}

\author[Andrew Manion]{Andrew Manion}
\thanks {AM was supported by an NSF MSPRF fellowship, grant number DMS-1502686.}
\address{Department of Mathematics, USC, 3620 S. Vermont Ave., Los Angeles, CA 90089}
\email{amanion@usc.edu}

\author[Marco Marengon]{Marco Marengon}
\address {Department of Mathematics, UCLA, 520 Portola Plaza, Los Angeles, CA 90095}
\email {marengon@math.ucla.edu}

\author[Michael Willis]{Michael Willis}
\thanks {MW was supported by the NSF grant DMS-1563615.}
\address {Department of Mathematics, UCLA, 520 Portola Plaza, Los Angeles, CA 90095}
\email {mike.willis@math.ucla.edu}

\title[Generators, relations, and homology for Ozsv{\'a}th--Szab{\'o}'s algebras]{Generators, relations, and homology for Ozsv{\'a}th--Szab{\'o}'s Kauffman-states algebras}

\date{}

\begin{abstract} 
We give a generators-and-relations description of differential graded algebras recently introduced by Ozsv{\'a}th and Szab{\'o} for the computation of knot Floer homology. We also compute the homology of these algebras and determine when they are formal.
\end{abstract}
\maketitle

\section{Introduction}

Heegaard Floer homology \cite{HFOrig, PropsApps} is a powerful family of invariants for $3$- and $4$-manifolds. It originated from the study of Seiberg--Witten theory and Donaldson theory, although its methods involve holomorphic curves rather than gauge theory, and it shares these theories' applicability to the exotic world of smooth $4$-manifolds. Compared with its gauge-theoretic relatives, Heegaard Floer homology is often the easiest for computations, and many forms of Heegaard Floer homology have now been given combinatorial definitions.

One form of Heegaard Floer homology, called knot Floer homology (or $HFK$), assigns graded abelian groups to knots and links in $3$-manifolds \cite{OSzHFK, RasmussenThesis}. Like Khovanov homology \cite{KhOrig}, $HFK$ is especially well-adapted to the study of problems in knot theory with a $4$-dimensional character, such as the structure of the knot concordance group. There are many interesting similarities between $HFK$ and Khovanov homology; for example, while the Euler characteristic of Khovanov homology is the Jones polynomial, the Euler characteristic of $HFK$ is the Alexander polynomial. 

Combined with constructions of the Jones and Alexander polynomial from the representation theory of $\sltwo$ and $\gloneone$ respectively, this analogy suggests a close link between Heegaard Floer homology and categorifications of the Witten--Reshetikhin--Turaev topological quantum field theory (TQFT) invariants, see e.g. \cite{KPKH, DGR}. Indeed, both Donaldson--Floer theories in $4$ dimensions and Witten--Reshetikhin--Turaev theories in $3$ dimensions were initial motivations for the mathematical study of TQFTs, and Heegaard Floer homology offers a promising framework for understanding the relationship between these two types of theories.

Among Heegaard Floer theories, $HFK$ admits an especially wide variety of combinatorial descriptions, some allowing very fast computations. In particular, Ozsv{\'a}th--Szab{\'o} have a computer program \cite{HFKCalc} that can compete with Bar-Natan's fast Khovanov homology program \cite{FastKh}. Ozsv{\'a}th--Szab{\'o}'s program can quickly compute $HFK$ for most knots with up to around $40$ or $50$ crossings, and can even handle the larger $90+$ crossing examples from the paper \cite{ManMachine}. 

Ozsv{\'a}th--Szab{\'o}'s program is based on an exciting new description of $HFK$ \cite{OSzNew,OSzNewer,OSzHolo,OSzPong} in the algebraic language of bordered Floer homology, an extended TQFT approach to Heegaard Floer homology. We will refer to Ozsv{\'a}th--Szab{\'o}'s theory here as the \emph{Kauffman-states functor}, since to tangles it assigns bimodules whose tensor product for a closed knot projection is a complex with generators in bijection with Kauffman states for the projection as defined in \cite{FKT}.

The Kauffman states are a very natural set of generators; some readers may be more familiar with them as spanning trees of the Tait graph of a knot projection. Ozsv{\'a}th--Szab{\'o}'s crossing bimodules have an equally natural set of generators: by the results of \cite{ManionDecat}, they are in bijection with nonzero matrix entries in a certain canonical-basis representation of the $\gloneone$-linear map associated to the crossing. Thus, the Kauffman-states functor yields a categorification of this $\gloneone$ representation theory that is ``minimal'' in some sense. 

In this paper, we study the algebras $\B(n,k,\Sc)$ over which Ozsv{\'a}th--Szab{\'o} define their tangle bimodules in \cite{OSzNew}. These are defined by taking quotients of and adjoining variables to a set of algebras that Ozsv\'ath--Szab\'o call $\B_0(n,k)$; the algebras $\B_0(n,k)$ also appear in Alishahi--Dowlin's recent work \cite{AlishahiDowlin}. We start by giving a description of $\B_0(n,k)$ in terms of generators and relations.
\begin{theorem}\label{thm:IntroQuiverDescriptionB0}
The algebra $\B_0(n,k)$ is isomorphic to the path algebra of the quiver $\Gamma(n,k)$ of Definition~\ref{def:B_0 Quiver Algebra} modulo the two-sided ideal generated by the set of relations given there.
\end{theorem}
We prove Theorem~\ref{thm:IntroQuiverDescriptionB0} by defining explicit isomorphisms between the two sides, illustrated with figures. From this description of $\B_0(n,k)$ we deduce a description of the algebras $\B(n,k,\Sc)$, stated below.
\begin{theorem}\label{thm:IntroQuiverDescription}
The (differential graded) algebra $\B(n,k,\Sc)$ is isomorphic to the path algebra of the quiver $\Gamma(n,k,\Sc)$ of Definition~\ref{def:BnksQuiverDescription}, modulo the two-sided ideal generated by the set of relations given in Definitions~\ref{def:B_0 Quiver Algebra}, \ref{def:B Quiver Algebra}, and \ref{def:BnksQuiverDescription}, with differential given in Definition~\ref{def:BnksQuiverDescription} and gradings given in Section~\ref{sec:OSz gradings}.
\end{theorem}

Theorem~\ref{thm:IntroQuiverDescription} generalizes the path-algebra descriptions of $\B(n,k,\Sc)$ for $n = 1,2$ given in \cite[Section 3.5]{OSzNew}. Theorems~\ref{thm:IntroQuiverDescriptionB0} and \ref{thm:IntroQuiverDescription} have already seen use in \cite{ManionDecat, ManionKS}, as well as in \cite[Section 4.1]{AlishahiDowlin}. Theorem~\ref{thm:IntroQuiverDescription} will be especially useful in \cite{MMW2}, where we use it to define a quasi-isomorphism from $\B(n,k,\Sc)$ to a certain generalized strands algebra $\A(n,k,\Sc)$ as discussed in the motivational section below. 

We show how to define Ozsv{\'a}th--Szab{\'o}'s two algebra symmetries $\mathcal{R}$ and $o$ in terms of quiver generators and relations. We also give quiver descriptions for idempotent-truncated versions of Ozsv{\'a}th--Szab{\'o}'s algebras; see Proposition~\ref{prop:QuiverDescriptionTruncated}. Derived categories of these truncations were shown to categorify representations of $\gloneone$ in \cite{ManionDecat}.

Our next result computes the homology of $\B(n,k,\Sc)$.
\begin{theorem}\label{thm:IntroHomology}
As a chain complex, $\B(n,k,\Sc)$ is a direct sum of complexes $\Ib_{\x} \B(n,k,\Sc) \Ib_{\y}$ for $\x,\y$ in a finite set $V(n,k)$ defined in Section~\ref{sec:Istates}, and a basis for $H_*(\Ib_{\x} \B(n,k,\Sc) \Ib_{\y})$ is as described in Theorem~\ref{thm:OSzHomology}.
\end{theorem}
Theorem~\ref{thm:IntroHomology} allows us to compute the homology of the truncated algebras as well. Finally, we determine when $\B(n,k,\Sc)$ is formal.

\begin{theorem}\label{thm:IntroFormality}
The differential graded algebra $\B(n,k,\Sc)$ is formal if and only if $\Sc = \varnothing$ or $k \in \{0,n,n+1\}$.
\end{theorem}

We have similar results for the truncated algebras, which are a bit more interesting; see Theorems~\ref{thm:RightTruncationFormality}, \ref{thm:LeftTruncationFormality}, and \ref{thm:DoubleTruncationFormality}. In the cases where $\B(n,k,\Sc)$ or its truncations are not formal, we give examples of higher $\A_{\infty}$ actions on their homology which must be nonzero, but we do not attempt to characterize all such actions. An explicit description of these actions might be useful for the further algebraic study of $\B(n,k,\Sc)$.

\subsection*{Motivation and further directions}

This paper is the first in a series of at least three, including \cite{MMW2} and \cite{MMW3}. The Kauffman-states functor is motivated by holomorphic curve counting as in bordered Floer homology, and such counting can be used to prove its relationship with $HFK$ as Ozsv{\'a}th--Szab{\'o} will show in \cite{OSzHolo}. The Heegaard diagrams in which one counts these curves can be viewed in terms of a natural topological framework generalizing Zarev's bordered sutured Floer homology \cite{BSFH}. No attempt has been made to define bordered Floer homology analytically in this level of generality; this is expected to be quite difficult, with the Kauffman-states functor and Lipshitz--Ozsv{\'a}th--Thurston's forthcoming ``bordered $HF^-$'' theory for $3$-manifolds \cite{LOTMinus} with torus boundary arising as special cases.

Unlike in \cite{LOTMinus}, $\A_{\infty}$ deformations are not required for the algebras in \cite{OSzNew}, suggesting that the generalized bordered sutured theory hypothesized above should assign a reasonable generalization of the usual bordered strands algebras to the topological data motivating the algebras $\B(n,k,\Sc)$. However, this reasonable generalization gives algebras that are larger than $\B(n,k,\Sc)$, with nontrivial differential even when $\Sc = \varnothing$ (the meaning of $\Sc$ will be discussed below in Section~\ref{sec:GeneralOrientations}). 

In \cite{MMW2}, we construct the reasonably-generalized strands algebra mentioned above, in the case relevant for the Kauffman-states functor (this algebra is a special case of more general strands algebras that will be constructed by Rapha{\"e}l Rouquier and the first named author in \cite{ManionRouquier}). We call this algebra $\sac nk\Sc$, and we prove some useful properties about it. We define gradings on $\sac nk\Sc$ combinatorially and show how these gradings arise naturally from the group-valued gradings typical of the general bordered Floer setup.

We then exhibit a quasi-isomorphism $\Phi$ from $\B(n,k,\Sc)$ to $\sac nk\Sc$, giving evidence that the algebraic structure of the Kauffman-states functor may indeed be part of a generalization of bordered sutured Floer homology as mentioned above. Theorems~\ref{thm:IntroQuiverDescription} and \ref{thm:IntroHomology} in this paper are key elements of the construction; we use the generators-and-relations description of $\B(n,k,\Sc)$ to define the homomorphism $\Phi$ out of it, and we use Theorem~\ref{thm:IntroHomology} to help show $\Phi$ is a quasi-isomorphism. We also define symmetries on $\sac nk\Sc$ analogous to Ozsv{\'a}th--Szab{\'o}'s and show that $\Phi$ preserves them.

In \cite{MMW3}, which is in preparation, we will discuss bimodules in the context of \cite{MMW2}. In the language of bordered Floer homology, we will construct $DA$ bimodules for positive and negative crossings such that, after applying induction and restriction functors appropriately, we have a homotopy equivalence between the $DA$ bimodules we construct and the ones constructed in \cite{OSzNew}.

In this paper as well as \cite{MMW2, MMW3}, we work with the algebras of \cite{OSzNew}. Ozsv{\'a}th--Szab{\'o} use algebras that are related, but different to varying degrees, in \cite{OSzNewer,OSzHolo,OSzPong}. It would be very interesting to find strands algebra interpretations for any of these relatives of $\B(n,k,\Sc)$; to us, it seems like the strands interpretation is most immediate for the original algebra $\B(n,k,\Sc)$.

Following \cite{OSzNew} as well as the general convention in bordered Floer homology, we will work over the field $\F_2$. There are serious analytic difficulties that arise in bordered Floer homology when working over $\Z$. While it is plausible that our algebraic results could be formulated over $\Z$, the $\F_2$ versions would still be more directly comparable to a generalized bordered sutured theory as discussed above, unless one could also formulate that theory over $\Z$.

\subsection*{Organization}

In Section \ref{sec:OSz} we give an alternative quiver description of Ozsv\'ath--Szab\'o's algebra $\B_0(n,k)$, and prove Theorem \ref{thm:IntroQuiverDescriptionB0}. We discuss the quotient $\B(n,k)$ of $\B_0(n,k)$ and the more general algebra $\B(n,k,\Sc)$ in Section~\ref{sec:OSzB}. Corollary~\ref{cor:OSzQuiverEquivDG} concludes the proof of Theorem~\ref{thm:IntroQuiverDescription}. 

When working with algebras (like $\B_0(n,k)$ and $\B(n,k,\Sc)$) that come with a distinguished collection of idempotents, we freely make use of the perspective of differential graded categories. For the reader's convenience, a review of the relevant category theory is included in Appendix~\ref{app:Algebra}. 
In Section~\ref{sec:OSzStructure}, we use Ozsv{\'a}th--Szab{\'o}'s notion of ``generating intervals'' to give a decomposition theorem (Corollary~\ref{cor:IBItoTensorProduct}) for Hom-spaces in the category associated to $\B(n,k,\Sc)$. In Section \ref{sec:HomologyAndFormality} we use this decomposition theorem to compute the homology of $\B(n,k,\Sc)$, proving Theorem \ref{thm:IntroHomology}. We also investigate formality in Section~\ref{sec:HomologyAndFormality}, proving Theorem~\ref{thm:IntroFormality} and its analogues for the truncated algebras.

\subsection*{Acknowledgments}

The authors would like to thank Francis Bonahon, Ko Honda, Aaron Lauda, Robert Lipshitz, Ciprian Manolescu, Peter Ozsv{\'a}th, Rapha{\"e}l Rouquier, and Zolt{\'a}n Szab{\'o} for many useful conversations. The first named author would especially like to thank Zolt{\'a}n Szab{\'o} for teaching him about the Kauffman-states functor.
\section{Quiver descriptions of Ozsv\'ath-Szab\'o's algebra \texorpdfstring{$\B_0$}{B0}}
\label{sec:OSz}

\subsection{Quiver algebras}\label{sec:QuiverAlgs}

\begin{definition}\label{def:Paths}
Let $\Gamma$ be a finite directed graph, allowed to have loops and multi-edges, and let $V$ and $E$ denote the sets of vertices and edges (or arrows) of $\Gamma$ respectively. A \emph{path} $\gamma$ in $\Gamma$ is given by a finite sequence of edges, written $\gamma = (\gamma_1, \ldots, \gamma_l)$ (when $l = 1$ we omit the parentheses), such that for all $i = 1, \ldots, l-1$, the ending vertex of $\gamma_i$ coincides with the starting vertex of $\gamma_{i+1}$.  The \emph{start} of a path $\gamma = (\gamma_1, \ldots, \gamma_l)$ is the starting vertex of $\gamma_1$, denoted by $v_1$. Likewise the \emph{end} of $\gamma$ is the ending vertex of $\gamma_l$, denoted by $v_{l+1}$. The number $l$ is called the \emph{length} of $\gamma$.
\end{definition}

For every vertex $v \in V$, there is a distinguished path $I_v$ from $v$ to $v$ of length $0$, given by the empty sequence of edges. 

Given a finite directed graph $\Gamma$, one can construct the path algebra over $\Gamma$ with coefficients in a commutative ring $\ring$ (in this paper, $\ring$ will always be either the two element field $\F_2$ or a polynomial ring over $\F_2$). The path $I_v$ induces a (distinguished) idempotent in this algebra. In order to remember that this algebra comes with a set of distinguished idempotents, we can use the category defined below. For a review of some definitions concerning algebras and categories (e.g. $\ring$-linear category), see Appendix \ref{app:Algebra}. Note that for us a $\ring$-algebra is a ring $\mc A$ equipped with a ring homomorphism $\ring \to \mc A$; see Remark \ref{rem:SAS}.

\begin{definition}\label{def:CatOfAQuiver}
Let $\Gamma$ be a finite directed graph as above and let $\ring$ be a commutative ring. Define $\ring\Gamma$ to be the $\ring$-linear category whose objects are vertices $v \in V$ of $\Gamma$ and such that for two vertices $v_1, v_2 \in V$, $\Hom_{\ring\Gamma}(v_2,v_1)$ is the free $\ring$-module formally spanned by all paths in $\Gamma$ from $v_1$ to $v_2$. Composition of morphisms in $\Gamma$ is given by concatenation of paths, extended linearly over $\ring$, and identity morphisms $\Ib_v$ for $v \in V$ are given by the ``empty'' paths $I_v$.
\end{definition}

\begin{remark}
The reversal of directions in Definition~\ref{def:CatOfAQuiver} is intentional; one wants compositions $fg$ in a category, thought of as ``$f$ after $g$,'' to agree with multiplications $ab$ in a path algebra, thought of as ``$b$ after $a$'' and determined by edges $v_1 \xrightarrow{a} v_2 \xrightarrow{b} v_3$ in $\Gamma$.
\end{remark}

As defined in Section~\ref{sec:AlgsAndCats}, we have an $\IdemRing$-algebra $\Alg_{\ring \Gamma}$, where $\IdemRing = \ring^V$. This algebra is called the \emph{path algebra} of $\Gamma$; we will denote it by $\Path(\Gamma)$. When the coefficient ring $\ring$ is not clear from the context, we will denote it by $\Path_{\ring}(\Gamma)$. The set $\set{\Ib_v \,\middle|\, v \in V}$ is a set of pairwise orthogonal idempotents in $\Path(\Gamma)$: for all $v$, $v' \in V$ we have
\[
\Ib_v \cdot \Ib_{v'} =
\begin{cases}
\Ib_v & \text{if $v = v'$} \\
0 & \text{otherwise.}
\end{cases}
\]
The identity element of $\Path(\Gamma)$ is $\sum_{v \in V} \Ib_v$.

The edges of $\Gamma$ give a natural set of multiplicative generators for $\Path(\Gamma)$; paths in $\Gamma$ as defined in Definition~\ref{def:Paths} give a basis for $\Path(\Gamma)$ as a free module over $\ring$.

More generally, as explained in detail in Appendix \ref{app:Algebra}, given a $\ring$-linear category $\mathcal C$ with (finite) object set $V$, one can form a corresponding $\IdemRing$-algebra $\Alg_{\mathcal C}$ by summing over the morphism spaces. The composition of the inclusion of constant functions $\ring \to \IdemRing$ with the ring homomorphism $\IdemRing \to \Alg_{\mc C}$ has image contained in the center of $\Alg_{\mc C}$. Vice versa, given an $\IdemRing$-algebra $\mc A$ such that the natural map $\ring \to \IdemRing \to \mc A$ has image in $Z(\mc A)$, one can form a $\ring$-linear category $\Cat_{\mc A}$ with object set $V$. Moreover, functors between $\ring$-linear categories that are the identity on objects correspond to $\IdemRing$-algebra homomorphisms.  

\begin{remark}
We will often refer to certain functors as being equivalent to algebra homomorphisms; we assume without further mention that all functors discussed in this context are the identity on objects. Also, when discussing algebras $\A$ over $\IdemRing = \ring^V$ for a finite set $V$, we will assume that the natural map $\ring \to \IdemRing \to \A$ has image in $Z(\A)$.
\end{remark}

\begin{definition}\label{def:AlgOfAQuiver}
Let $\Gamma$ be a directed graph with vertex set $V$. For $v,v' \in V$, let $\mc R_{v,v'}$ be a subset of $\Hom_{\ring\Gamma}(v',v)$. Let $\mc R$ be the union of $\mc R_{v,v'}$ over all $v,v' \in V$, viewed as a subset of $\Path(\Gamma)$; we will call $\mc R$ a set of \emph{relations}. Let $\IdemRing = \ring^V$ as usual. We define the quiver algebra with relations $\Quiv(\Gamma, \mc R)$ to be the $\IdemRing$-algebra
\[
\Quiv(\Gamma, \mc R) = \frac{\Path(\Gamma)}{\mc I_{\mc R}},
\]
where $\mc I_{\mc R}$ is the two-sided ideal generated by the relation set $\mc R$. We have a corresponding category $\Cat_{\Quiv(\Gamma, \mc R)}$ with object set $V$.
\end{definition}

The edges of $\Gamma$ still give a natural set of multiplicative generators for $\Quiv(\Gamma, \mc R)$. Paths in $\Gamma$ give a spanning set for $\Quiv(\Gamma, \mc R)$ over $\ring$; since we have imposed relations, the set of paths in $\Gamma$ might no longer be linearly independent. We will refer to elements of this spanning set as \emph{path-like elements} or \emph{additive generators} of $\Quiv(\Gamma, \mc R)$.

The following proposition is standard.

\begin{proposition}\label{prop:QuiverAlgUniversalProp}
Let $\Gamma$ be a finite directed graph with vertex set $V$ and edge set $E$. Let $\mc C$ be any $\ring$-linear category whose set of objects is $V$. Suppose that for each edge $\gamma \in E$ starting at $v_1$ and ending at $v_2$, we have a morphism $F(\gamma): v_2 \to v_1$ in $\mc C$. Then, there is a unique $\ring$-linear functor from $\ring\Gamma$ to $\mc C$ sending $v$ to $v$ for all $v \in V$ and sending $\gamma$ to $F(\gamma)$ for all $\gamma \in E$. 
\end{proposition}

This functor gives us a homomorphism of $\IdemRing$-algebras from $\Path(\Gamma)$ to $\Alg_{\mc C}$. If this homomorphism sends $\mc R \subset \Path(\Gamma)$ to zero, we get a homomorphism of $\IdemRing$-algebras from $\Quiv(\Gamma, \mc R)$ to $\Alg_{\mc C}$, or equivalently a functor from $\Cat_{\Quiv(\Gamma, \mc R)}$ to $\mc C$.

It will be convenient to have dg (i.e. differential graded) versions of the above constructions; the proofs of the below propositions are left to the reader. We discuss gradings first.
\begin{proposition}\label{prop:GradedQuiverAlg}
Let $\Gamma$, $V$, and $\mc R$ be as in Definition~\ref{def:AlgOfAQuiver}, and let $G$ be a group. Suppose that for each edge $\gamma$ of $\Gamma$, we are given an element $\deg(\gamma) \in G$. Extend $\deg$ multiplicatively to a map from paths in $\Gamma$ to $G$. Assume that each relation $r \in \mc R$ is a sum of paths with the same degree. Then $\deg$ gives $\Quiv(\Gamma,\mc R)$ the structure of a $G$-graded $\IdemRing$-algebra (see Appendix~\ref{sec:AppendixAlgebras} for a brief review of our grading conventions).
\end{proposition}

Next we discuss the differential.
\begin{proposition}\label{prop:DGAlgFromQuiver}
Let $\Gamma$, $V$, and $\mc R$ be as in Definition~\ref{def:AlgOfAQuiver}. Suppose we are given an element $\partial(\gamma)$ of $\Ib_{\x} \Path(\Gamma) \Ib_{\y}$ for each edge $\gamma$ of $\Gamma$ from a vertex $\x$ to another vertex $\y$.  Extend $\partial$ to $\Path(\Gamma)$ linearly and using the Leibniz rule. Assume that $\partial(\mc R) = 0$ and that $\partial^2(\gamma) = 0$ for each edge $\gamma$. Then $\partial$ gives $\Quiv(\Gamma, \mc R)$ the structure of a differential $\IdemRing$-algebra.

If $\Quiv(\Gamma, \mc R)$ has a $G$-grading for some group $G$, we have $\lambda \in Z(G)$, and $\partial$ is homogeneous of degree $\lambda^{-1}$, then $\partial$ gives $\Quiv(\Gamma, \mc R)$ the structure of a $(G,\lambda)$-graded dg $\IdemRing$-algebra. If the $G$-grading on $\Quiv(\Gamma, \mc R)$ comes from Proposition~\ref{prop:GradedQuiverAlg}, then $\de$ is homogeneous of degree $\lambda^{-1}$ as long as $\partial(\gamma)$ is homogeneous of degree $\lambda^{-1} \deg(\gamma)$ for each edge $\gamma$ of $\Gamma$.
\end{proposition}

There are graded, differential, and dg analogues of Proposition~\ref{prop:QuiverAlgUniversalProp}.

\subsection{The algebra \texorpdfstring{$\B_0(n,k)$}{B0(n,k)}}\label{sec:B0Section}
In this section we present two quiver descriptions for the algebra $\B_0(n,k)$ from \cite{OSzNew}.  The first of these will be a direct translation of the definition in that paper.  Proving that the second description is equivalent to the first will be the goal of Section \ref{sec:Eqv of descriptions}.

\subsubsection{I-states}\label{sec:Istates}

Throughout the paper, let $[0,n]$ denote the set $\{0,1,2,\ldots,n\}$. Ozsv{\'a}th--Szab{\'o} \cite[Section 3.1]{OSzNew} define an \emph{I-state} to be a subset $\x$ of $[0,n]$ with $|x| = k$. By convention, we write the elements of an I-state $\x$ in increasing order as $\x = \{x_1,\ldots,x_k\}$ with $x_1 < \cdots < x_k$. Let $V(n,k)$ be the set of I-states for a given $n$ and $k$. The algebras $\B_0(n,k)$ and $\B(n,k)$ can be viewed as algebras over $\mathbb{F}_2^{V(n,k)}$; we will call $\mathbb{F}_2^{V(n,k)}$ the \emph{ring of idempotents} and denote it by $\Ib(n,k)$.

\begin{definition}[{\cite[Section 3.1]{OSzNew}}]
\label{def:rel weight vec}
Let $\x, \y \in V(n,k)$. The \emph{minimal relative weight vector} $v(\x,\y)=(v_1(\x,\y),\ldots,v_n(\x,\y))\in\Z^n$ is defined by the formula
\begin{equation*}
v_i(\x,\y)= |\y\cap[i,n]| - |\x\cap[i,n]|.
\end{equation*}
The \emph{minimal relative grading vector} $|v|(\x,\y)$ is defined by $|v|_i(\x,\y) = |v_i(\x,\y)|$.
\end{definition}

It is straightforward to check that, for $\x, \y, \z \in V(n,k)$, we have
\begin{equation}
\label{eq:Additivity of v}
v(\x, \z) = v(\x, \y) + v(\y, \z).
\end{equation}

The functions $|v|_i$ are not additive in general, but they are subadditive.
\begin{proposition}[{\cite[Section 3.1]{OSzNew}}]
\label{prop:AbsViSubadditive}
For $\x,\y,\z \in V(n,k)$ and $1 \leq i \leq n$, $|v|_i(\y,\z) - |v|_i(\x,\z) + |v|_i(\x,\y)$ is a nonnegative even integer.
\end{proposition}

\subsubsection{The algebras}

We start with a simple rephrasing of Ozsv{\'a}th--Szab{\'o}'s definition of $\B_0(n,k)$ (see \cite[Section 3.1]{OSzNew}).
\begin{definition}\label{def:OSzStyleDef}
Let $K(n,k)$ denote the complete directed graph with vertex set $V = V(n,k)$ (this graph has a unique edge from $\x$ to $\y$ for every ordered pair $(\x,\y) \in V^2$). Denote the edge in $K(n,k)$ from $\x$ to $\y$ by $f_{\x,\y}$. Write $\Path(K(n,k))$ for $\Alg_{\F_2[U_1,\ldots,U_n] K(n,k)}$. Let $\mc R_K \subset \Path(K(n,k))$ denote the set of elements
\[
f_{\x,\y} f_{\y,\z} - \prod_{i=1}^n U_i^{(|v|_i(\y,\z) - |v|_i(\x,\z) + |v|_i(\x,\y))/2} f_{\x,\z}
\]
for all ordered triples $(\x,\y,\z) \in V^3$. Define
\[
\B_0(n,k) := \Quiv(K(n,k), \mc R_K),
\]
an algebra over $\F_2[U_1,\ldots,U_n]^{V(n,k)}$; we have an $\F_2[U_1,\ldots,U_n]$-linear category $\Cat_{\B_0(n,k)}$ with set of objects $V(n,k)$. One can check that for $\x,\y \in V(n,k)$, the $\F_2[U_1,\ldots,U_n]$-linear map $\phi^{\x,\y}$ from $\F_2[U_1,\ldots,U_n]$ to $\Ib_{\x} \B_0(n,k) \Ib_{\y}$ sending $1$ to $f_{\x,\y}$ is a bijection. Thus, $\Quiv(K(n,k), \mc R_K)$ agrees with Ozsv{\'a}th--Szab{\'o}'s definition of $\B_0(n,k)$.
\end{definition}

We now give an alternate definition of $\B_0(n,k)$; we will prove below that the algebra $\Quiv(\Gamma(n,k), \mc R)$ constructed in the following definition is isomorphic to $\B_0(n,k)$.

\begin{definition}\label{def:B_0 Quiver Algebra}
The directed graph $\Gamma(n,k)$ has vertex set $V = V(n,k)$. Its arrows are given as follows:
\begin{itemize}
\item For vertices $\x$ with $i-1 \in \x$ and $i \notin \x$, there is an arrow from $\x$ to $(\x \setminus \set{i-1}) \cup \set{i}$, said to have label $R_i$.
\item For vertices $\x$ with $i \in \x$ and $i-1 \notin \x$, there is an arrow from $\x$ to $(\x \setminus \set{i}) \cup \set{i-1}$, said to have label $L_i$.
\item For all vertices $\x$ and all $i$ between $1$ and $n$, there is an arrow from $\x$ to $\x$, said to have label $U_i$.
\end{itemize}
We write $\Path(\Gamma(n,k))$ for $\Alg_{\F_2 \Gamma(n,k)}$. To each basis element $\gamma$ of $\Path(\Gamma(n,k))$, we can associate a (non-commutative) monomial $\mu(\gamma)$ in the letters $R_i$, $L_i$, and $U_i$ for $1 \leq i \leq n$. Note that two paths $\gamma, \gamma'$ with the same monomial and starting at the same vertex are equal. We then extend $\mi$ $\F_2$-linearly to each element of $\Path(\Gamma(n,k))$. For every pair of vertices $\x$ and $\y$, we define $\mc R_{\x,\y} \subset \Hom(\y, \x) \subset \Path(\Gamma(n,k))$ to be the set of elements $\gamma \in \Hom(\y, \x)$ such that $\mu(\gamma)$ is equal to one of the following:
\begin{enumerate}
\item\label{rel:central} $R_i U_j - U_j R_i$, \,\, $L_i U_j - U_j L_i$, or $U_i U_j - U_j U_i$ (the ``$U$ central relations''),
\item $R_i L_i - U_i$ or $L_i R_i - U_i$ (the ``loop relations''),
\item $R_i R_j - R_j R_i$, \,\, $L_i L_j - L_j L_i$, or $R_i L_j - L_j R_i$ for $|i-j|>1$ (the ``distant commutation relations'').
\end{enumerate}
The minus signs could equally well be plus signs, since we are working over $\F_2$. We have an $\Ib(n,k)$-algebra $\Quiv(\Gamma(n,k), \mc R)$, where $\mc R = \bigcup_{\x, \y} \mc R_{\x, \y}$, and an $\F_2$-linear category $\Cat_{\Quiv(\Gamma(n,k), \mc R)}$ with object set $V(n,k)$. Using the edges of $\Gamma(n,k)$ with label $U_i$, we can give $\Quiv(\Gamma(n,k), \mc R)$ the structure of an algebra over $\F_2[U_1,\ldots,U_n]^{V(n,k)}$. The relations \eqref{rel:central} imply that the natural map $\F_2[U_1,\ldots,U_n] \to \F_2[U_1,\ldots,U_n]^{V(n,k)} \to \Quiv(\Gamma(n,k), \mc R)$ has image in the center of $\Quiv(\Gamma(n,k), \mc R)$. Equivalently, we can view $\Cat_{\Quiv(\Gamma(n,k),\mc R)}$ as an $\F_2[U_1,\ldots,U_n]$-linear category.
\end{definition}

\subsection{Graphical interpretations}\label{sec:GraphicalInterp}

\begin{figure}
\includegraphics[scale=0.5]{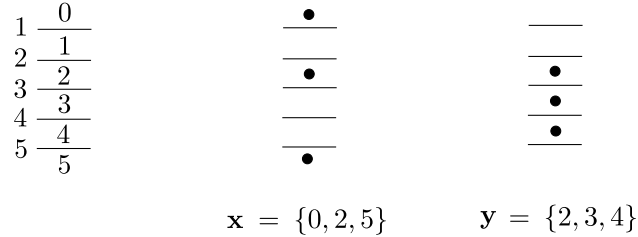}
\caption{Elements $\x$ and $\y$ of $V(5,3)$ viewed as dots occupying regions.}
\label{fig:IStates}
\end{figure}

In this section we will give graphical interpretations of the algebras $\Quiv(K(n,k), \mc R_K)$ and $\Quiv(\Gamma(n,k), \mc R)$. In both interpretations we follow \cite{OSzNew} and interpret an I-state $\x$ as a choice of ``occupied" regions between $n$ lines, illustrated with dots as in Figure~\ref{fig:IStates}.

\begin{remark}\label{rem:NinetyDegRot}
Our graphical conventions can be obtained from those used in \cite{OSzNew,OSzNewer} by $90^{\circ}$ rotation clockwise. We perform this rotation to match Lipshitz--Ozsv{\'a}th--Thurston and Zarev's conventions for strands pictures in bordered Floer homology; see \cite{MMW2} where we construct a quasi-isomorphism from $\B(n,k,\Sc)$ to a generalized strands algebra.
\end{remark}

\subsubsection{Graphical interpretation of $\Quiv(K(n,k), \mc R_K)$}

\begin{figure}
\includegraphics[scale=0.5]{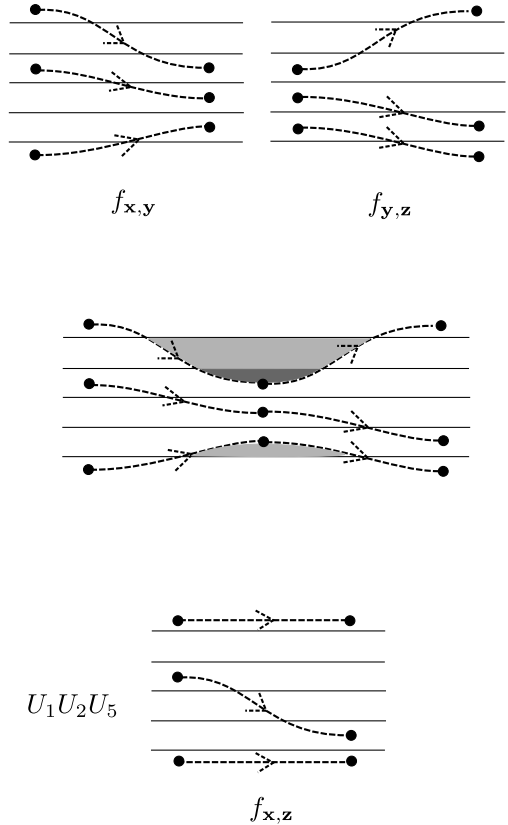}
\caption{Graphical interpretation of $\Quiv(K(n,k), \mc R_K)$, for $n = 5$ and $k = 3$. We have $\x = \{0,2,5\}$, $\y = \{2,3,4\}$, and $\z = \{0,4,5\}$.}
\label{fig:KInterpretation}
\end{figure}

We start with an interpretation of the algebra $\Quiv(K(n,k), \mc R_K)$. The generator $f_{\x,\y}$ of $K(n,k)$ is interpreted as a motion ``all at once'' of the dots comprising $\x$ to the dots comprising $\y$; see the top line of Figure~\ref{fig:KInterpretation}. Multiplying $f_{\x,\y}$ with $f_{\y,\z}$ always gives $f_{\x,\z}$ times a monomial $p$ in the $U_i$ variables which can be described graphically as follows. Draw the picture for $f_{\x,\y}$ on the left of the picture for $f_{\y,\z}$. The power of $U_i$ in $p$ equals the number of distinct bigons with one edge on line $i$ and the other edge on a motion of a dot in the concatenated picture. See Figure~\ref{fig:KInterpretation} for an illustration.

\subsubsection{Graphical interpretation of $\Quiv(\Gamma(n,k), \mc R)$}

\begin{figure}
\includegraphics[scale=0.5]{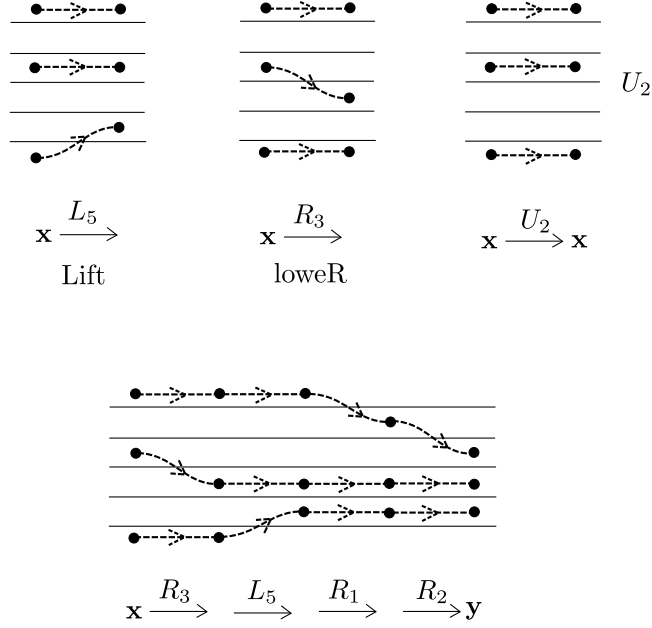}
\caption{Graphical interpretation of $\Quiv(\Gamma(n,k), \mc R)$, for $n = 5$ and $k = 3$.}
\label{fig:GammaInterpretation}
\end{figure}

Now we give a graphical interpretation of the algebra $\Quiv(\Gamma(n,k), \mc R)$; see Figure~\ref{fig:GammaInterpretation}. A path in $\Gamma(n,k)$ from $\x$ to $\y$ is interpreted as a motion ``one dot-step at a time'' of the dots comprising $\x$ to the dots comprising $\y$. An edge labeled $R_i$ moves a dot downwards one step; an edge labeled $L_i$ moves a dot upwards one step (we suggest the mnemonics ``Lift'' and ``loweR''). An edge labeled $U_i$ does not move the dots at all, but we record that a $U_i$ edge has been traversed. The basic idempotent $\Ib_{\x}$ can be interpreted as a stationary motion of the dots from $\x$ to itself. See also \cite[Remark 3.1]{OSzNew}.

\begin{figure}
\includegraphics[scale=0.5]{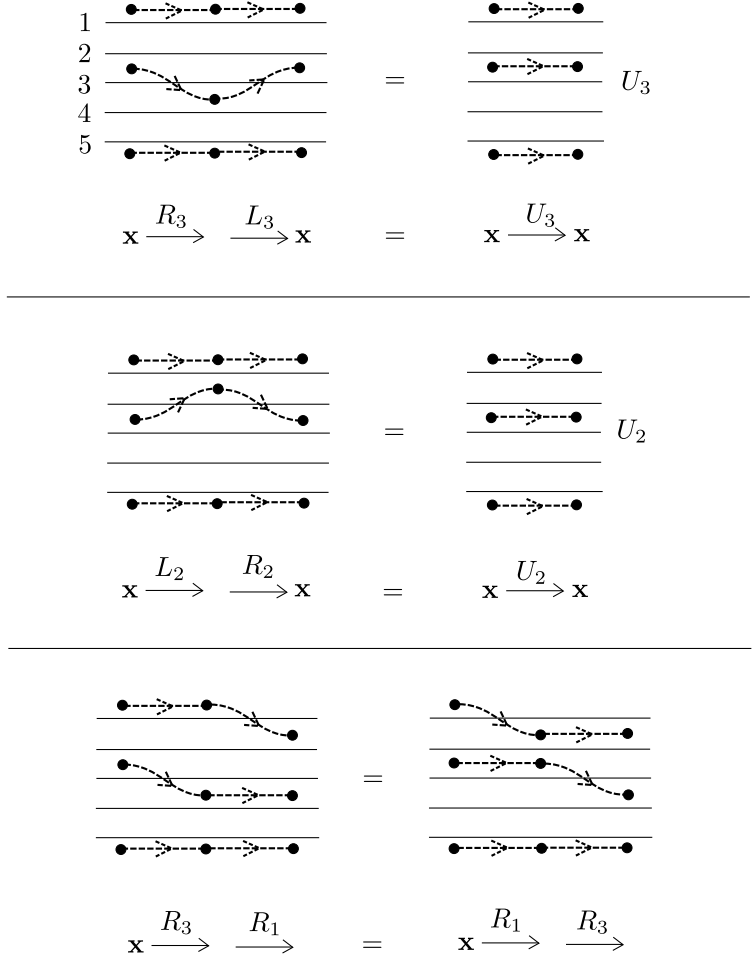}
\caption{Relations in $\mc R$, viewed as motions as in Figure~\ref{fig:GammaInterpretation}.}
\label{fig:IdemsAndMotionsPart2}
\end{figure}

Note that during the motion of dots represented by any path $\gamma$ in $\Gamma(n,k)$, there are never two dots in the same region between lines at any moment in the motion. The relations $\mc R \subset \Path(\Gamma)$ can be described in terms of motions as follows (see Figure~\ref{fig:IdemsAndMotionsPart2}):
\begin{enumerate}
\item Each $U_i$ loop commutes with all other moves (this is the meaning of the $U$ central relations),
\item A dot passing back and forth through the $i^\text{th}$ line is equivalent to a $U_i$ loop (this is the meaning of the loop relations),
\item If two dots can each be moved by one slot independently of each other, then the order in which they are moved does not matter and thus they can be viewed as moving at the same time (this is the meaning of the distant commutation relations).
\end{enumerate}
The next lemma says that the relative weight vector $v(\x,\y)$ counts (with sign) the number of dots passing through each line in any path $\gamma$ from $\x$ to $\y$.
\begin{lemma}
\label{lem:vi=Ri-Li}
Let $\x,\y \in V(n,k)$. Given any path $\gamma$ in $\Gamma(n,k)$ from $\x$ to $\y$, we have
\[
v_i(\x,\y) = \rho_i(\gamma) - \lambda_i(\gamma)
\]
where $\rho_i(\gamma)$ (respectively $\lambda_i(\gamma)$) counts the number of edges labeled $R_i$ (respectively $L_i$) in the path $\gamma$.
\end{lemma}
\begin{proof}
By additivity of the right hand side of the formula under concatenation of paths, and by equation \eqref{eq:Additivity of v}, we can assume that the path $\gamma$ consists of a single edge. The formula is true when this edge has label $R_j$, $L_j$, or $U_j$ (for both $j=i$ and $j\neq i$), as one can easily check from Definition \ref{def:B_0 Quiver Algebra}.
\end{proof}

For a path $\gamma$ in $\Gamma(n,k)$, represented visually by a motion of dots occurring one dot at a time, let $F(\gamma)$ denote the same motion with all dots moving simultaneously (we will define $F$ rigorously below). If $\x \xrightarrow{\gamma_1} \y \xrightarrow{\gamma_2} \z$ are minimal-length paths in $\Gamma(n,k)$, then Lemma~\ref{lem:vi=Ri-Li} implies that the quantity $(|v|_i(\y,\z) - |v|_i(\x,\z) + |v|_i(\x,\y))/2$ counts extraneous pairs $R_i, L_i$ in the concatenation $\gamma_1 \gamma_2$ (which the relations declare should be equivalent to $U_i$ loops). This quantity is also the exponent of $U_i$ appearing in the relations of Definition~\ref{def:OSzStyleDef} for $\B_0(n,k)$ when one computes the product $F(\gamma_1) F(\gamma_2)$. 

Thus it is visually plausible that the ``forget-the-ordering'' map $F$ should be an algebra homomorphism from $\Quiv(\Gamma(n,k), \mc R_K)$ to $\B_0(n,k)$ (we will prove this fact in Proposition~\ref{prop:FRespectsRelations}). To prove that $F$ is a surjection, we just need to exhibit a path from $\x$ to $\y$ in $\Gamma(n,k)$, for all $\x,\y \in V(n,k)$, whose image under $F$ is the generator $f_{\x,\y}$ of $\B_0(n,k)$; the existence of such a path is visually clear. 

We will prove that $F$ is injective by constructing an inverse function $G$ to $F$. For each pair $\x,\y$ of elements of $V(n,k)$, we will need to pick an explicit path $\gamma_{\x,\y}$ in the preimage $F^{-1}(f_{\x,\y})$, and we will need to show that the function $G(f_{\x,\y}) := \gamma_{\x,\y}$ respects the relations of Definition~\ref{def:OSzStyleDef}. These will be the main technical tasks required to prove that $\B_0(n,k)$ and $\Quiv(\Gamma(n,k),\mc R)$ describe the same algebra, which is the goal of Section \ref{sec:Eqv of descriptions}.

\subsection{Equivalence of descriptions}\label{sec:Eqv of descriptions}

\subsubsection{An intermediate description}

We want to show that $\B_0(n,k)$ and $\Quiv(\Gamma(n,k), \mc R)$ are isomorphic as $\F_2[U_1,\ldots,U_n]^{V(n,k)}$-algebras. To do so, it is convenient to introduce a third description of the same algebra.

\begin{definition}\label{def:IntermediateQuiver}
Let $\Gamma_U(n,k)$ be the directed graph $\Gamma(n,k)$ with all edges labeled $U_i$ removed. Let $\mc R_U \subset \Path_{\F_2[U_1, \ldots, U_n]}(\Gamma_U(n,k))$ be defined analogously to $\mc R \subset \Path_{\F_2}(\Gamma(n,k))$ from Definition~\ref{def:B_0 Quiver Algebra}, without elements of type \eqref{rel:central}, and interpreting any instance of $U_i$ in a loop relation as a coefficient in $\mathbb{F}_2[U_1,\ldots,U_n]$ rather than a label for an edge.
\end{definition}

Recall the discussion after Definition \ref{def:B_0 Quiver Algebra} about viewing $\Cat_{\Quiv(\Gamma(n,k), \mc R)}$ as being linear over $\F_2[U_1,\ldots,U_n]$.  The following lemma essentially says that this viewpoint is equivalent to building the category using $\Quiv(\Gamma_U(n,k), \mc R_U)$ rather than $\Quiv(\Gamma(n,k), \mc R)$.

\begin{lemma}\label{lem:IntermediateQuiverEqv}
We have an isomorphism of $\F_2[U_1,\ldots,U_n]$-linear categories
\[
\Cat_{\Quiv(\Gamma(n,k), \mc R)} \xrightarrow{\cong} \Cat_{\Quiv(\Gamma_U(n,k), \mc R_U)}.
\]
Equivalently, we have an isomorphism of $\F_2[U_1,\ldots,U_n]^{V(n,k)}$-algebras 
\[
\Quiv(\Gamma(n,k), \mc R) \xrightarrow{\cong} \Quiv(\Gamma_U(n,k), \mc R_U).
\]
\end{lemma}
\begin{proof}
By Proposition~\ref{prop:QuiverAlgUniversalProp}, we have an $\F_2$-linear functor 
\[\zeta: \F_2 \Gamma(n,k) \to \Cat_{\Quiv(\Gamma_U(n,k), \mc R_U)}\]
defined as the identity on objects and by sending any edge (from some vertex $v$ to itself) labeled by $U_i$ to the corresponding empty path $I_v$ in $\Quiv(\Gamma_U(n,k), \mc R_U)$ with coefficient $U_i$, which gives the morphism $v\stackrel{U_i}{\longrightarrow} v$ in $\Cat_{\Quiv(\Gamma_U(n,k), \mc R_U)}$. The functor $\zeta$ is defined to be the identity on all other edges. Incorporating the relations, $\zeta$ induces an $\F_2$-linear functor 
\[\tilde{\zeta}: \Cat_{\Quiv(\Gamma(n,k), \mc R)} \to \Cat_{\Quiv(\Gamma_U(n,k), \mc R_U)};\] one can check that $\tilde{\zeta}$ is in fact $\F_2[U_1,\ldots,U_n]$-linear when viewing $\Cat_{\Quiv(\Gamma(n,k), \mc R)}$ as an $\F_2[U_1,\ldots,U_n]$-linear category as discussed after Definition~\ref{def:B_0 Quiver Algebra}.

Similarly, we can build an $\F_2$-linear functor
\[\xi: \F_2[U_1,\ldots,U_n] \Gamma_U(n,k) \to \Cat_{\Quiv(\Gamma(n,k), \mc R)}\]
by sending all objects $v$ to themselves, and sending any morphism in $\F_2[U_1,\ldots,U_n] \Gamma_U(n,k)$ (which is an $\F_2[U_1,\ldots,U_n]$-linear combination of paths) to the corresponding morphism in $\Cat_{\Quiv(\Gamma(n,k), \mc R)}$, except that coefficients $U_i$ are reinterpreted as concatenation with an edge labeled $U_i$.  Because the $U_i$-edges exist at every vertex and commute with all others in $\Quiv(\Gamma(n,k), \mc R)$, this assignment is well-defined independently of the choice of ordering.  Our functor $\xi$ induces an $\F_2$-linear functor 
\[\tilde{\xi}: \Cat_{\Quiv(\Gamma_U(n,k), \mc R_U)} \to \Cat_{\Quiv(\Gamma(n,k), \mc R)}\]
and again one can check that $\tilde{\xi}$ is $\F_2[U_1,\ldots,U_n]$-linear when viewing $\Cat_{\Quiv(\Gamma(n,k), \mc R)}$ as an $\F_2[U_1,\ldots,U_n]$-linear category.

By construction, $\tilde{\zeta}$ and $\tilde{\xi}$ are inverse isomorphisms of $\F_2[U_1,\ldots,U_n]$-linear categories between $\Cat_{\Quiv(\Gamma(n,k), \mc R)}$ and $\Cat_{\Quiv(\Gamma_U(n,k), \mc R_U)}$, proving the claim.
\end{proof}

Thus, it suffices to show that $\B_0(n,k)$ and $\Quiv(\Gamma_U(n,k), \mc R_U)$ are isomorphic as algebras over $\F_2[U_1,\ldots,U_n]^{V(n,k)}$.

\subsubsection{Constructing the forward homomorphism}\label{sec:DefinitionOfF}

Let $\x,\y \in V(n,k)$. For any edge $e$ of $\Gamma_U(n,k)$ from $\x$ to $\y$, define $F(e) = f_{\x,\y} \in \Hom_{\Cat_{\B_0(n,k)}}(\y,\x)$. By Proposition~\ref{prop:QuiverAlgUniversalProp}, $F$ extends uniquely to an $\F_2[U_1, \ldots, U_n]$-linear functor from $\F_2[U_1,\ldots,U_n] \Gamma_U(n,k)$ to $\Cat_{\B_0(n,k)}$, or equivalently a homomorphism of $\F_2[U_1,\ldots,U_n]^{V(n,k)}$-algebras from $\Path(\Gamma_U(n,k))$ to $\B_0(n,k)$.

\begin{proposition}\label{prop:FRespectsRelations}
The homomorphism of $\F_2[U_1,\ldots,U_n]^{V(n,k)}$-algebras 
\[
F: \Path(\Gamma_U(n,k)) \to \B_0(n,k)
\]
sends each element of $\mc R_U \subset \Path(\Gamma_U(n,k))$ to zero. 
\end{proposition}

\begin{proof}
First, consider an element of $\mc R_U$ of the form $(\gamma_1, \gamma_2) - U_i I_{\x}$, where $\gamma_1$ is an edge from $\x$ to $\y$ with label $R_i$ and $\gamma_2$ is an edge from $\y$ to $\x$ with label $L_i$.  By Lemma \ref{lem:vi=Ri-Li}, we have $|v|_i(\x,\y) = |v|_i(\y,\x) = 1$ and $|v|_j(\x,\y) = |v|_j(\y,\x) = 0$ for $j \neq i$.  Since $|v_j|(\x,\x)=0$ for all $j$,  we have
\[
F(\gamma_1) F(\gamma_2) = f_{\x,\y} f_{\y,\x} = U_i f_{\x,\x} = U_i F(I_{\x}).
\]
The argument for elements of $\mc R_U$ of the form $(\gamma_1, \gamma_2) - U_i I_{\x}$ where $\gamma_1$ has label $L_i$ and $\gamma_2$ has label $R_i$ is similar.

Next, consider an element of $\mc R_U$ of the form $(\gamma_1, \gamma_2) - (\gamma_3, \gamma_4)$ where $\gamma_1$ is an edge from $\x$ to $\y$ with label $R_i$, $\gamma_2$ is an edge from $\y$ to $\z$ with label $R_j$, $\gamma_3$ is an edge from $\x$ to $\y'$ with label $R_j$, $\gamma_4$ is an edge from $\y'$ to $\z$ with label $R_i$, and $|i-j| > 1$.  Again, by Lemma \ref{lem:vi=Ri-Li} we have 
\[
|v|_i(\x,\y) = |v|_j(\y,\z) = |v|_i(\x,\z) = |v|_j(\x,\z) = 1.
\]
By the same lemma we also have $|v|_j(\x,\y) = |v|_i(\y,\z) = 0$ as well as $|v|_l(\x,\y)=|v|_l(\x,\z)=|v|_l(\y,\z)=0$ for all $l\neq i,j$.  It follows that
\[
F(\gamma_1)F(\gamma_2) = f_{\x,\y} f_{\y,\z} = f_{\x,\z}.
\]
A parallel argument shows that $F(\gamma_3)F(\gamma_4) = f_{\x,\z}$, so we have $F(\gamma_1)F(\gamma_2) = F(\gamma_3)F(\gamma_4)$. The rest of the cases are similar to this one.
\end{proof}

As a result we have a homomorphism of $\F_2[U_1,\ldots,U_n]^{V(n,k)}$-algebras 
\[
F: \Quiv(\Gamma_U(n,k), \mc R_U) \to \B_0(n,k).
\]

\subsubsection{Constructing the inverse homomorphism}\label{sec:DefinitionOfGPart1}

In this section, we will define a homomorphism of $\F_2[U_1,\ldots,U_n]^{V(n,k)}$-algebras $G: \B_0(n,k) \to \Quiv(\Gamma_U(n,k), \mc R_U)$, or equivalently an $\F_2[U_1,\ldots,U_n]$-linear functor from $\Cat_{\B_0(n,k)}$ to $\Cat_{\Quiv(\Gamma_U(n,k), \mc R_U)}$. 

We start by defining an $\F_2[U_1,\ldots,U_n]$-linear functor 
\[
G: \F_2[U_1,\ldots,U_n] K(n,k) \to \Cat_{\Quiv(\Gamma_U(n,k), \mc R_U)}.
\]
By Proposition~\ref{prop:QuiverAlgUniversalProp}, it suffices to choose, for all pairs $(\x,\y) \in V(n,k)^2$, a path $\gamma_{\x,\y}$ in $\Gamma_U(n,k)$ from $\x$ to $\y$. We will further require that $F(\gamma_{\x,\y}) = f_{\x,\y}$. We will choose $\gamma_{\x,\y}$ recursively, but we need a few results first. The proof of the following lemma is left to the reader.

\begin{lemma}\label{lem:UniquePathForRightMovingDot}
For $\x \in V(n,k)$, suppose $l \in \x$ and $l < m \leq n$ such that $\x \cap [l,m] = \{l\}$. Let $\x' = (\x \setminus \{l\}) \cup \{m\}$. There exists a path $\gamma$ from $\x$ to $\x'$ in $\Gamma_U(n,k)$ whose edges are labeled $R_{l+1},\ldots,R_m$ in order. The path $\gamma$ is the unique path from $\x$ to $\x'$ with this property.
\end{lemma}

\begin{lemma}\label{lem:RSegmentFactoring}
Under the assumptions of Lemma~\ref{lem:UniquePathForRightMovingDot}, let $\x = \x^1, \x^2, \ldots, \x^p = \x'$ be the sequence of vertices traversed in the path $\gamma$ of that lemma. We have $f_{\x,\x'} = f_{\x^1,\x^2} \cdots f_{\x^{p-1},\x^p}$.
\end{lemma}

\begin{proof}
Induct on $p \geq 2$; the case $p=2$ is tautological. Assume that 
\[
f_{\x^1,\x^{p-1}} = f_{\x^1,\x^2}, \cdots, f_{\x^{p-2},\x^{p-1}}.
\]
It suffices to show that $f_{\x^1,\x^p} = f_{\x^1,\x^{p-1}} f_{\x^{p-1},\x^p}$. For $l+1 \leq i \leq m-1$, we have $|v|_i(\x^1,\x^p) = 1$, $|v|_i(\x^1,\x^{p-1}) = 1$, and $|v|_i(\x^{p-1},\x^p) = 0$ by Lemma~\ref{lem:vi=Ri-Li}. For $i = m$, we have $|v|_i(\x^1,\x^p) = 1$, $|v|_i(\x^1,\x^{p-1}) = 0$, and $|v|_i(\x^{p-1},\x^p) = 1$. For $i \notin [l+1,m]$, we have $|v|_i(\x^1,\x^p) = |v|_i(\x^1,\x^{p-1}) = |v|_i(\x^{p-1},\x^p) = 0$. Thus, the lemma follows from the relations defining multiplication in $\B_0(n,k)$.
\end{proof}

\begin{corollary}\label{cor:WhereFSendsRSegment}
Under the assumptions of Lemma~\ref{lem:UniquePathForRightMovingDot}, we have $F(\gamma) = f_{\x,\x'}$, where $\gamma$ is the path constructed in that lemma and $F$ was defined in Section~\ref{sec:DefinitionOfF}.
\end{corollary}

\begin{proof}
Write $\gamma = (\gamma_{l+1},\ldots,\gamma_m)$ where $\gamma_i$ has label $R_i$. We have $F(\gamma) = F(\gamma_{l+1}) \cdots F(\gamma_m) = f_{\x^1,\x^2} \cdots f_{\x^{p-1}, \x^p}$, which equals $f_{\x,\x'}$ by Lemma~\ref{lem:RSegmentFactoring}.
\end{proof}

\begin{lemma}\label{lem:RSegmentImpliesMissingDots}
Let $\x,\y \in V(n,k)$ with $x_a < y_a$ for some $a\in[1,k]$; let $a$ be the maximal such index. We have $\x \cap [x_a,y_a] = \{x_a\}$.
\end{lemma}

\begin{proof}
Suppose $x_b \in \x$ satisfies $x_a < x_b \leq y_a$. We have $a < b$, so $y_a < y_b$. Thus, $x_b \leq y_a < y_b$, contradicting the maximality of $a$.
\end{proof}

\begin{corollary}\label{cor:RSegment_fFactors}
Under the assumptions of Lemma~\ref{lem:RSegmentImpliesMissingDots}, let $\x' = (\x \setminus \{x_a\}) \cup \{y_a\}$. We have $f_{\x,\y} = f_{\x,\x'} f_{\x',\y}$.
\end{corollary}

\begin{proof}
For $1 \leq i \leq n$, we want to show that $|v|_i(\x,\y) = |v|_i(\x,\x') + |v|_i(\x',\y)$.  First, note that if $i\notin[x_a+1,y_a]$, then $|\x'\cap[i,n]|=|\x\cap[i,n]|$ so that $v_i(\x,\x')=0$ and
\[|v_i|(\x,\y) = \Big| |\y\cap[i,n]| - |\x\cap[i,n]|\Big| = \Big| |\y\cap[i,n]| - |\x'\cap[i,n]|\Big| = |v_i|(\x',\y)+0\]
as desired.  Meanwhile, if $i\in[x_a+1,y_a]$, we must have $v_i(\x,\x')=1$.  Furthermore, Lemma \ref{lem:RSegmentImpliesMissingDots} ensures that $x_a'=y_a$ is the smallest element of $\x'$ that is greater than or equal to $i$, which in turn forces $v_i(\x',\y)\geq0$.  Thus the additivity of $v_i$ shows that $v_i(\x,\y)\geq 0$ as well, and then the desired equality is immediate from the additivity of $v_i$.
\end{proof}

There are also ``upward-moving'' versions of the above results involving edges labeled $L_i$ rather than $R_i$.

\begin{lemma}\label{lem:UniquePathForLeftMovingDot}
For $\x \in V(n,k)$, suppose $l \in \x$ and $0 \leq m < l$ such that $\x \cap [m,l] = \{l\}$. Let $\x' = (\x \setminus \{l\}) \cup \{m\}$. There exists a path $\gamma$ from $\x$ to $\x'$ in $\Gamma_U(n,k)$ whose edges are labeled $L_l, \ldots, L_{m+1}$ in order. The path $\gamma$ is the unique path from $\x$ to $\x'$ with this property.
\end{lemma}

\begin{lemma}\label{lem:LSegmentFactoring}
Under the assumptions of Lemma~\ref{lem:UniquePathForLeftMovingDot}, let $\x = \x^1, \x^2, \ldots, \x^p = \x'$ be the sequence of vertices traversed in the path $\gamma$ of that lemma. We have $f_{\x,\x'} = f_{\x^1,\x^2} \cdots f_{\x^{p-1},\x^p}$.
\end{lemma}

\begin{corollary}\label{cor:WhereFSendsLSegment}
Under the assumptions of Lemma~\ref{lem:UniquePathForLeftMovingDot}, we have $F(\gamma) = f_{\x,\x'}$, where $\gamma$ is the path constructed in that lemma.
\end{corollary}

\begin{lemma}\label{lem:LSegmentImpliesMissingDots}
Let $\x,\y \in V(n,k)$ with $x_a > y_a$ for some $a\in[1,k]$; let $a$ be the minimal such index. We have $\x \cap [y_a,x_a] = \{x_a\}$.
\end{lemma}

\begin{corollary}\label{cor:LSegment_fFactors}
Under the assumptions of Lemma~\ref{lem:LSegmentImpliesMissingDots}, let $\x' = (\x \setminus \{x_a\}) \cup \{y_a\}$. We have $f_{\x,\y} = f_{\x,\x'} f_{\x',\y}$.
\end{corollary}

Now we recursively define a path $\gamma_{\x,\y}$ in $\Gamma_U(n,k)$ from $\x$ to $\y$.
\begin{definition}\label{def:Recursive}
Our recursion scheme will involve the quantity $k - |\x \cap \y|$. If this quantity is zero, then $\x = \y$. Define $\gamma_{\x,\x} = I_{\x}$, the identity path at $\x$. Since $F$ sends identity morphisms to identity morphisms, we have $F(\gamma_{\x,\x}) = f_{\x,\x}$.

Now suppose we have $\x$ and $\y$, and that we have constructed $\gamma_{\x',\y'}$ such that $F(\gamma_{\x',\y'}) = f_{\x',\y'}$ for all $(\x',\y')$ with $k - |\x' \cap \y'| < k - |\x \cap \y|$. To define $\gamma_{\x,\y}$, first suppose that $x_a < y_a$ for some $a\in[1,k]$. Let $a$ be the maximal such index. By Lemma~\ref{lem:RSegmentImpliesMissingDots}, we have $\x \cap [x_a,y_a] = \{x_a\}$. Let $\x' = (\x \setminus \{x_a\}) \cup \{y_a\}$. By Lemma~\ref{lem:UniquePathForRightMovingDot}, there exists a unique path $\gamma$ in $\Gamma_U(n,k)$ from $\x$ to $\x'$ with edges labeled $R_{x_a + 1}, \ldots, R_{y_a}$ in order, and we have $F(\gamma) = f_{\x,\x'}$ by Corollary~\ref{cor:WhereFSendsRSegment}. Since $k - |\x' \cap \y| < k - |\x \cap \y|$, we have already constructed a path $\gamma_{\x',\y}$ from $\x'$ to $\y$ with $F(\gamma_{\x',\y}) = f_{\x',\y}$. Define
\[
\gamma_{\x,\y} := \gamma \cdot \gamma_{\x',\y}.
\]
We have $F(\gamma_{\x,\y}) = F(\gamma) F(\gamma_{\x',\y}) = f_{\x,\x'} f_{\x',\y}$, which equals $f_{\x,\y}$ by Corollary~\ref{cor:RSegment_fFactors}.

If $x_a \geq y_a$ for all $a$ but $\x \neq \y$, let $a$ be the minimal index such that $x_a > y_a$. By Lemma~\ref{lem:LSegmentImpliesMissingDots}, we have $\x \cap [y_a,x_a] = \{x_a\}$. By Lemma~\ref{lem:UniquePathForLeftMovingDot}, there exists a unique path $\gamma$ from $\x$ to $\x'$ with edges labeled $L_{x_a}, \ldots, L_{y_a + 1}$ in order, and we have $F(\gamma) = f_{\x,\x'}$ by Corollary~\ref{cor:WhereFSendsLSegment}. As before, we have already constructed a path $\gamma_{\x',\y}$ from $\x'$ to $\y$ with $F(\gamma_{\x',\y}) = f_{\x',\y}$. Define
\[
\gamma_{\x,\y} := \gamma \cdot \gamma_{\x',\y}.
\]
We have $F(\gamma_{\x,\y}) = F(\gamma) F(\gamma_{\x',\y}) = f_{\x,\x'} f_{\x',\y}$ which equals $f_{\x,\y}$ by Corollary~\ref{cor:LSegment_fFactors}.
\end{definition}

\begin{figure}
\includegraphics[scale=0.5]{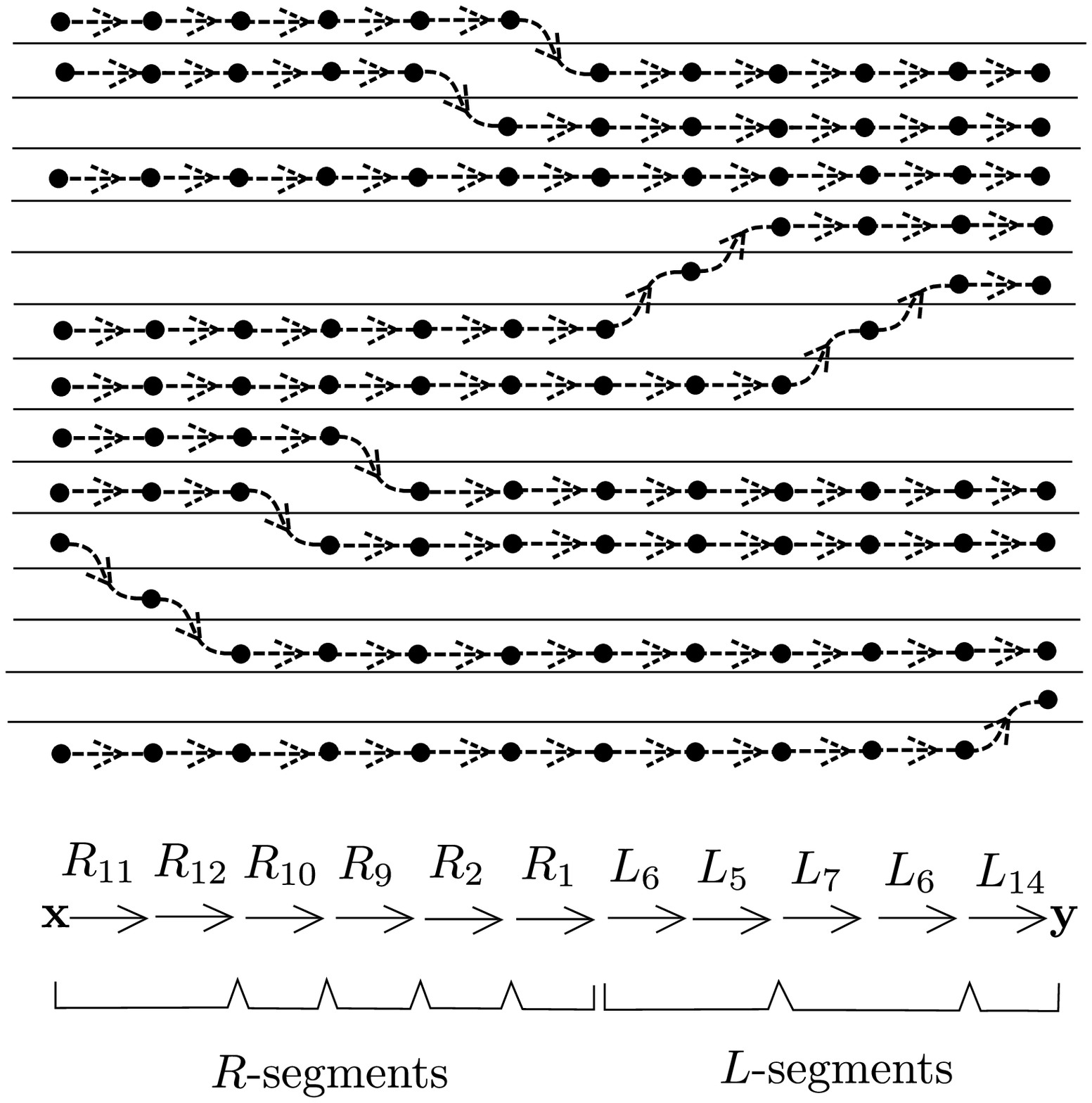}
\caption{The path $\gamma_{\x,\y}$ in Example~\ref{ex:GammaXYExample}.}
\label{fig:RecursiveDefExample}
\end{figure}

\begin{example}\label{ex:GammaXYExample}
Let $n = 14$ and $k = 9$. For the elements $\x = \{0,1,3,6,7,8,9,10,14\}$ and $\y = \{1,2,3,4,5,9,10,12,13\}$ of $V(n,k)$, the path $\gamma_{\x,\y}$ is shown in Figure~\ref{fig:RecursiveDefExample} in the graphical notation of Section~\ref{sec:GraphicalInterp}. Visually speaking, downward motions of dots happen first in $\gamma_{\x,\y}$, one dot at a time from bottom to top, and then upward motions of dots happen one dot at a time from top to bottom.
\end{example}

\begin{definition}\label{def:B0InverseFunctor}
Let 
\[
G: \F_2[U_1,\ldots,U_n] K(n,k) \to  \Cat_{\Quiv(\Gamma_U(n,k), \mc R_U)}
\]
be the unique $\F_2[U_1,\ldots,U_n]$-linear functor that is the identity on objects and such that $G(f_{\x,\y})$ is the morphism in $\Cat_{\Quiv(\Gamma_U(n,k), \mc R_U)}$ represented by $\gamma_{\x,\y}$ for all edges $f_{\x,\y}$ in $K(n,k)$. We have a corresponding homomorphism 
\[
G: \Path(K(n,k)) \to \Quiv(\Gamma_U(n,k), \mc R_U)
\]
of $\F_2[U_1,\ldots,U_n]^{V(n,k)}$-algebras.
\end{definition}

\subsubsection{Proving that the inverse homomorphism is well-defined}\label{sec:DefinitionOfGPart2}

In this section, we will show that the homomorphism $G$ of Definition~\ref{def:B0InverseFunctor} descends to a homomorphism $G: \B_0(n,k) \to \Quiv(\Gamma_U(n,k), \mc R_U)$; in Section~\ref{sec:B0EquivalenceOfDescriptions} below, we will show that $G$ is the inverse of $F: \Quiv(\Gamma_U(n,k), \mc R_U) \to \B_0(n,k)$. We start with some definitions and basic lemmas.

\begin{definition}\label{def:RLSegment}
If $\x, \y \in V(n,k)$ with $x_a < y_a$ for some $a\in[1,k]$ and $x_b = y_b$ for all $b \neq a$, we will call $\gamma_{\x,\y}$ an $R$-\emph{segment}. We define $L$-\emph{segments} similarly.
\end{definition}

We note the following consequences of Definitions~\ref{def:Recursive} and \ref{def:RLSegment}.
\begin{corollary}
If $\gamma_{\x,\y}$ is an $R$-segment or an $L$-segment, then $\gamma_{\x,\y}$ is the path constructed in Lemma~\ref{lem:UniquePathForRightMovingDot} or \ref{lem:UniquePathForLeftMovingDot} respectively.
\end{corollary}

\begin{corollary}\label{cor:GammaXYDecomposition}
For any $\x,\y \in V(n,k)$, either $\gamma_{\x,\y}$ is trivial or we can write $\gamma_{\x,\y} = \gamma_{\x,\x'} \gamma_{\x',\y}$ where:
\begin{enumerate}
\item $\gamma_{\x,\x'}$ is an $R$-segment or an $L$-segment;
\item $|\x' \cap \y| > |\x \cap \y|$;
\item\label{item:maximality} if $\gamma_{\x,\x'}$ is an $R$-segment then $y_a = x'_a$ for the index $a$ with $x_a < x'_a$, and $y_b \leq x'_b$ for all $b > a$;
\item\label{item:GammaXYDecomp L means no Rs} if $\gamma_{\x,\x'}$ is an $L$-segment then $y_a = x'_a$ for the index $a$ with $x'_a < x_a$, $y_b \leq x'_b$ for all $b$, and $y_b = x'_b$ for all $b < a$.
\end{enumerate}
\end{corollary}

\begin{figure}
\includegraphics[scale=0.38]{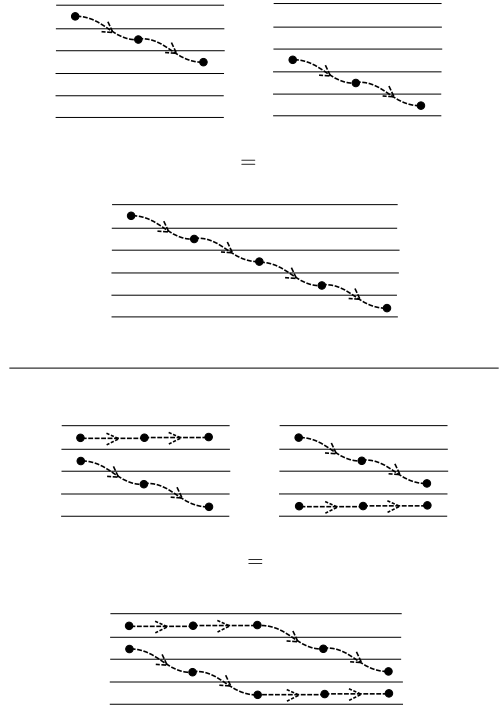}
\caption{Products of two $R$-segments in Lemma~\ref{lem:BothFactorsAreSegments}: first two cases.}
\label{fig:RandLsegmentsPart1}
\end{figure}

\begin{figure}
\includegraphics[scale=0.62]{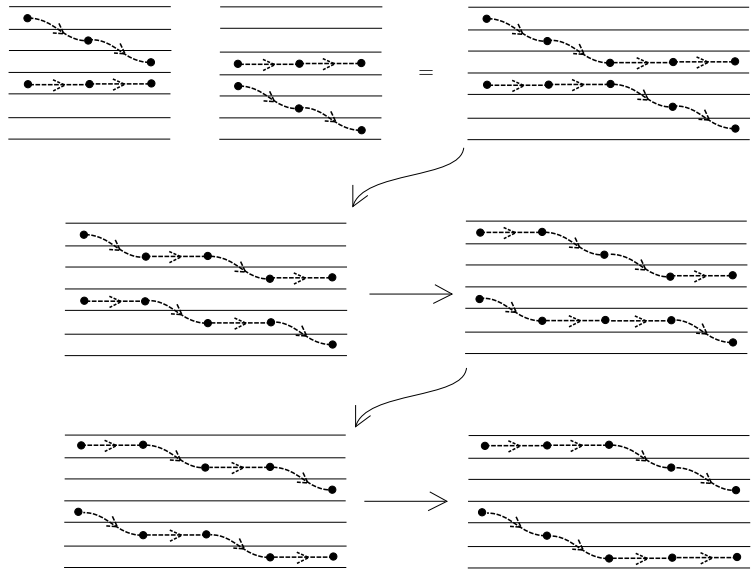}
\caption{Products of two $R$-segments in Lemma~\ref{lem:BothFactorsAreSegments}: final case.}
\label{fig:RandLsegmentsPart2}
\end{figure}

\begin{figure}
\includegraphics[scale=0.5]{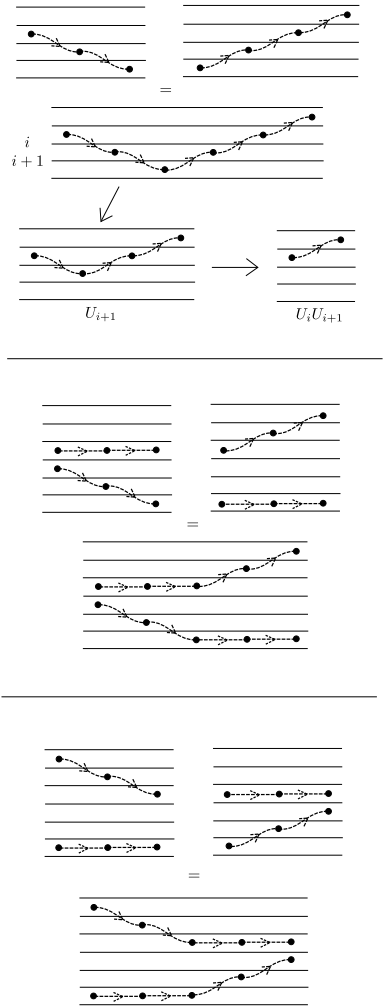}
\caption{Products of an $R$-segment and an $L$-segment in Lemma~\ref{lem:BothFactorsAreSegments}.}
\label{fig:RandLsegmentsPart3}
\end{figure}

\begin{figure}
\includegraphics[scale=0.5]{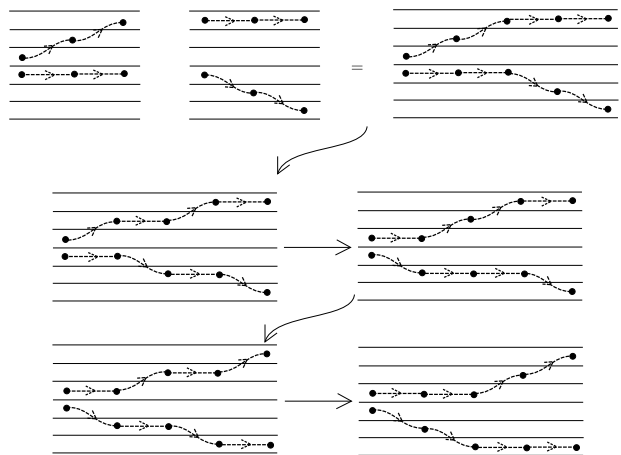}
\caption{Products of an $L$-segment and an $R$-segment in Lemma~\ref{lem:BothFactorsAreSegments}: one case.}
\label{fig:RandLsegmentsPart4}
\end{figure}

\begin{lemma}\label{lem:BothFactorsAreSegments}
If $\gamma_{\x,\x'}$ and $\gamma_{\x',\x''}$ are both $R$-segments, then in $\Quiv(\Gamma_U(n,k), \mc R_U)$, the element $\gamma_{\x,\x'} \gamma_{\x',\x''}$ is equal to a monomial in the $U_i$ variables times $\gamma_{\x,\x''}$. The same statement holds if $\gamma_{\x,\x'}$ and/or $\gamma_{\x',\x''}$ is an $L$-segment rather than an $R$-segment.
\end{lemma}

\begin{proof}
The proof is a case-by-case analysis that is left to the reader. See Figures~\ref{fig:RandLsegmentsPart1} and \ref{fig:RandLsegmentsPart2} when multiplying two $R$-segments, Figure~\ref{fig:RandLsegmentsPart3} when multiplying an $R$-segment by an $L$-segment, and Figure~\ref{fig:RandLsegmentsPart4} for one case of multiplying an $L$-segment by an $R$-segment. Multiplying two $L$-segments is similar to multiplying two $R$-segments.
\end{proof}

\begin{lemma}\label{lem:LeftFactorSegmentRightFactorArbitrary}
If $\gamma_{\x,\x'}$ is an $R$-segment or an $L$-segment and $\y \in V(n,k)$ is arbitrary, then in $\Quiv(\Gamma_U(n,k), \mc R_U)$, the element $\gamma_{\x,\x'} \gamma_{\x',\y}$ is equal to a monomial in the $U_i$ variables times $\gamma_{\x,\y}$. 
\end{lemma}

\begin{proof}
We will induct on $k - |\x' \cap \y|$. When $\x' = \y$, we have $\gamma_{\x,\x'} \gamma_{\x',\y} = \gamma_{\x,\x'} = \gamma_{\x,\y}$. 

For the inductive step, decompose $\gamma_{\x',\y}$ as $\gamma_{\x',\x''} \gamma_{\x'',\y}$ as in Corollary~\ref{cor:GammaXYDecomposition}. Since $\gamma_{\x,\x'}$ and $\gamma_{\x',\x''}$ are $R$-segments or $L$-segments, Lemma~\ref{lem:BothFactorsAreSegments} implies that $\gamma_{\x,\x'} \gamma_{\x',\x''}$ equals a monomial in the $U_i$ variables times $\gamma_{\x,\x''}$ in $\Quiv(\Gamma_U(n,k), \mc R_U)$. If $\gamma_{\x,\x''}$ is trivial, an $R$-segment, or an $L$-segment, we are done by induction because $|\x'' \cap \y| > |\x' \cap \y|$. Otherwise $\gamma_{\x,\x'} \gamma_{\x',\x''}$ is a product of two $R$-segments, two $L$-segments, an $R$-segment and an $L$-segment, or an $L$-segment and an $R$-segment.

First assume that $\gamma_{\x,\x'}$ and $\gamma_{\x',\x''}$ are $R$-segments but that $\gamma_{\x,\x''}$ is not an $R$-segment. We can write $\gamma_{\x,\x''} = \gamma_{\x,\tilde{\x}'} \gamma_{\tilde{\x}', \x''}$ where both factors are $R$-segments and we have $x_a < \tilde{x}'_a$, $\tilde{x}'_b < x''_b$, and $b < a$. By induction, $\gamma_{\tilde{\x}',\x''} \gamma_{\x'',\y}$ is equal to a monomial in the $U_i$ variables times $\gamma_{\tilde{\x}',\y}$ in $\Quiv(\Gamma_U(n,k), \mc R_U)$. Note that if $c > a$, then $y_c \leq x''_c = \tilde{x}'_c$ by item \eqref{item:maximality} of Corollary~\ref{cor:GammaXYDecomposition}. Thus, we have $\gamma_{\x,\y} = \gamma_{\x,\tilde{\x}'} \gamma_{\tilde{\x}',\y}$ by the recursive definition of $\gamma_{\x,\y}$. The case when $\gamma_{\x,\x'}$ and $\gamma_{\x',\x''}$ are both $L$-segments but $\gamma_{\x,\x''}$ is not an $L$-segment is similar.

Next, suppose that $\gamma_{\x,\x'}$ is an $R$-segment, $\gamma_{\x',\x''}$ is an $L$-segment, and $\gamma_{\x,\x''}$ is 
\begin{itemize}
\item nontrivial, 
\item not an $R$-segment, and
\item not an $L$-segment.
\end{itemize}
In this case we have $\gamma_{\x,\y} = \gamma_{\x,\x'} \gamma_{\x',\y}$ by item~\eqref{item:GammaXYDecomp L means no Rs} of Corollary~\ref{cor:GammaXYDecomposition} and the recursive definition of $\gamma_{\x,\y}$.

Finally, suppose that $\gamma_{\x,\x'}$ is an $L$-segment, $\gamma_{\x',\x''}$ is an $R$-segment, and $\gamma_{\x,\x''}$ is
\begin{itemize}
\item nontrivial,
\item not an $R$-segment, and
\item not an $L$-segment.
\end{itemize}
We can write $\gamma_{\x,\x''} = \gamma_{\x,\tilde{\x}'} \gamma_{\tilde{\x}',\x''}$ where $\gamma_{\x,\tilde{\x}'}$ is an $R$-segment, $\gamma_{\tilde{\x}',\x''}$ is an $L$-segment, $x_a < \tilde{x}'_a$, and $x''_b < \tilde{x}'_b$. By induction, $\gamma_{\tilde{\x}',\x''} \gamma_{\x'',\y}$ is equal to a monomial in the $U_i$ variables times $\gamma_{\tilde{\x}',\y}$ in $\Quiv(\Gamma_U(n,k), \mc R_U)$. For $c > a$, we have $y_c \leq x''_c \leq \tilde{x}'_c$ by item \eqref{item:maximality} of Corollary~\ref{cor:GammaXYDecomposition} together with the fact that $\gamma_{\tilde{\x}',\x''}$ is an $L$-segment. Thus, $\gamma_{\x,\y} = \gamma_{\x,\tilde{\x}'} \gamma_{\tilde{\x}',\y}$ by the recursive definition of $\gamma_{\x,\y}$, proving the lemma.
\end{proof}

\begin{proposition}\label{prop:BothFactorsArbitrary}
If $\x,\y,\z \in V(n,k)$, then $\gamma_{\x,\y} \gamma_{\y,\z}$ is equal to a monomial in the $U_i$ variables times $\gamma_{\x,\z}$ in $\Quiv(\Gamma_U(n,k), \mc R_U)$. 
\end{proposition}

\begin{proof}
We will induct on $k - |\x \cap \y|$. When $\x = \y$, we have $\gamma_{\x,\y} \gamma_{\y,\z} = \gamma_{\y,\z} = \gamma_{\x,\z}$. For the inductive step, decompose $\gamma_{\x,\y}$ as $\gamma_{\x,\y} = \gamma_{\x,\x'} \gamma_{\x',\y}$ using Corollary~\ref{cor:GammaXYDecomposition}; in particular, $\gamma_{\x,\x'}$ is an $R$-segment or an $L$-segment. We have $|\x' \cap \y| > |\x \cap \y|$, so by induction, $\gamma_{\x',\y} \gamma_{\y,\z}$ equals a monomial in the $U_i$ variables times $\gamma_{\x',\z}$ in $\Quiv(\Gamma_U(n,k), \mc R_U)$. By Lemma~\ref{lem:LeftFactorSegmentRightFactorArbitrary}, $\gamma_{\x,\x'} \gamma_{\x',\z}$ equals a monomial in the $U_i$ variables times $\gamma_{\x,\z}$ in $\Quiv(\Gamma_U(n,k), \mc R_U)$.
\end{proof}

\begin{corollary}
The monomial in Proposition~\ref{prop:BothFactorsArbitrary} is $\prod_{i=1}^n U_i^{(|v|_i(\y,\z) - |v|_i(\x,\z) + |v|_i(\x,\y))/2}$.
\end{corollary}

\begin{proof}
Let $M$ denote the monomial in question. Applying the functor $F$ from Section~\ref{sec:DefinitionOfF} in Proposition~\ref{prop:BothFactorsArbitrary}, we get $F(\gamma_{\x,\y}) F(\gamma_{\y,\z}) = MF(\gamma_{\x,\z})$ in $\Cat_{\B_0(n,k)}$, i.e. $f_{\x,\y} f_{\y,\z} = Mf_{\x,\z}$. We also know that 
\[
f_{\x,\y} f_{\y,\z} = \prod_{i=1}^n U_i^{(|v|_i(\y,\z) - |v|_i(\x,\z) + |v|_i(\x,\y))/2} f_{\x,\z}
\]
by the relations defining $\B_0(n,k)$. Since $\{f_{\x,\z}\}$ is a basis for 
\[
\Hom_{\Cat_{\B_0(n,k)}}(\z,\x) \cong \F_2[U_1,\ldots, U_n]
\]
as a free module over $\F_2[U_1,\dots,U_n]$, we conclude that $M$ is given by the stated formula.
\end{proof}

\begin{corollary}
The homomorphism $G: \Path(K(n,k)) \to \Quiv(\Gamma_U(n,k), \mc R_U)$ descends to a homomorphism of $\F_2[U_1,\ldots,U_n]^{V(n,k)}$-algebras $G: \B_0(n,k) \to \Quiv(\Gamma_U(n,k), \mc R_U)$.
\end{corollary}

Equivalently, we have a $\F_2[U_1,\ldots,U_n]$-linear functor $G: \Cat_{\B_0(n,k)} \to \Cat_{\Quiv(\Gamma_U(n,k), \mc R_U)}$.

\subsubsection{Equivalence of descriptions}\label{sec:B0EquivalenceOfDescriptions}
The following corollary concludes the proof of Theorem \ref{thm:IntroQuiverDescriptionB0}.

\begin{corollary}\label{cor:B_0 Equivalence}
The homomorphisms 
\[
F: \Quiv(\Gamma_U(n,k), \mc R_U) \to \B_0(n,k)
\]
from Section~\ref{sec:DefinitionOfF} and 
\[
G: \B_0(n,k) \to \Quiv(\Gamma_U(n,k), \mc R_U)
\]
from Sections~\ref{sec:DefinitionOfGPart1} and \ref{sec:DefinitionOfGPart2} are inverse isomorphisms of $\F_2[U_1,\ldots,U_n]^{V(n,k)}$-algebras. 

Equivalently, $F$ and $G$ can be viewed as inverse isomorphisms of $\F_2[U_1,\ldots,U_n]$-linear categories. 
\end{corollary}

\begin{proof}
By Definitions~\ref{def:Recursive} and \ref{def:B0InverseFunctor}, we have $FG(f_{\x,\y}) = f_{\x, \y}$ for all $\x,\y \in V(n,k)$, so it suffices to show that $GF(\gamma) = \gamma$ for every edge $\gamma$ of $\Gamma_U(n,k)$. If $\x$ and $\y$ are connected by an edge $\gamma$, then $\gamma_{\x,\y} = \gamma$. Thus, we have $GF(\gamma) = G(f_{\x,\y}) = \gamma_{\x,\y} = \gamma$ as desired. 
\end{proof}

By Lemma~\ref{lem:IntermediateQuiverEqv}, $F$ and $G$ give us inverse isomorphisms of $\F_2[U_1, \ldots, U_n]^{V(n,k)}$-algebras between $\B_0(n,k)$ and $\Quiv(\Gamma(n,k), \mc R)$. 
\begin{remark}\label{rem:FGNotationAbuse}
Below, we will often abuse notation and write $F$ and $G$ for these isomorphisms, rather than working with $\Gamma_U(n,k)$.
\end{remark}

\section{A quiver description of Ozsv\'ath--Szab\'o's algebra \texorpdfstring{$\B$}{B}}
\label{sec:OSzB}

\subsection{The algebras \texorpdfstring{$\B(n,k)$}{B(n,k)}}

While the algebra $\B_0(n,k)$ can be useful on its own (it is related to the degree-zero part of Ozsv{\'a}th--Szab{\'o}'s ``Pong algebra'' \cite{OSzPong} and appears in \cite{AlishahiDowlin}), Ozsv{\'a}th--Szab{\'o} work primarily with a quotient $\B(n,k)$ of $\B_0(n,k)$ in \cite{OSzNew}. They define this quotient in \cite[Definition 3.4]{OSzNew}; we will give an equivalent description in terms of the quiver $\Gamma(n,k)$. First, we review Ozsv{\'a}th--Szab{\'o}'s definition.

\begin{definition}[page 1115 of \cite{OSzNew}]\label{def:OSzStyleRLU}
Define an element $R_i$ of $\B_0(n,k)$ by
\[
R_i := \sum_{\x \in V(n,k), \,\, \x \cap \{i-1,i\} = \{i-1\}} f_{\x,\y},
\]
where for a given $\x$, we take $\y = (\x \setminus \{i-1\}) \cup \{i\}$. Similarly, define
\[
L_i := \sum_{\x \in V(n,k), \,\, \x \cap \{i-1,i\} = \{i\}} f_{\x,\y},
\]
where for a given $\x$, we take $\y = (\x \setminus \{i\}) \cup \{i-1\}$. Define
\[
U_i := \sum_{\x \in V(n,k)} U_i f_{\x,\x}.
\]
\end{definition}

\begin{remark}
We abuse notation by writing $U_i$ both for an element of $\B_0(n,k)$ and for an element of $\F_2[U_1,\ldots,U_n]$, on top of our further use of $R_i$, $L_i$, and $U_i$ for labels of edges in $\Gamma(n,k)$. We think this notation is well-motivated despite the potential risk of confusion.
\end{remark}

\begin{definition}[Definition 3.4 of \cite{OSzNew}]\label{def:OSzStyleBnk}
The algebra $\B(n,k)$ is the quotient of $\B_0(n,k)$ by the two-sided ideal generated by the following elements:
\begin{itemize}
\item $R_i R_{i+1}$ and $L_{i+1} L_i$ for $1 \leq i \leq n-1$
\item $\Ib_{\x} U_i$ for $1 \leq i \leq n$ and $\x \in V(n,k)$ with $\x \cap \{i-1,i\} = \varnothing$.
\end{itemize}
\end{definition}
More specifically, $\B(n,k)$ may be viewed as an algebra over $\F_2[U_1,\ldots,U_n]^{V(n,k)}$, since $\B_0(n,k)$ has this structure.

We now give a quiver description for $\B(n,k)$. Recall that to a path $\gamma$ in $\Gamma(n,k)$, we associate a noncommutative monomial $\mu(\gamma)$ in the letters $R_i$, $L_i$, and $U_i$ for $1 \leq i \leq n$.
\begin{definition}\label{def:B Quiver Algebra}
Define $\tilde{\mc R}$ to be the union of $\mc R \subset \Path(\Gamma(n,k))$ with the set of paths $\gamma$ such that $\mu(\gamma)$ is one of the following monomials for some $i$:
\begin{enumerate}
\item $R_i R_{i+1}$ or $L_{i+1} L_i$ (the ``two-line pass relations''),
\item\label{it:IsolatedUIsZero} $U_i$ if $\gamma$ is a loop at a vertex $\x \in V(n,k)$ with $\x\cap\{i-1,i\}=\emptyset$ (the ``$U$ vanishing relations'').
\end{enumerate}
\end{definition}
We can view the quotient $\Quiv(\Gamma(n,k), \td{\mc R})$ of $\Path(\Gamma(n,k))$ by the two-sided ideal generated by $\td{\mc R}$ as an algebra over $\F_2[U_1,\ldots,U_n]^{V(n,k)}$.

\begin{lemma}\label{lem:B Equivalence}
The isomorphisms $F$ and $G$ from Corollary~\ref{cor:B_0 Equivalence} (see also Remark~\ref{rem:FGNotationAbuse}) descend to isomorphisms of $\F_2[U_1,\ldots,U_n]^{V(n,k)}$-algebras between $\Quiv(\Gamma(n,k), \td{\mc R})$ and $\B(n,k)$.
\end{lemma}
\begin{proof}
By Corollary \ref{cor:B_0 Equivalence}, $F$ and $G$ descend to isomorphisms between $\Quiv(\Gamma(n,k), \td{\mc R})$ and the quotient of $\B_0(n,k)$ by the two-sided ideal ${\mc I}_{\td{\mc R}}$ generated by the image under $F$ of $\td{\mc R} \subset \Path(\Gamma(n,k))$. The elements listed in Definition~\ref{def:OSzStyleBnk} are sums of generators of ${\mc I}_{\td{\mc R}}$, so they are in ${\mc I}_{\td{\mc R}}$. Conversely, any generator of ${\mc I}_{\td{\mc R}}$ can be obtained from an element listed in Definition~\ref{def:OSzStyleBnk} via left multiplication by $\Ib_{\x}$ for some $\x \in V(n,k)$. 
\end{proof}

Definition \ref{def:B Quiver Algebra} can be understood visually in the same way as Definition \ref{def:B_0 Quiver Algebra}, with the new relations imposing the following new restrictions:
\begin{enumerate}
\item If a dot moves twice in the same direction, then the result is zero (the two-line pass relations),
\item $U_i$ loops are zero at vertices $\x$ having $\x\cap\{i-1,i\}=\emptyset$ (the $U$ vanishing relations).
\end{enumerate}
See Figure \ref{fig:IdemsAndMotionsPart3} for an illustration.

\begin{figure}
\includegraphics[scale=0.5]{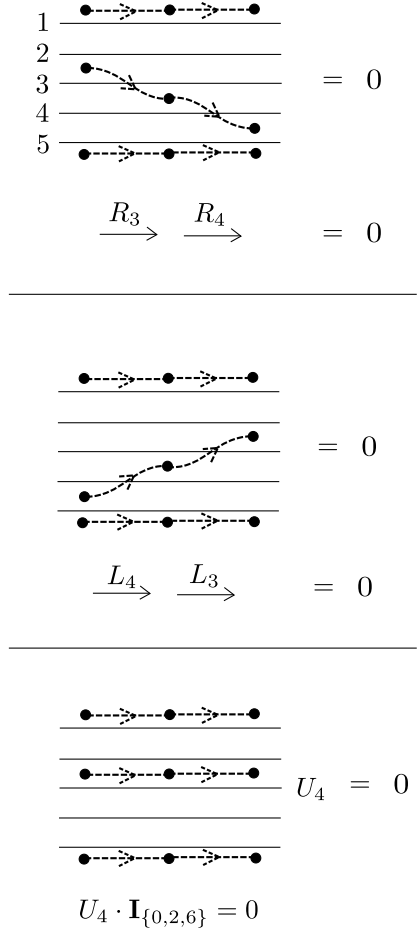}
\caption{Relations defining $\B(n,k)$ as a quotient of $\B_0(n,k)$: some examples.}
\label{fig:IdemsAndMotionsPart3}
\end{figure}

\subsection{The algebras \texorpdfstring{$\B(n,k,\Sc)$}{B(n,k,S)} for general orientations}\label{sec:GeneralOrientations}

In \cite{OSzNew}, bimodules over the algebra $\B(n,k)$ are assigned to braids oriented downwards (for us, these braids point leftwards; see Remark~\ref{rem:NinetyDegRot}). For more general orientations, Ozsv{\'a}th--Szab{\'o} define dg algebras $\B(n,k,\Sc)$ in \cite[Section 3.3]{OSzNew}. We review the definitions of these dg algebras below.

Let $\Sc$ be a subset of $[1,n]$; we think of $[1,n]$ as the left or right endpoints of a tangle projection (numbered from top to bottom, or from left to right in Ozsv{\'a}th--Szab{\'o}'s conventions), and then $i \in \Sc$ if and only if the projection is oriented rightwards through point $i$.

\begin{definition}[Section 3.3, \cite{OSzNew}]\label{def:GeneralBnksDef}
For $\Sc \subset [1,n]$ and $0 \leq k \leq n$, the dg algebra $\B(n,k,\Sc)$ is defined to be the tensor product of $\B(n,k)$ with an exterior algebra in variables $C_i$ for $i \in \Sc$, where $\de(C_i) = U_i = \sum_{\x \in V(n,k)} \Ib_{\x} U_i$. More concisely,
\[
\B(n,k,\Sc) := \frac{\B(n,k)[C_i \,|\, i \in \Sc]}{(C_i^2 = 0, \,\, \de(C_i) = U_i)_{i \in \Sc}}.
\]
We may identify $\B(n,k)$ with $\B(n,k,\varnothing)$. Gradings will be defined in Section~\ref{sec:OSz gradings} below.
\end{definition}

Generalizing the description of $\B(n,k)$ as $\Quiv(\Gamma(n,k), \tilde{\mc R})$, we can give a quiver description of $\B(n,k,\Sc)$. 
\begin{definition}\label{def:BnksQuiverDescription}
Let $\Gamma(n,k,\Sc)$ be obtained from $\Gamma(n,k)$ by adding an arrow from $\x$ to itself, for all $\x \in V(n,k)$ and for each $i \in \Sc$, with a new type of label $C_i$. To the relation set $\tilde{\mc R}$, we add the set of elements $\gamma$ such that $\mu(\gamma)$ is equal to one of the following:
\begin{itemize}
\item $C_i^2$ (the ``$C^2$ vanishing relations''),
\item $C_i A + A C_i$ for any label $A = R_j, L_j, U_j$, or $C_j$ (the ``$C$ central relations'').
\end{itemize}
Let $\tilde{\mc R}_\Sc$ denote this new relation set.  We declare that for each arrow labeled $C_i$ at a vertex $\x$, the differential of the corresponding generator of $\Quiv(\Gamma(n,k,\Sc), \tilde{\mc R}_{\Sc})$ is the arrow labeled $U_i$ at the vertex $\x$. By Proposition~\ref{prop:DGAlgFromQuiver}, we get a differential algebra structure on $\Quiv(\Gamma(n,k,\Sc), \tilde{\mc R}_{\Sc})$.
\end{definition}
We will also use the notation $\tilde{\mc R}_{\x,\y,\Sc}:=\tilde{\mc R}_\Sc \cap \Ib_\x \Path(\Gamma(n,k)) \Ib_\y$.
\begin{proposition}\label{prop:OSzQuiverEquivDifferential}
The isomorphisms $F$ and $G$ from Lemma~\ref{lem:B Equivalence} extend to isomorphisms of differential algebras over $\F_2[U_1,\ldots,U_n]^{V(n,k)}$ between $\B(n,k,\Sc)$ and $\Quiv(\Gamma(n,k,\Sc), \tilde{\mc R}_{\Sc})$. 
\end{proposition}

\begin{proof}
One can extend $F$ by sending the $C_i$ loops at $\x$ in $\Gamma(n,k,\Sc)$ to the elements $I_\x C_i$ in $\B(n,k,\Sc)$, while $G$ sends $C_i\in\B(n,k,\Sc)$ to the sum of $C_i$ loops at all $\x$ in $\Gamma(n,k,\Sc)$. By the differential analogue of Proposition~\ref{prop:QuiverAlgUniversalProp}, the maps $F$ and $G$ are still inverse isomorphisms of differential algebras; the new relations on each side are satisfied on the other and the extended maps $F$ and $G$ respect the differential.
\end{proof}

\subsection{Alexander and Maslov gradings} \label{sec:OSz gradings}

In \cite[Section 3.4]{OSzNew}, Ozsv{\'a}th--Szab{\'o} define an Alexander multi-grading $w$ by $\left(\frac12\Z\right)^n$ and a Maslov grading $\m$ by $\Z$ on $\B(n,k,\Sc)$ which we review below.  However, we begin with an ``unrefined'' version of the Alexander multi-grading which is not mentioned in \cite{OSzNew}.

\begin{definition} \label{def:UnrefinedAlexOSz}
The \emph{unrefined Alexander multi-grading} on $\B(n,k,\Sc)$ is a grading by $\Z^{2n}$ denoted by $w^{\un}$ and defined as follows.  Write $\tau_1, \beta_1, \ldots, \tau_n, \beta_n$ for the standard basis of $\Z^{2n}$. For $1 \leq i \leq n$ and an edge $\gamma$ of $\Gamma(n,k,\Sc)$, we set
\begin{equation*}
\begin{aligned}
w^{\un}(\gamma)&:=\tau_i \textrm{ if } \gamma \textrm{ has label } R_i \\
w^{\un}(\gamma)&:=\beta_i \textrm{ if } \gamma \textrm{ has label } L_i \\
w^{\un}(\gamma)&:=\tau_i + \beta_i \textrm{ if } \gamma \textrm{ has label } U_i \textrm{ or } C_i. \\
\end{aligned}
\end{equation*}
For $\gamma = (\gamma_1, \ldots, \gamma_l)$, we define the unrefined Alexander multi-degree $w^{\un}(\gamma)$ to be 
\[
w^{\un}(\gamma) := \sum_{j=1}^l w^{\un}(\gamma_j).
\]

Since each element of $\tilde{\mc R}_{\Sc} \subset \Path(\Gamma(n,k,\Sc))$ is $w^{\un}$-homogeneous, we get an Alexander multi-grading by $\Z^{2n}$ on $\B(n,k,\Sc) \cong \Quiv(\Gamma(n,k,\Sc), \tilde{\mc R}_{\Sc})$ by Proposition~\ref{prop:GradedQuiverAlg}.
\end{definition}

We can pass from the unrefined Alexander multi-grading to the ``refined'' version of Ozsv{\'a}th--Szab{\'o} as in the following definition.

\begin{definition}[{\cite[Section 3.4]{OSzNew}}] \label{def:RefinedAlexOSz}
The \emph{(refined) Alexander multi-grading} on $\B(n,k,\Sc)$ is a grading by $\left(\frac{1}{2} \Z\right)^n$ denoted by $w$ and defined as follows.  Write $e_1,\ldots,e_n$ for the standard basis elements of $\Z^n$.  Let $\varphi: \Z^{2n} \rightarrow \left( \frac{1}{2} \Z \right)^n$ denote the homomorphism defined by setting $\varphi(\tau_i)=\varphi(\beta_i)=\frac{1}{2}e_i$ for all $1\leq i\leq n$, where $\tau_i,\beta_i$ form the basis for $\Z^{2n}$ as in Definition \ref{def:UnrefinedAlexOSz}.  Then we define $w:=\varphi\circ w^{\un}$.  Explicitly, for an edge $\gamma$ of $\Gamma(n,k,\Sc)$ representing an element of $\B(n,k,\Sc)$, we have
\begin{equation*} \label{eq:Multigrading}
\begin{aligned}
w(\gamma)&:=\frac{1}{2}e_i \textrm{ if } \gamma \textrm{ has label } R_i \textrm{ or } L_i\\
w(\gamma)&:=e_i \textrm{ if } \gamma \textrm{ has label } U_i \textrm{ or } C_i.
\end{aligned}
\end{equation*}
We will also use the notation $w_i(a)$ to denote the coefficient of $w(a)$ on the basis element $e_i$.
\end{definition}

We will often refer to the refined Alexander multi-grading as simply the Alexander multi-grading.

\begin{remark}
In \cite{ManionKS}, the meaning of ``refined'' and ``unrefined'' grading was reversed. Here we use terminology following \cite[Section 3.3]{LOT} in line with \cite{MMW2} where we relate the gradings on Ozsv{\'a}th--Szab{\'o}'s algebras with group-valued gradings on strands algebras.
\end{remark}

Going one step further, we can collapse the Alexander multi-grading to a single Alexander grading by $\frac{1}{2}\Z$.
\begin{definition}[Equation 3.8 of \cite{OSzNew}] \label{def:SingleAlexOSz}
The (single) \emph{Alexander grading} on $\B(n,k,\Sc)$ is a grading by $\frac{1}{2}\Z$ defined by 
\[
\Alex(\gamma) := -\sum_{i \in \Sc} w_i(\gamma) + \sum_{i \notin \Sc} w_i(\gamma).
\]
\end{definition}

Finally, we define homological gradings, also known as Maslov gradings.
\begin{definition}[Equation 3.9 of \cite{OSzNew}] \label{def:OSzMaslovDegrees}
For a path $\gamma$ in $\Gamma(n,k,\Sc)$, we define the \emph{Maslov degree} of $\gamma$ to be 
\[
\m(\gamma) := \#_{C}(\gamma) -2 \sum_{i \in \Sc} w_i(\gamma),
\]
where $\#_{C}(\gamma)$ is the number of edges in $\gamma$ labeled $C_i$ for some $i \in \Sc$. Concretely, if $\gamma$ is a single edge we have:
\begin{itemize}
\item $\m(\gamma) = 0$ if $\gamma$ has label $R_i$, $L_i$, or $U_i$ and $i \notin \Sc$,
\item $\m(\gamma) = -1$ if $\gamma$ has label $R_i$, $L_i$, or $C_i$ and $i \in \Sc$, and
\item $\m(\gamma) = -2$ if $\gamma$ has label $U_i$ and $i \in \Sc$.
\end{itemize}
\end{definition}

The following corollary concludes the proof of Theorem \ref{thm:IntroQuiverDescription}.

\begin{corollary} \label{cor:OSzQuiverEquivDG}
When using the refined or single Alexander gradings, together with the Maslov grading, the isomorphisms $F$ and $G$ from Proposition~\ref{prop:OSzQuiverEquivDifferential} are isomorphisms of dg algebras over $\F_2[U_1,\ldots,U_n]^{V(n,k)}$ between $\B(n,k,\Sc)$ and $\Quiv(\Gamma(n,k,\Sc), \tilde{\mc R}_{\Sc})$.
\end{corollary}

\begin{proof}
By Proposition~\ref{prop:OSzQuiverEquivDifferential}, we only need to check that $F$ preserves gradings (we have $G = F^{-1}$). Translating Ozsv{\'a}th--Szab{\'o}'s definition of the Alexander multi-grading from \cite[Section 3.4]{OSzNew} into our terminology, let 
\[
a = U_1^{r_1} \cdots U_n^{r_n} f_{\x,\y}
\]
be a generator of $\B(n,k)$. Ozsv{\'a}th--Szab{\'o} define the Alexander multi-degree of $a$ to have $i^{th}$ component equal to the quantity they call $w_i(a)$, which in our notation is $r_i + \frac{|v|_i(\x,\y)}{2}$. More generally, for a generator 
\[
a = C_{i_1} \cdots C_{i_l} U_1^{r_1} \cdots U_n^{r_n} f_{\x,\y}
\]
of $\B(n,k,\Sc)$ (with $i_1, \ldots, i_l \in \Sc$), Ozsv{\'a}th--Szabo define the degree of $a$ by declaring that $C_i$ contributes $1$ to the $i^{th}$ component of the Alexander multi-degree and $0$ to all other components. One can check that for an edge $\gamma$ of $\Gamma(n,k,\Sc)$ labeled $R_i$, $L_i$, $U_i$, or $C_i$, the Alexander multi-degree of $\gamma$ from Definition~\ref{def:RefinedAlexOSz} agrees with Ozsv{\'a}th--Szab{\'o}'s Alexander multi-degree of $F(\gamma) \in \B(n,k,\Sc)$; indeed, a similar computation is given in the second paragraph of \cite[Section 3.4]{OSzNew} for the elements $R_i$, $L_i$, and $U_i$ of Definition~\ref{def:OSzStyleRLU}.

Our single Alexander degree and Ozsv{\'a}th--Szab{\'o}'s are obtained from the Alexander multi-degrees by the same specialization, so $F$ preserves single Alexander degrees as well.  Finally, since $F$ preserves Alexander multi-degrees and the number of $C_i$ variables, $F$ also preserves Maslov degrees.
\end{proof}
Ozsv{\'a}th--Szab{\'o} do not discuss the unrefined Alexander multi-grading, so there is no need for a comparison result in this case.

\begin{remark}
Using Propositions~\ref{prop:GradedQuiverAlg} and \ref{prop:DGAlgFromQuiver}, one can package the Maslov grading and the Alexander multi-grading into a grading by $(G,\lambda)$ as in Section~\ref{sec:AppendixAlgebras}, with $G = \Z \oplus \left(\frac12\Z\right)^n$ and $\lambda = (1,0)$. One can also package the Maslov grading with the unrefined grading of Definition~\ref{def:UnrefinedAlexOSz} or the single Alexander grading of Definition~\ref{def:SingleAlexOSz} similarly.
\end{remark}

\subsection{Idempotent-truncated algebras}\label{sec:OSzTruncatedAlgs}

As in \cite[Section 12]{OSzNew}, one can define dg algebras related to $\B(n,k,\Sc)$ by taking full subcategories of $\Cat_{\B(n,k,\Sc)}$.
\begin{definition}\label{def:TruncatedOSzAlgs}
Define algebras $\B_r(n,k,\Sc)$, $\B_l(n,k,\Sc)$, and $\B'(n,k,\Sc)$ as follows:
\begin{itemize}
\item $\B_r(n,k,\Sc) := \Alg_{B_r}$ where $B_r$ is the full dg subcategory of $\Cat_{\B(n,k,\Sc)}$ on objects $\x \in V(n,k)$ with $0 \notin \x$; call the set of these objects $V_r(n,k)$.
\item $\B_l(n,k,\Sc) := \Alg_{B_l}$ where $B_l$ is the full dg subcategory of $\Cat_{\B(n,k,\Sc)}$ on objects $\x \in V(n,k)$ with $n \notin \x$; call the set of these objects $V_l(n,k)$.
\item $\B'(n,k,\Sc) := \Alg_{B'}$ where $B'$ is the full dg subcategory of $\Cat_{\B(n,k,\Sc)}$ on objects $\x \in V(n,k)$ with $0, n \notin \x$; call the set of these objects $V'(n,k)$.
\end{itemize}
\end{definition}

Without reference to categories, we can write $\B_r(n,k,\Sc)$ as
\[
\bigg( \sum_{\x: 0 \notin \x} \Ib_{\x} \bigg) \B(n,k,\Sc) \bigg( \sum_{\x: 0 \notin \x} \Ib_{\x} \bigg),
\]
and similarly for $\B_l(n,k,\Sc)$ and $\B'(n,k,\Sc)$. The gradings on $\B(n,k,\Sc)$ give rise to gradings on the idempotent-truncated algebras $\B_r(n,k,\Sc)$, $\B_l(n,k,\Sc)$, and $\B'(n,k,\Sc)$.

\begin{remark}
In \cite{ManionDecat}, the algebras $\B_{r}(n,k,\Sc)$ and $\B_l(n,k,\Sc)$ were called $\Cc_{r}(n,k,\Sc)$ and $\Cc_l(n,k,\Sc)$, following old notation of Ozsv{\'a}th--Szab{\'o}, and shown to categorify tensor products $V^{\otimes \Sc}$ of the vector representation $V$ of $\gloneone$ and its dual (depending on the orientations $\Sc$).
\end{remark}

Quiver descriptions of $\B_r(n,k,\Sc)$, $\B_l(n,k,\Sc)$, and $\B'(n,k,\Sc)$ will be given in Section~\ref{sec:QuiversForTruncatedAlgs}, once we have reviewed \cite[Proposition 3.7]{OSzNew} (see Proposition~\ref{prop:OSzBasis} below).

\section{The structure of Hom-spaces}\label{sec:OSzStructure}

\subsection{Far pairs of vertices and crossed lines}\label{sec:NotFar}

The visual interpretation of Section \ref{sec:GraphicalInterp} motivates the following definitions.
\begin{definition}[\cite{OSzNew}, Definition 3.5]
\label{def:far}
Vertices $\x,\y \in V(n,k)$ are \emph{far} (from each other) if there is some $a\in[1,k]$ such that $|x_a - y_a| > 1$. Otherwise they are \emph{not far}.
\end{definition}

Note that if $\x$ and $\y$ are not far, then $v_i(\x, \y) \in \set{-1,0,1}$ for all $i\in[1,n]$. It follows from \cite[Proposition 3.7]{OSzNew} that if $\x$ and $\y$ are far from each other then $\Ib_{\x} \B(n,k,\Sc) \Ib_{\y} = 0$; we will review this proposition below.

The terminology in the next definitions is also due to Ozsv{\'a}th--Szab{\'o}, although it does not appear explicitly in \cite{OSzNew}.
\begin{definition}\label{def:CrossedLine}
Let $\x,\y \in V(n,k)$ and suppose $\x$ and $\y$ are not far. In terms of the graphical interpretation from Section~\ref{sec:GraphicalInterp}, indices $i \in [1,n]$ correspond to horizontal lines, arranged in parallel and numbered from top to bottom. We say that line $i$ is a \emph{crossed line} if $v_i(\x,\y) \neq 0$, and we let $\CL{\x,\y}:=\{i\in[1,n] \,|\, v_i(\x,\y)\neq 0\}$ denote the set of crossed lines from $\x$ to $\y$.

By Lemma~\ref{lem:vi=Ri-Li}, if $i\in\CL{\x,\y}$ then every path $\gamma$ in $\Gamma(n,k)$ from $\x$ to $\y$ contains at least one edge labeled $R_i$ or $L_i$; the converse is also true. Thus, graphically speaking, we have $i\in\CL{\x,\y}$ if and only if every motion of dots from $\x$ to $\y$ involves a dot crossing over line $i$.
\end{definition}

\begin{definition}\label{def:FullyUsed}
Let $\x,\y \in V(n,k)$. A coordinate $i \in [0,n] \setminus (\x \cap \y)$ is called a \emph{not-fully-used} coordinate. Coordinates $i \in \x \cap \y$ are called \emph{fully-used} coordinates.
\end{definition}
Visually, elements of $[0,n]$ correspond to the regions between and outside the horizontal lines in Definition~\ref{def:CrossedLine}. A coordinate is fully-used if both $\x$ and $\y$ have a dot in the corresponding region. Note that this does not ensure that this dot was ``stationary'' in a minimal motion from $\x$ to $\y$; algebraically, a coordinate $i \in [0,n]$ is fully used if $i = x_a = y_b$ for some $a,b \in [1,k]$, but we may have $a \neq b$.

\subsection{The structure of \texorpdfstring{$\B(n,k,\Sc)$}{B(n,k,S)} via generating intervals}\label{sec:GenInt}
Given two vertices $\x,\y\in V(n,k)$ that are not far, there is a helpful way of describing the relations in $\Ib_\x \B(n,k,\Sc) \Ib_\y$, as discussed in \cite[Section 3.2]{OSzNew}.  The main idea is as follows.  An additive generator $a\in\Ib_\x \B(n,k,\Sc) \Ib_\y$ can be represented by a path $\gamma\in\Gamma(n,k,\Sc)$ from $\x$ to $\y$.  Modulo the relations, one may be able to replace some $U_i$-loops in $\gamma$ by sequences of edges labeled $(R_i,L_i)$ or $(L_i,R_i)$. In this way, one may get a new path $\gamma'$ representing $a$ that passes through a new vertex $\x'$ (not passed through by $\gamma$). If $\x' \cap\{j-1,j\} = \varnothing$ for some $j$ and $\gamma'$ contains a $U_j$ loop, one can then commute this $U_j$ loop past other edges of $\gamma'$ until it is based at $\x'$, implying that $\gamma' = 0$ and thus $a = 0$. Fortunately, the cases where this occurs can be summarized in a simple way with the help of the following definition.

\begin{definition}[Definition 3.6 of \cite{OSzNew}]\label{def:GenInt}
Let $\x,\y \in V(n,k)$ and suppose that $\x$ and $\y$ are not far. A \emph{generating interval} for $\x$ and $\y$ is a sequence of coordinates $[j+1,j+l] \subset [0,n]$ such that:
\begin{itemize}
\item The coordinates $j$ and $j+l$ are not fully used, but the coordinate $t$ is fully used for $j < t < j+l$.
\item $\CL{\x,\y}\cap[j+1,j+l]=\varnothing$, i.e. all of the dots between lines $j+1$ and $j+l$ can be viewed as ``stationary'' (see Definition \ref{def:CrossedLine}).
\end{itemize}
We say the \emph{length} of a generating interval $G = [j+1,j+l]$ is $l$. If $G = [j+1,j+l]$ is a generating interval, it has an associated (commutative) monomial $p_G$ in the variables $U_1,\ldots,U_n$ defined by $p_G := U_{j+1} \cdots U_{j+l}$.
\end{definition}

Visually, a generating interval is a sequence of lines surrounding stationary dots for the minimal motion from $\x$ to $\y$, with each region on either end of the interval being empty in either $\x$ or $\y$. 

\begin{figure}
\includegraphics[scale=0.5]{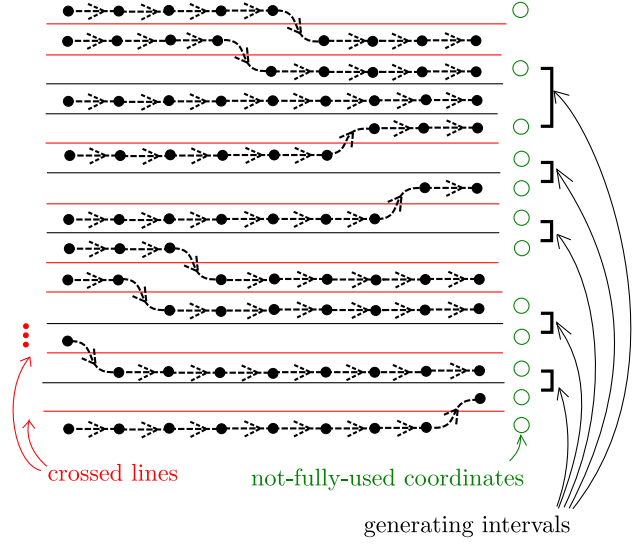}
\caption{Crossed lines, not-fully-used coordinates, and generating intervals in Example~\ref{ex:GenIntsEtc}.}
\label{fig:GenIntsExample}
\end{figure}

\begin{example}\label{ex:GenIntsEtc}
Let $\x = \{0,1,3,5,7,8,9,11,14\}$ and $\y = \{1,2,3,4,6,9,10,12,13\}$; $\x$ and $\y$ are elements of $V(14,9)$ that are not far (the similar-looking elements in Example~\ref{ex:GammaXYExample} are far). Crossed lines between $\x$ and $\y$ are shown in red in Figure~\ref{fig:GenIntsExample}. Not-fully-used coordinates are indicated with green circles to the right of the figure, and the generating intervals are also shown to the right. The monomials $p_G$ for the generating intervals $G$ are $U_3 U_4$, $U_6$, $U_8$, $U_{11}$, and $U_{13}$. 
\end{example}

In analogy to Definition \ref{def:GenInt}, we define a variant of generating intervals, which we call edge intervals.

\begin{definition}\label{def:LeftEdgeInterval}
Let $\x, \y \in V(n,k)$ be I-states that are not far. For $1 \leq l < n$, we say that $[[1,l]$ is a \emph{left edge interval} for $\x$ and $\y$ if the coordinate $l$ is not fully used, but the coordinate $t$ is fully used for $0 \leq t < l$. Note that in this case, up to coordinate $l$ there are no crossed lines, i.e. all of the dots above line $l$ can be viewed as stationary. We say the \emph{length} of a left edge interval $G = [[1,l]$ is $l$.
\end{definition}

\begin{definition}\label{def:RightEdgeInterval}
Let $\x, \y \in V(n,k)$ be I-states that are not far. For $1 \leq l < n$, we say that $[n-l+1,n]]$ is a \emph{right edge interval} for $\x$ and $\y$ if the coordinate $n-l$ is not fully used, but the coordinate $t$ is fully used for $n-l+1 \leq t \leq n$. In this case, after coordinate $n-l$ there are no crossed lines, i.e. all of the dots below line $n-l+1$ can be viewed as stationary. We say the \emph{length} of a right edge interval $G = [n-l+1,n]]$ is $l$.
\end{definition}

\begin{definition}
\label{def:TwoFacedEdgeInterval}
Let $\x=\y=[0,n]\in V(n,n+1)$.  We say that $[[1,n]]$ is a \emph{two-faced edge interval} for $\x$ and $\y$, with length $n+1$.
\end{definition}
Note that for $\x = \y = [0,n]$, all coordinates $t \in [0,n]$ are fully-used, and there are no crossed lines from $\x$ to $\y$.

\begin{proposition}
\label{prop:QuartumNonDatur}
Given $\x, \y \in V(n,k)$ not far, for each $i \in [1,n]$ exactly one of the following is true:
\begin{enumerate}
\item \label{it:QND1} $i\in\CL{\x,\y}$ (line $i$ is crossed);
\item \label{it:QND2} there exists a unique generating interval $G$ such that $i \in G$;
\item \label{it:QND3} there exists a unique (left, right, or two-faced) edge interval $G$ such that $i \in G$.
\end{enumerate}
\end{proposition}

We will prove Proposition~\ref{prop:QuartumNonDatur} with the help of the following lemma.

\begin{lemma}
\label{lem:QuartumNonDatur}
Suppose that $\x, \y \in V(n,k)$ are not far and fix $i\in[1,n]\setminus\CL{\x,\y}$. Then:
\begin{enumerate}
\item \label{it:ncl1} $i$ belongs to a generating interval $[j+1,j+l]$ if and only if there exist a coordinate $t \leq i-1$ and a coordinate $t' \geq i$ that are not fully used. In this case we have 
\begin{align}
\label{eq:iandj}
&j = \max\set{t \leq i-1 \,\middle|\, \text{$t$ is not fully used}}, \\
&l = \min\set{t \geq i \,\middle|\, \text{$t$ is not fully used}} - j. \nonumber
\end{align}

\item \label{it:ncl2} $i$ belongs to a left edge interval $[[1,l]$ if and only if all coordinates $t \leq i-1$ are fully used and there exists a non-fully used coordinate $t' \geq i$. In this case we have
\[
l = \min\set{t \geq i \,\middle|\, \text{$t$ is not fully used}}.
\]
\item \label{it:ncl3} $i$ belongs to a right edge interval $[n-l+1,n]]$ if and only if there exists a non-fully used coordinate $t \leq i-1$ and all coordinates $t' \geq i$ are fully used. In this case we have
\[
l = n - \max\set{t \leq i-1 \,\middle|\, \text{$t$ is not fully used}}.
\]
\item \label{it:ncl4} $i$ belongs to a two-faced edge interval $[[1,n]]$ if and only if all coordinates are fully used.
\end{enumerate}

\end{lemma}

\begin{proof}
We only prove \eqref{it:ncl1}, since the proof in the other cases requires only straightforward variations. First, if $i$ belongs to a generating interval $[j+1,j+l]$, then $j$ and $j+l$ are non-fully-used coordinates satisfying $j \leq i-1$ and $j+l\geq i$. Moreover, from Definition~\ref{def:GenInt}, $j$ and $l$ must be given by the formulas in equation \eqref{eq:iandj}.

Conversely, suppose that there exist a coordinate $t \leq i-1$ and a coordinate $t' \geq i$ that are not fully used. Then $j$ and $l$ from equation \eqref{eq:iandj} are well defined. Recall that line $i$ is not crossed, that is, $v_i(\x, \y) = 0$. It follows that $v_{t}(\x, \y) = 0$ for all $t \in [j+1, i-1]$, since $v_i(\x, \y) = 0$ and all coordinates between $j+1$ and $i-1$ are fully used. Analogously, $v_{t}(\x, \y) = 0$ for all $t \in [i,j+l]$. Then, by definition, the interval $[j+1, j+l]$ is a generating interval containing $i$.
\end{proof}

\begin{proof}[{Proof of Proposition \ref{prop:QuartumNonDatur}}]
If line $i$ is crossed, then, by Definitions \ref{def:GenInt}, \ref{def:LeftEdgeInterval}, \ref{def:RightEdgeInterval} and \ref{def:TwoFacedEdgeInterval}, $i$ does not belong to any generating or edge interval. Thus, \eqref{it:QND1} is true and \eqref{it:QND2} and \eqref{it:QND3} are false in this case.

Now suppose that line $i$ is not crossed. Note that we must be in exactly one of the following 4 cases:
\begin{enumerate}
\item there exist a coordinate $t \leq i-1$ and a coordinate $t' \geq i$ that are not fully used;
\item all coordinates $t \leq i-1$ are fully used and there exists a non-fully used coordinate $t'\geq i$;
\item there exists a non-fully used coordinate $t\leq i-1$ and all coordinates $t'\geq i$ are fully used;
\item all coordinates are fully used.
\end{enumerate}
By Lemma \ref{lem:QuartumNonDatur} there exists exactly one generating or edge interval containing $i$, so we are done.
\end{proof}

The following proposition from \cite{OSzNew} shows that generating intervals provide all the relations within $\Ib_\x \B(n,k) \Ib_\y$.
\begin{proposition}[{\cite[Proposition 3.7]{OSzNew}}]\label{prop:OSzBasis}
\label{prop:Generating Intervals}
For $\x,\y \in V(n,k)$, let $\phi = \phi^{\x,\y}$ be the isomorphism of $\F_2[U_1,\ldots,U_n]$-modules from Definition~\ref{def:OSzStyleDef}. Its inverse
\[
\phi^{-1} \colon \Ib_\x \B_0(n,k) \Ib_\y \to \F_2[U_1, \ldots, U_n]
\]
induces an isomorphism
\begin{equation*}
\phi^{-1} \colon \Ib_\x \B(n,k) \Ib_\y \stackrel{\sim}{\longrightarrow}
\begin{cases}
0 & \text{if $\x$ and $\y$ are far}\\
\frac{\F_2[U_1, \ldots, U_n]}{\left(\, {p_G \,|\, G \text{ generating interval}} \,\right) } & \text{otherwise.}
\end{cases}
\end{equation*}
Thus, a basis over $\F_2$ for $\Ib_{\x} \B(n,k) \Ib_{\y}$ is given by the elements $\phi(p)$ where $p$ is a monomial in $U_1, \ldots, U_n$ that is not divisible by $p_G$ for any generating interval $G$ for $\x$ and $\y$. It follows that a basis for $\Ib_{\x} \B(n,k,\Sc) \Ib_{\y}$ is given by elements $\phi(p)$ times square-free monomials in variables $C_i$ for $i \in \Sc$.
\end{proposition}

\begin{corollary}\label{cor:ExplicitQuiverGensForMonomials}
Under the isomorphism 
\[
\frac{\F_2[U_1,\ldots,U_n]}{(p_G)} \xrightarrow{\phi} \Ib_{\x} \B(n,k) \Ib_{\y} \xrightarrow{G} \Ib_{\x} \Quiv(\Gamma(n,k), \tilde{\mathcal{R}}) \Ib_{\y},
\]
a monomial $U_1^{r_1} \cdots U_n^{r_n}$ gets sent to $\gamma_U \cdot \gamma_{\x,\y}$ where $\gamma_{\x,\y}$ is the path from Definition~\ref{def:Recursive} and $\gamma_U$ is a product of $U_i$ loops at $\x$ with multiplicities $r_i$.
\end{corollary}

Corollary~\ref{cor:ExplicitQuiverGensForMonomials} will be useful in \cite{MMW2} when proving that the algebra map 
\[
\Phi: \B(n,k,\Sc) \to \sac nk\Sc
\]
constructed in that paper is a quasi-isomorphism in the case $\Sc = \varnothing$.

\subsection{A splitting theorem}\label{sec:OSzSplittingTheorem}

Proposition~\ref{prop:Generating Intervals} implies that $\Ib_{\x} \B(n,k,\Sc) \Ib_{\y}$ decomposes as a tensor product of chain complexes. First, we introduce special cases of Ozsv{\'a}th--Szab{\'o}'s algebras that we will call generating algebras and edge algebras.

\begin{definition}\label{def:GenAlgebra}
For $l \geq 0$ and $\Sc \subset [1,l]$, define the \emph{generating algebra} $\overline{B}(l,\Sc)$ to be
\[
\overline{B}(l,\Sc) := \Ib_{[1,l-1]} \B(l,l-1,\Sc) \Ib_{[1,l-1]}.
\]
Similarly, define the \emph{left edge algebra} $\overline{B}_{\lda}(l,\Sc)$ to be $\Ib_{[0,l-1]} \B(l,l,\Sc) \Ib_{[0,l-1]}$, and define the \emph{right edge algebra} $\overline{B}_{\rho}(l,\Sc)$ to be $\Ib_{[1,l]} \B(l,l,\Sc) \Ib_{[1,l]}$. Define the \emph{two-faced edge algebra} $\overline{B}_{\lda\rho}(l,\Sc)$ to be $\B(l,l+1,\Sc)$.
\end{definition}

Now let $\x, \y \in V(n,k)$ be not far. Based on the structure of the generating intervals and edge intervals for $\x$ and $\y$, we introduce a regrading of the generating and edge algebras $\overline{B}(l,\Sc)$ and $\overline{B}_{\rho}(l,\Sc)$. This regrading will be used primarily in Corollary~\ref{cor:IBItoTensorProduct}.

Let $[j_1+1, j_1+l_1], \ldots, [j_b+1, j_b+l_b]$ be the generating intervals for $\x$ and $\y$ (see Definition \ref{def:GenInt}), of lengths $l_1, \ldots, l_b$ respectively, ordered so that $j_1< \cdots < j_b$.

\begin{definition}\label{def:OSzRegradedGenAlgs}
If $G=[j_a+1, j_a + l_a]$ is a generating interval for $\x$ and $\y$, consider the dg algebra
\begin{equation*}
\Ib_{[j_a+1, j_a +l_a-1]} \B(n,l_a-1,\Sc \cap G) \Ib_{[j_a+1, j_a +l_a-1]}.
\end{equation*}
If $l_a=1$, note that this algebra is $\F_2$. We have a canonical isomorphism
\begin{equation}\label{eq:OSzpsiG for non edge}
\psi_G \colon \Ib_{[j_a+1, j_a +l_a-1]} \B(n,l_a-1,\Sc \cap G) \Ib_{[j_a+1, j_a +l_a-1]} \to \overline{B}(l_a,\Sc_a),
\end{equation}
by a simple re-indexing of the regions $[0,n]$ (omitting the ones outside $[j_a,j_a+l_a]$), where 
\[
\Sc_a := \{i-j_a \,|\, i \in \Sc \cap G\}.
\]
Redefine the Alexander multi-gradings on $\overline{B}(l_a,\Sc_a)$ by shifting the indices by $j_a$, so that $\tau_{i}, \beta_{i} \mapsto \tau_{i+j_a}, \beta_{i+j_a}$ and the isomorphism preserves the Maslov grading and all Alexander gradings from Section~\ref{sec:OSz gradings}. Similarly, if $G = [n-l_{b+1}+1,n]]$ is a right edge interval for $\x$ and $\y$, there is a canonical isomorphism
\begin{equation}\label{eq:OSzpsiG for right edge}
\psi_G \colon \Ib_{[n-l_{b+1}+1, n]} \B(n,l_{b+1},\Sc \cap [n-l_{b+1}+1, n]) \Ib_{[n-l_{b+1}+1,n]} \to \overline{B}_\rho(l_{b+1},\Sc_{b+1}),
\end{equation}
where
\[
\Sc_{b+1} := \{i-n+l_{b+1} \,|\, i \in \Sc \cap [n-l_{b+1}+1, n]\}.
\]
Modify the Alexander gradings on $\overline{B}_\rho(l_{b+1},\Sc_{b+1})$ so that $\psi_G$ preserves them as above.
\end{definition}

If $G = [[1,l_0]$ is a left edge interval for $\x$ and $\y$, then there is a canonical isomorphism
\begin{equation}\label{eq:OSzpsiG for left edge}
\psi_G \colon \Ib_{[0, l_0-1]} \B(n,l_0,\Sc \cap [1, l_0]) \Ib_{[0,l_0-1]} \to \overline{B}_\lda(l_0,\Sc_0),
\end{equation}
defined as above. By analogy with the previous cases, we set $\Sc_{0} := \Sc \cap [1,l_0]$.
Note that there is no need to redefine the Alexander multi-grading on $\overline{B}_\lda (l_0, \Sc_0)$ in this case, because $\psi_G$ already preserves it. If $G = [[1,n]]$ is a two-faced edge interval for $\x$ and $\y$, then $\overline{B}_{\lda\rho}(n,\Sc) = \Ib_{[0, n]} \B(n,n+1,\Sc) \Ib_{[0,n]}$ by definition.

\begin{definition}\label{def:CrossedLinesAlg}
We define a dg algebra for the crossed lines from $\x$ to $\y$ (see Definition \ref{def:CrossedLine}), denoted by $\Bcl{\x,\y}$, as follows:
\[
\Bcl{\x,\y}:=\frac{\F_2[U_i \,|\, i\in\CL{\x,\y}][C_j \,|\, j\in\CL{\x,\y}\cap\Sc]}{(C_j^2, \,\, \de C_j=U_j)}.
\]
The differential is zero on all terms other than $C_j$ as shown.  We define an Alexander multi-grading on $\Bcl{\x,\y}$ by setting
\[
w_i(1) = \begin{cases}
\frac12 & \text{if $i\in\CL{\x,\y}$}\\
0 & \text{otherwise}\\
\end{cases}
\]
and declaring that multiplication by either $U_j$ or $C_j$ increases $w_i$ by $\delta_{i,j}$ for all $j\in\CL{\x,\y}$ ($j\in\CL{\x,\y}\cap\Sc$ for the variables $C_j$) and $i\in[1,n]$.  We define a Maslov grading on $\Bcl{\x,\y}$ by setting
\[
\m(1) = -|\CL{\x,\y}\cap\Sc|
\]
and declaring that multiplication by $U_i$ (respectively $C_i$) decreases $\m$ by $2$ (respectively $1$) if $i\in\CL{\x,\y}\cap\Sc$.  Multiplication by $U_j$ for $j\in\CL{\x,\y}\setminus\Sc$ has no effect on $\m$.
\end{definition}

\begin{corollary}
\label{cor:IBItoTensorProduct}
Let $\x, \y \in V(n,k)$ be not far. Let $[j_1+1, j_1+l_1], \ldots, [j_b+1, j_b+l_b]$ be the generating intervals, of lengths $l_1, \ldots, l_b$ respectively, ordered so that $j_1< \cdots < j_b$.

There is an isomorphism of chain complexes
\[
\psi \colon \Ib_\x \B(n,k,\Sc) \Ib_\y \stackrel{\sim}{\longrightarrow} \Bcl{\x,\y} \otimes \overline{B}_\circ(l_0,\Sc_0) \otimes \overline{B}(l_1,\Sc_1) \otimes \cdots \otimes \overline{B}(l_{b},\Sc_b) \otimes \overline{B}_\circ(l_{b+1},\Sc_{b+1}),
\]
where $\overline{B}_\circ(l_0,\Sc_0)$ and $\overline{B}_\circ(l_{b+1},\Sc_{b+1})$ are defined as follows: 
\begin{itemize}
\item If $[[1,l_0]$ is a left edge interval for $\x$ and $\y$, then we set $\overline{B}_\circ(l_0,\Sc_0) = \overline{B}_\lda(l_0,\Sc_0)$; otherwise we set $\overline{B}_\circ(l_0,\Sc_0) = \F_2$.
\item If $[n-l_{b+1}+1,n]]$ is a right edge interval for $\x$ and $\y$, then we set $\overline{B}_\circ(l_{b+1},\Sc_{b+1}) = \overline{B}_\rho(l_{b+1},\Sc_{b+1})$; otherwise we set $\overline{B}_\circ(l_{b+1},\Sc_{b+1}) = \F_2$.
\item If $\x = \y = [0,n]$ (i.e. $[[1,n]]$ is a two-faced edge interval for $\x$ and $\y$), then we set the target of $\psi$ to be $\overline{B}_{\lda\rho}(n, \Sc)$.
\end{itemize}
The Alexander and Maslov gradings on the right hand side are as specified in Definitions~\ref{def:OSzRegradedGenAlgs} and \ref{def:CrossedLinesAlg}.
\end{corollary}

Corollary~\ref{cor:IBItoTensorProduct} is useful below when we compute the homology of $\B(n,k,\Sc)$. We will prove an analogous splitting theorem for the strands algebras $\sac nk\Sc$ in \cite{MMW2} and use it for homology computations in parallel fashion.

\begin{proof}
First suppose $\Sc = \varnothing$. By Proposition~\ref{prop:OSzBasis}, $\Ib_{\x} \B(n,k) \Ib_{\y}$ is isomorphic to $\frac{\F_2[U_1,\ldots,U_n]}{(p_G)}$, and if there are no edge intervals we can write
\[
\F_2[U_1,\ldots,U_n] \cong \F_2[U_i \,|\, i \in \CL{\x,\y}] \otimes \bigotimes_{a=1}^b \F_2[U_{j_a+1},\ldots,U_{j_a+l_a}]
\]
by Proposition \ref{prop:QuartumNonDatur} (when there are left or right edge intervals we have extra polynomial factors on the right hand side of the isomorphism). Since each generator $p_G$ of the ideal $(p_G)$ involves only $U_i$ generators with $j_a + 1 \leq i \leq j_a + l_a$, we have
\begin{equation}\label{eq:SplittingCorollaryHelper}
\frac{\F_2[U_1,\ldots,U_n]}{(p_G)} \cong \F_2[U_i \,|\, i \in \CL{\x,\y}] \otimes \bigotimes_{a=1}^b \frac{\F_2[U_{j_a+1},\ldots,U_{j_a+l_a}]}{(U_{j_a+1} \cdots U_{j_a+l_a})},
\end{equation}
where edge intervals again give extra polynomial factors with no relations. Using Proposition~\ref{prop:OSzBasis} a second time, the tensor product on the right side of \eqref{eq:SplittingCorollaryHelper} is isomorphic to the target of $\psi$ as in the statement of the corollary. 

We can thus define $\psi$ to be the isomorphism of \eqref{eq:SplittingCorollaryHelper}, composed on both sides with isomorphisms from Proposition~\ref{prop:OSzBasis}. By the grading shifts of Definition~\ref{def:OSzRegradedGenAlgs}, $\psi$ is an isomorphism of graded modules. When $\Sc \neq \varnothing$, we define $\psi$ by tensoring both sides with an exterior algebra on $\{C_i \,|\, i \in \Sc\}$ (with appropriate gradings and differentials). We get an isomorphism of chain complexes as desired. 
\end{proof}

\begin{remark}\label{rmk:split B with wun}
The various tensor factors in Corollary~\ref{cor:IBItoTensorProduct} can also be given unrefined Alexander gradings so that the splitting isomorphism $\psi$ respects these gradings as well.
\end{remark}

\subsection{Quiver descriptions of the truncated algebras}\label{sec:QuiversForTruncatedAlgs}

Using Proposition~\ref{prop:Generating Intervals}, we can give quiver descriptions of the truncated algebras $\B_r(n,k,\Sc)$, $\B_l(n,k,\Sc)$, and $\B'(n,k,\Sc)$ from Definition~\ref{def:TruncatedOSzAlgs}.

\begin{definition}
We define quivers and relations for the truncated algebras as follows:
\begin{itemize}
\item Let $\Gamma_r(n,k,\Sc)$ be the subgraph of $\Gamma(n,k,\Sc)$ on vertices $\x \in V_r(n,k)$, and let 
\[
\td{\mc R}_r := \bigcup_{\x,\y \in V_r(n,k)} \td{\mc R}_{\x,\y,\Sc}.
\]
Define $\Gamma_l(n,k,\Sc)$ and $\td{\mc R}_l$ similarly.
\item Let $\Gamma'(n,k,\Sc)$ be the subgraph of $\Gamma(n,k,\Sc)$ on vertices $\x \in V'(n,k)$, and let $\td{\mc R}'$ be the union of $\bigcup_{\x,\y \in V'(n,k)} \td{\mc R}_{\x,\y,\Sc}$ with one additional element $U_1 \cdots U_n$ in $\mc R_{\x,\x,\Sc}$ for $\x = [1,n-1]$ (if $k = n-1$).  
\end{itemize}
\end{definition}

\begin{proposition}\label{prop:QuiverDescriptionTruncated}
The dg algebra homomorphism $F$ from $\Path(\Gamma_r(n,k,\Sc))$ to $\B_r(n,k,\Sc)$ obtained by restricting the map $F$ from Proposition~\ref{prop:OSzQuiverEquivDifferential} to 
\[
\Path(\Gamma_r(n,k,\Sc)) \subset \Path(\Gamma(n,k,\Sc))
\]
induces an isomorphism from $\Quiv(\Gamma_r(n,k,\Sc), \td{\mc R}_r)$ to $\B_r(n,k,\Sc)$. Similar statements hold for $\B_l(n,k,\Sc)$ and $\B'(n,k,\Sc)$.
\end{proposition}

\begin{proof}
Since the relations in $\td{\mc R}_r$ hold in $\B_r(n,k,\Sc)$, it suffices to show that $F$ induces a bijection 
\[
F: \Ib_{\x} \Quiv(\Gamma_r(n,k,\Sc), \td{\mc R}_r) \Ib_{\y} \to \Ib_{\x} \B_r(n,k,\Sc) \Ib_{\y}
\]
for all $\x,\y \in V_r(n,k)$. First, we prove this in the case $\Sc = \varnothing$.

For surjectivity, note that by Proposition~\ref{prop:Generating Intervals} (or by definition), $\Ib_{\x} \B_r(n,k) \Ib_{\y}$ is spanned by products of $f_{\x,\y}$ with $U_i$ generators. One can check that $\gamma_{\x,\y}=F^{-1}(f_{\x,\y})$ is a path in $\Gamma_r(n,k)$, and all loops giving rise to $U_i$ generators are in $\Gamma_r(n,k)$, so $F$ is surjective.

For injectivity, first note that Lemma~\ref{lem:BothFactorsAreSegments}, Lemma~\ref{lem:LeftFactorSegmentRightFactorArbitrary}, and Proposition~\ref{prop:BothFactorsArbitrary} also hold when the elements $\gamma_{\x,\y}$ in these proofs are viewed as elements of $\Quiv(\Gamma_r(n,k),\td{\mc{R}}_r)$ rather than $\Quiv(\Gamma_U(n,k), \mc{R}_U)$. In fact, the case of $\Quiv(\Gamma_r(n,k),\td{\mc{R}}_r)$, like $\Quiv(\Gamma(n,k),\td{\mc{R}})$, is a bit simpler because $R$-segments and $L$-segments with length greater than $1$ are zero. The only thing to check is that all relations used in the proofs of Lemma~\ref{lem:BothFactorsAreSegments}, Lemma~\ref{lem:LeftFactorSegmentRightFactorArbitrary}, and Proposition~\ref{prop:BothFactorsArbitrary} are in $\td{\mc R}_r \subset \td{\mc R}$, and this follows from re-examining the proofs.

Since any path $\gamma$ in $\Gamma_r(n,k)$ can be written as some $U_i$ loops times a product of edges $\gamma_{\x,\y}$, we see from these results that $\gamma$ is equivalent modulo $\td{\mc R}_r$ to $\gamma_{\x,\y}$ times some $U_i$ loops. It follows that as $\F_2[U_1,\ldots,U_n]$-modules, $\Ib_{\x} \Quiv(\Gamma_r(n,k), \td{\mc R}_r) \Ib_{\y}$ is isomorphic to $\frac{\F_2[U_1,\ldots,U_n]}{\mc J}$ where $\mc J$ is an ideal in $\F_2[U_1,\ldots,U_n]$.

Since $F$ is well-defined, $\mc J$ is contained in the ideal generated by the monomials $p_G$ for $G$ a generating interval between $\x$ and $\y$. To show that $F$ is injective, it suffices to show that each of these monomials is in $\mc J$. Indeed, writing $p_G = U_{j+1} \cdots U_{j+l}$, we consider several cases.

If $\x\cap\{j,j+l\} = \varnothing$ (respectively $\y\cap\{j,j+l\} = \varnothing$), then we can factor $p_G$ as the path $R_{j+l} \cdots R_{j+2} U_{j+1} L_{j+2} \cdots L_{j+l}$ in $\Gamma_r(n,k)$ based at $\x$ (respectively at $\y$).  Note that all of the edges used are indeed in $\Gamma_r(n,k)$ even in the extreme case that $j=0$.  Then since $j\notin\x$ (respectively $j\notin\y$), the edge labeled $U_{j+1}$ in this expression is zero by the $U$ vanishing relations.

Since $G$ is a generating interval, the only cases that remain are the cases where 
\[
|\x\cap\{j,j+l\}|=|\y\cap\{j,j+l\}|=1,
\]
but $\x\cap\{j,j+l\}\neq\y\cap\{j,j+l\}$, which we analyze separately below.
\begin{itemize}
\item If $\x\cap\{j,j+l\}=\{j+l\}$ and $\y\cap\{j,j+l\}=\{j\}$, then $v_j(\x,\y) = v_{j+l+1}(\x,\y) = 1$ and the construction of Definition \ref{def:Recursive} ensures that $\gamma_{\x,\y}$ contains a pair of consecutive edges labeled $R_{j+l+1} R_j$. Let $\x'$ be the vertex common to these two edges; we have $\x' \cap \{j,j+l\} = \varnothing$, so we can factorize $p_G$ as a path based at $\x'$ and get zero.

\item If $\x\cap\{j,j+l\}=\{j\}$ and $\y\cap\{j,j+l\}=\{j+l\}$, the argument is similar with $\gamma_{\x,\y}$ containing a pair of consecutive edges labeled $L_j L_{j+l+1}$. Let $\x'$ be the vertex common to these edges; we again have $\x' \cap \{j,j+l\} = \varnothing$, so we can factorize $p_G$ as a path based at $\x'$ and get zero.
\end{itemize}

It follows that
\[
F \colon \Quiv(\Gamma_r(n,k), \td{\mc R}_r) \to \B_r(n,k)
\]
is a dg algebra isomorphism as claimed. By adjoining the $C_i$ generators on both sides, we get the statement for general $\Sc$.

The case of $\B_l(n,k,\Sc)$ is analogous; one uses the factorization
\[
L_{j+1} \cdots L_{j+l-1} U_{j+l} R_{j+l-1} \cdots R_{j+1},
\]
the edges of which are included in $\Gamma_l(n,k)$.

For $\B'(n,k,\Sc)$, the argument needs a minor modification: if $k = n-1$, so $\x = \y = [1,n-1]$, the unique generating interval $G$ between $\x$ and $\y$ has $p_G = U_1 \cdots U_n$, and none of the $U_i$ generators can be factored into $R_i$ and $L_i$ generators in $\Ib_{\x} \Quiv(\Gamma'(n,k), \td{\mc R}') \Ib_{\y}$. However, $U_1 \cdots U_n$ was explicitly added as an element of $\td{\mc R}'$ (and thus $\mc J$), so it still follows that each monomial $p_G$ is in $\mc J$. Thus, $F$ is a dg algebra isomorphism in this case as well.
\end{proof}

\begin{remark}
The proof of Proposition~\ref{prop:QuiverDescriptionTruncated} also gives us an alternate and shorter proof of Proposition~\ref{prop:OSzQuiverEquivDifferential}, relying fundamentally on Proposition~\ref{prop:OSzBasis} (\cite[Proposition 3.7]{OSzNew}).

Our first proof of Proposition~\ref{prop:OSzQuiverEquivDifferential}, though, has the advantage that it comes from a quiver description of the algebra $\B_0(n,k)$, which may be of independent interest and to which \cite[Proposition 3.7]{OSzNew} does not apply.
\end{remark}

\subsection{Symmetries}\label{sec:OSzSymmetries}
In \cite[Section 3.6]{OSzNew}, Ozsv{\'a}th--Szab{\'o} exhibit two symmetries of the algebras $\B(n,k,\Sc)$, taking the form of algebra isomorphisms
\[
\rho: \B(n,k,\Sc) \to \B(n,k,\rho(\Sc))
\]
and
\[
o: \B(n,k,\Sc) \to \B(n,k,\Sc)^{\op},
\]
where
\[
\rho(\Sc) = \set{n+1-i \, \middle| \, i \in \Sc}.
\]
Ozsv\'ath and Szab\'o refer to the first symmetry as $\mc R$ rather than $\rho$. We use the latter to avoid confusion with the relation ideals in the quiver description of $\B(n,k,\Sc)$. In the language of Definition~\ref{def:OSzStyleDef}, we can describe these symmetries as follows; first, define $\rho: V(n,k) \to V(n,k)$ by 
\[
\rho(\x) = \{n-a \,|\, a \in \x\}.
\]
\begin{definition}[Section 3.6 of \cite{OSzNew}]
The isomorphism $\rho: \B(n,k) \to \B(n,k)$ is induced from the isomorphism $\rho: \B_0(n,k) \to \B_0(n,k)$ sending $U_1^{r_1} \cdots U_n^{r_n} f_{\x,\y}$ to $U_1^{r_n} \cdots U_n^{r_1} f_{\rho(\x),\rho(\y)}$. The isomorphism $\rho: \B(n,k,\Sc) \to \B(n,k,\rho(\Sc))$ is induced from $\rho: \B(n,k) \to \B(n,k)$ by sending $C_i$ to $C_{n+1-i}$.

We can view $\rho$ as an isomorphism of $\F_2[U_1,\ldots,U_n]^{V(n,k)}$-algebras if we modify the left and right actions of $\F_2[U_1,\ldots,U_n]^{V(n,k)}$ on $\B(n,k,\rho(\Sc))$ by precomposing them with the endomorphism of $\F_2[U_1,\ldots,U_n]^{V(n,k)}$ sending $U_i$ to $U_{n+1-i}$ and $\x$ to $\rho(\x)$. One can check that $\rho$ is an involution, i.e. $\rho^2 = \id$; this makes sense because $\rho$ is defined for all $\Sc$. 
\end{definition}

\begin{definition}[Section 3.6 of \cite{OSzNew}]
The involution $o: \B(n,k,\Sc) \to \B(n,k,\Sc)^{\op}$ sends 
\[
U_1^{r_1} \cdots U_n^{r_n} f_{\x,\y} \mapsto U_1^{r_1} \cdots U_n^{r_n} f_{\y,\x},
\]
and it sends $C_i$ to $C_i$. Unlike with $\rho$, we can view $o$ as an involution of $\F_2[U_1,\ldots,U_n]^{V(n,k)}$-algebras without modifying the actions on either side. One can check that this involution $o$ sends $R_i$ to $L_i$, $L_i$ to $R_i$, $U_i$ to $U_i$, and $C_i$ to $C_i$ where $R_i$, $L_i$, $U_i$, and $C_i$ are defined as in Definition~\ref{def:OSzStyleRLU}, so that our definition of $o$ agrees with Ozsv{\'a}th--Szab{\'o}'s.
\end{definition}

In fact, we can see both symmetries $\rho$ and $o$ from the perspective of $\Quiv(\Gamma(n,k,\Sc), \td{\mc R}_{\Sc})$.
\begin{definition}\label{def:RhoOnDirectedGraphs}
Define an involution of directed graphs $\rho: \Gamma(n,k,\Sc) \to \Gamma(n,k,\rho(\Sc))$ by sending an edge $\gamma$ from $\x$ to $\y$ in $V(n,k)$ labeled $R_i$, $L_i$, $U_i$, or $C_i$ to the unique edge from $\rho(\x)$ to $\rho(\y)$ labeled $L_{n+1-i}$, $R_{n+1-i}$, $U_{n+1-i}$, or $C_{n+1-i}$ respectively.
\end{definition}

The involution $\rho$ induces an involution of path algebras 
\[
\rho: \Path(\Gamma(n,k,\Sc)) \xrightarrow{\cong} \Path(\Gamma(n,k,\rho(\Sc))),
\]
at least as algebras over $\F_2$. We can view $\rho$ as an involution of $\F_2[U_1,\ldots,U_n]^{V(n,k)}$-algebras if we precompose the usual left and right actions of $\F_2[U_1,\ldots,U_n]^{V(n,k)}$ on $\Path(\Gamma(n,k,\rho(\Sc)))$ with the involution of $\F_2[U_1,\ldots,U_n]^{V(n,k)}$ sending $\x$ to $\rho(\x)$ and sending $U_i$ to $U_{n+1-i}$.

\begin{proposition}\label{prop:RhoPreservesRelations}
The involution $\rho: \Path(\Gamma(n,k,\Sc)) \xrightarrow{\cong} \Path(\Gamma(n,k,\rho(\Sc)))$ sends the relation set $\td{\mc R}_{\Sc}$ to the relation set $\td{\mc R}_{\rho(\Sc)}$.
\end{proposition}

\begin{proof}
One can check that the image under $\rho$ of each relation in $\td{\mc R}_{\Sc}$ listed in Definitions \ref{def:B_0 Quiver Algebra}, \ref{def:B Quiver Algebra}, and \ref{def:BnksQuiverDescription} is also listed in one of these definitions for $\td{\mc R}_{\rho(\Sc)}$.
\end{proof}

It follows from Proposition~\ref{prop:RhoPreservesRelations} that $\rho$ induces an involution
\[
\rho: \Quiv(\Gamma(n,k,\Sc), \td{\mc R}_{\Sc}) \xrightarrow{\cong} \Quiv(\Gamma(n,k,\rho(\Sc)), \td{\mc R}_{\rho(\Sc)})
\]
of $\F_2$-algebras, or of $\F_2[U_1,\ldots,U_n]^{V(n,k)}$-algebras if we modify the left and right actions on the right-hand side as discussed below Definition~\ref{def:RhoOnDirectedGraphs}. We have $\rho \circ \de = \de \circ \rho$ (it suffices to check that $\rho(\de(C_i)) = \de(\rho(C_i))$), so $\rho$ is an involution of differential algebras.

Finally, Definition~\ref{def:OSzMaslovDegrees} implies that $\rho$ preserves the Maslov grading. If we postcompose the unrefined Alexander multi-degree function on the right-hand side with the involution of $\Z^{2n}$ sending $\tau_i$ to $\beta_{n+1-i}$ and sending $\beta_i$ to $\tau_{n+1-i}$, then $\rho$ preserves the unrefined Alexander multi-grading as well. Similar statements hold for the refined Alexander multi-grading and the single Alexander grading, so that in any case we may view $\rho$ as an involution of dg algebras over $\F_2[U_1,\ldots,U_n]^{V(n,k)}$.

\begin{proposition}\label{prop:RSymmetryQuiverOSz}
The isomorphisms $F$ and $G$ between $\Quiv(\Gamma(n,k,\Sc), \td{\mc R}_{\Sc})$ and $\B(n,k,\Sc)$ from Proposition~\ref{prop:OSzQuiverEquivDifferential} intertwine the symmetry $\rho$ of $\Quiv(\Gamma(n,k,\Sc), \td{\mc R}_{\Sc})$ with Ozsv{\'a}th--Szab{\'o}'s first symmetry of $\B(n,k,\Sc)$, which we also denote by $\rho$ as discussed above.
\end{proposition}

\begin{proof}
We only need to prove the result for $F$, since $G = F^{-1}$. If $\gamma$ is an edge in $\Gamma(n,k,\Sc)$ from $\x$ to $\y$ with label $R_i$, then $F(\gamma) = f_{\x,\y}$, so $\rho(F(\gamma)) = f_{\rho(\x),\rho(\y)}$. The edge $\rho(\gamma)$ in $\Gamma(n,k,\rho(\Sc))$ from $\rho(\x)$ to $\rho(\y)$ is the unique such edge (and its label is $L_{n+1-i}$), so $F(\rho(\gamma)) = f_{\rho(\x),\rho(\y)}$ and we have $\rho(F(\gamma)) = F(\rho(\gamma))$. If $\gamma$ has label $L_i$ rather than $R_i$, the proof is similar.

Let $\gamma$ be an edge in $\Gamma(n,k,\Sc)$ from $\x$ to $\x$ with label $U_i$. We have $F(\gamma) = U_i f_{\x,\x}$, so 
\[
\rho(F(\gamma)) = U_{n+1-i} f_{\rho(\x),\rho(\x)}.
\]
The edge $\rho(\gamma)$ from $\rho(\x)$ to $\rho(\x)$ has label $U_{n+1-i}$, so we have 
\[
F(\rho(\gamma)) = U_{n+1-i} f_{\rho(\x),\rho(\x)}
\]
and thus $\rho(F(\gamma)) = F(\rho(\gamma))$.

Finally, $F$ sends $C_i$ to $C_i$, and $\rho(C_i) = C_{n+1-i}$ on both sides. Since $F$ and $\rho$ are multiplicative, it follows that $\rho \circ F = F \circ \rho$.
\end{proof}

Note that $\x \in V_r(n,k)$ if and only if $\rho(\x) \in V_l(n,k)$; similarly, $\x \in V'(n,k)$ if and only if $\rho(x) \in V'(n,k)$. Thus, the following proposition is immediate.

\begin{proposition} The symmetry $\rho: \B(n,k,\Sc) \to \B(n,k,\rho(\Sc))$ restricts to isomorphisms 
\begin{align*}
&\rho: \B_r(n,k,\Sc) \xrightarrow{\cong} \B_l(n,k,\rho(\Sc)), \\
&\rho: \B_l(n,k,\Sc) \xrightarrow{\cong} \B_r(n,k,\rho(\Sc)), \textrm{ and}\\
&\rho: \B'(n,k,\Sc) \xrightarrow{\cong} \B'(n,k,\rho(\Sc)).
\end{align*}
\end{proposition}

Edges of $\Gamma_r(n,k,\Sc)$ are mapped to edges of $\Gamma_l(n,k,\rho(\Sc))$ by $\rho$ and vice versa; elements of $\td{\mc R}_r$ are mapped to elements of $\td{\mc R}_l$ and vice versa. Thus, the symmetry $\rho$ is apparent in our quiver descriptions of $\B_r(n,k,\Sc)$ and $\B_l(n,k,\rho(\Sc))$; the same is true for $\B'(n,k,\Sc)$.

Now we will consider Ozsv{\'a}th--Szab{\'o}'s second symmetry $o$, relating $\B(n,k,\Sc)$ and its opposite algebra. Note that the opposite $\Gamma^{\op}$ of a directed graph $\Gamma$ can be defined by reversing the orientation of all edges in $\Gamma$. Each edge of $\Gamma^{\op}$ keeps the same label as it had in $\Gamma$. We have identifications $(\F_2\Gamma)^{\op} \cong \F_2(\Gamma^{\op})$ of path categories, and similarly for path algebras.
\begin{definition}\label{def:OOnDirectedGraphs}
Define an automorphism of directed graphs $o: \Gamma(n,k,\Sc) \to \Gamma(n,k,\Sc)^{\op}$ as follows:
\begin{itemize}
\item For a vertex $\x \in V(n,k)$ of $\Gamma(n,k,\Sc)$, define $o(\x) = \x$.
\item For an edge $\gamma$ from $\x$ to $\y$ in $\Gamma(n,k,\Sc)$ labeled $R_i$, $L_i$, $U_i$, or $C_i$, define $o(\gamma)$ to be the unique edge from $\x$ to $\y$ in $\Gamma(n,k,\Sc)^{\op}$ labeled $L_i$, $R_i$, $U_i$, or $C_i$ respectively.
\end{itemize}
As with $\rho$, one can check that $o$ is an involution, so $o$ induces an involution of path categories $o: \F_2\Gamma(n,k,\Sc) \xrightarrow{\cong} \F_2\Gamma(n,k,\Sc)^{\op}$ and thus an involution of path algebras $o: \Path(\Gamma(n,k,\Sc)) \xrightarrow{\cong} \Path(\Gamma(n,k,\Sc)^{\op})$. Unlike with $\rho$, the involution $o$ is $\F_2[U_1,\ldots,U_n]^{V(n,k)}$-linear without modification of the actions on either side.
\end{definition}

Note that we may view $\td{\mc R}_{\Sc}$ as a set of elements in $\Path(\Gamma(n,k,\Sc))^{\op}$, and the quotient of the opposite algebra by the ideal generated by this set can be identified with the opposite algebra of the original quotient. 

\begin{proposition}\label{prop:oPreservesRelations}
The involution $o: \Path(\Gamma(n,k,\Sc)) \xrightarrow{\cong} \Path(\Gamma(n,k,\Sc))^{\op}$ sends the relation set $\td{\mc R}_{\Sc}$ to the relation set $\td{\mc R}_{\Sc}$.
\end{proposition}

\begin{proof}
As with $\rho$, one can check that the image under $o$ of each relation in $\td{\mc R}_{\Sc}$ is also a relation in $\td{\mc R}_{\Sc}$. For example, $o(R_i R_{i+1})$ is $L_i L_{i+1}$ where the product is taken in the opposite algebra; this is the element we would more typically call $L_{i+1} L_i$, and it is an element of $\td{\mc R}_{\Sc}$.
\end{proof}

Thus, $o$ induces an involution
\[
o: \Quiv(\Gamma(n,k,\Sc), \td{\mc R}_{\Sc}) \xrightarrow{\cong} \Quiv(\Gamma(n,k,\Sc), \td{\mc R}_{\Sc})^{\op}
\]
of $\F_2[U_1,\ldots,U_n]^{V(n,k)}$-algebras. We have $o \circ \de = \de \circ o$, so $o$ is an involution of differential algebras, and $o$ preserves the Maslov grading. If we postcompose the unrefined Alexander multi-degree function on the right-hand side with the involution of $\Z^{2n}$ sending $\tau_i$ to $\beta_i$ and sending $\beta_i$ to $\tau_i$, then $o$ preserves the unrefined Alexander multi-grading, so $o$ is an involution of dg algebras over $\F_2[U_1,\ldots,U_n]^{V(n,k)}$. Similar statements hold for the refined and single Alexander gradings.

\begin{proposition}
The isomorphisms $F$ and $G$ between $\Quiv(\Gamma(n,k,\Sc), \td{\mc R}_{\Sc})$ and $\B(n,k,\Sc)$ from Proposition~\ref{prop:OSzQuiverEquivDifferential} intertwine the symmetry $o$ of $\Quiv(\Gamma(n,k,\Sc), \td{\mc R}_{\Sc})$ with Ozsv{\'a}th--Szab{\'o}'s symmetry $o$ of $\B(n,k,\Sc)$.
\end{proposition}

\begin{proof}
The proof is similar to that of Proposition~\ref{prop:RSymmetryQuiverOSz}. Again, we only need to prove the result for $F$. If $\gamma$ is an edge in $\Gamma(n,k,\Sc)$ from $\x$ to $\y$ with label $R_i$, then $F(\gamma) = f_{\x,\y}$, so $o(F(\gamma)) = f_{\y,\x}$. The edge $o(\gamma)$ in $\Gamma(n,k,\Sc)^{\op}$ from $\x$ to $\y$ is the unique such edge (and its label is $L_i$), so $F(o(\gamma)) = f_{\y,\x}$ and we have $o(F(\gamma)) = F(o(\gamma))$. If $\gamma$ has label $L_i$ rather than $R_i$, the proof is similar. For an edge $\gamma$ in $\Gamma(n,k,\Sc)$ with label $U_i$ or $C_i$, we have $o(\gamma) = \gamma$ and $o(F(\gamma)) = F(\gamma)$. Since $o$ and $F$ are multiplicative, it follows that $o \circ F = F \circ o$.
\end{proof}

Like with $\rho$, the symmetry $o$ can be restricted to the truncated algebras. 

\begin{proposition}
The symmetry $o: \B(n,k,\Sc) \to \B(n,k,\Sc)^{\op}$ restricts to isomorphisms 
\begin{align*}
&o: \B_r(n,k,\Sc) \xrightarrow{\cong} \B_r(n,k,\Sc)^{\op}, \\
&o: \B_l(n,k,\Sc) \xrightarrow{\cong} \B_l(n,k,\Sc)^{\op}, \textrm{ and}\\
&o: \B'(n,k,\Sc) \xrightarrow{\cong} \B'(n,k,\Sc)^{\op}.
\end{align*}
\end{proposition}

Again, the symmetry $o$ is apparent in our quiver descriptions of the truncated algebras.
Note that $\rho \circ o = o \circ \rho$, properly interpreted.
In \cite{MMW2}, we will define analogues of the symmetries $\rho$ and $o$ for the strands algebras defined there and interpret them geometrically in terms of rotations of strands pictures.

\section{Homology of Ozsv\'ath-Szab\'o's algebras}
\label{sec:HomologyAndFormality}

When $\Sc \neq \varnothing$, the algebra $\B(n,k,\Sc)$ has a differential; in Section \ref{sec:Homology} we compute its homology, and in Section \ref{sec:Formality} we discuss formality and higher products.

\subsection{Homology computations}
\label{sec:Homology}

Let $\x,\y \in V(n,k)$. By Corollary~\ref{cor:IBItoTensorProduct}, $\Ib_{\x} \B(n,k,\Sc) \Ib_{\y}$ decomposes as a tensor product over $\F_2$ of the crossed line complex $\Bcl{\x,\y}$ with chain complexes $\overline{B}(l_a,\Sc_a)$ for each generating interval $[j_a+1,j_a+l_a]$ as well as $\overline{B}_{\lda}(l_0,\Sc_0)$, $\overline{B}_{\rho}(l_{b+1},\Sc_{b+1})$, and $\overline{B}_{\lda\rho}(n,\Sc)$ for various types of edge intervals.

We will compute the homology of the factors appearing in this tensor product; we start with $\Bcl{\x,\y}$.

\begin{lemma}\label{lem:OSzCLHomology}
Fix $0\leq k \leq n$, $\x,\y\in V(n,k)$, and $\Sc\subset[1,n]$. A basis for the homology of the crossed line complex $\Bcl{\x,\y}$ is given by monomials in the variables $U_i\in\CL{\x,\y}\setminus\Sc$.
\end{lemma}
\begin{proof}
$\Bcl{\x,\y}$ can be decomposed as a tensor product having a factor $\F_2[U_i]$ for each $i\in\CL{\x,\y}\setminus\Sc$ and a factor $\frac{\F_2[U_j,C_j]}{(C_j^2=0, \,\, \de C_j=U_j)}$ for each $j\in\CL{\x,\y}\cap\Sc$.  Both types have easily computable homology; the K{\"u}nneth theorem allows us to combine them to compute $H_*(\Bcl{\x,\y})$ and see that the basis elements are as described. 
\end{proof}

Next we consider the complexes for edge intervals.
\begin{lemma}\label{lem:OSzEdgeIntHomology}
For $n \geq 0$ and $\Sc \subset [1,n]$, a basis for the homology of any of the edge-interval complexes $\overline{B}_{\lda}(n,\Sc)$, $\overline{B}_{\rho}(n,\Sc)$, or $\overline{B}_{\lda\rho}(n,\Sc)$ is given by elements $[\phi(p)]$ (in the notation of Proposition~\ref{prop:OSzBasis}) where $p$ is a monomial in the variables $U_i$ for $i \notin \Sc$.
\end{lemma}

\begin{proof}
Similarly to Lemma~\ref{lem:OSzCLHomology}, $\overline{B}_{\lda}(n,\Sc)$ is isomorphic via the map $\phi$ of Proposition~\ref{prop:OSzBasis} to a tensor product of complexes $\F_2[U_i]$ for each $i\in[1,n]\setminus\Sc$ and $\frac{\F_2[U_j,C_j]}{(C_j^2=0, \,\, \de C_j=U_j)}$ for each $j\in\Sc$. The same argument works for $\overline{B}_{\rho}(n,\Sc)$ and $\overline{B}_{\lda\rho}(n,\Sc)$.
\end{proof}

We now consider the case of generating intervals, which are a bit more complicated.

\begin{lemma}\label{lem:OSzGenIntHomology}
For $n \geq 0$ and $\Sc = \{i_1,\ldots,i_l\} \subset [1,n]$, a basis for the homology of $\overline{B}(n,\Sc)$ is given by:
\begin{itemize}
\item elements $[\phi(p)]$ where $p$ is a monomial in the variables $U_i$ for $i \notin \Sc$, together with
\item elements $[C_{i_1} \phi(p)]$ where $p$ is $U_1 \cdots \widehat{U_{i_1}} \cdots U_{n}$ times a monomial in the variables $U_i$ for $i \notin \Sc$.
\end{itemize}
For the second type of basis element, we could use any $C_{i_j}$ for $i_j \in \Sc$ in place of $i_1$, removing $U_{i_j}$ instead of $U_{i_1}$ from the monomial $p$; the resulting elements represent the same homology class.
\end{lemma}

\begin{proof}
Let $l = |\Sc|$. We will induct on $l$; recall that $\overline{B}(n,\varnothing) \cong \F_2[U_1,\ldots,U_n]/(U_1\cdots U_n)$.

For $l = 1$, so $\Sc = \{i\}$ for some $i$, the complex $\overline{B}(n,\Sc)$ is isomorphic to the mapping cone of $U_i: \overline{B}(n,\varnothing) \to \overline{B}(n,\varnothing)$. Since $\overline{B}(n, \varnothing)$ is a complex with vanishing differential, we have
\[H_*(\overline{B}(n,\Sc)) \cong (\overline{B}(n,\varnothing)/\im(U_i)) \oplus \ker(U_i).\]
 A basis for $\overline{B}(n,\varnothing)/\im(U_i)$ is given by monomials in the variables $U_j$ for $j \neq i$. A basis for $\ker(U_i)$ is given by monomials that are divisible by $U_j$ for all $j \neq i$, but not divisible by $U_i$ (otherwise they would be zero). The identification of the homology of the mapping cone with $H_*(\overline{B}(n,\Sc))$ sends these monomials to the basis elements stated in the lemma.

Now assume $l > 1$; write $\Sc = \{i_1,\ldots,i_l\}$ and $\Sc' = \Sc \setminus \{i_l\}$. The complex $\overline{B}(n,\Sc)$ is isomorphic to the mapping cone of $U_{i_l}: \overline{B}(n,\Sc') \to \overline{B}(n,\Sc')$. Thus, we have a long exact sequence in homology
\[
\cdots \to H_*(\overline{B}(n,\Sc')) \to H_*(\overline{B}(n,\Sc)) \to H_*(\overline{B}(n,\Sc')) \to \cdots
\]
with connecting map the map $[U_{i_l}]$ induced by $U_{i_l}$ on homology.

From this long exact sequence, we can extract a short exact sequence
\[
0 \to H_*(\overline{B}(n,\Sc')) /\im([U_{i_l}]) \to H_*(\overline{B}(n,\Sc)) \to \ker([U_{i_l}]) \to 0.
\]
Multiplication by $[U_{i_l}]$ sends a basis vector of $H_*(\overline{B}(n, \Sc'))$ to another such basis vector. Moreover, it is straightforward to check that no two basis vectors are sent to the same one. Thus, $\ker([U_{i_l}])=0$ and
\[
H_*(\overline{B}(n,\Sc)) \cong H_*(\overline{B}(n,\Sc')) /\im([U_{i_l}]).
\]
A basis for this quotient is given by basis elements for $H_*(\overline{B}(n,\Sc'))$ that are not equal to another basis element times $U_{i_l}$; by induction, these are the basis elements stated in the lemma.
\end{proof}

We can assemble the results above to compute the homology of $\B(n,k,\Sc)$ in general.

\begin{theorem}\label{thm:OSzHomology}
Let $\x,\y \in V(n,k)$ be not far. Let $[j_1+1, j_1+l_1], \ldots, [j_b+1, j_b+l_b]$ be the generating intervals from $\x$ to $\y$. For $1 \leq a \leq b$, let $i_a$ be an element of $[j_a+1,j_a+l_a] \cap \Sc$, if such an element exists.

For a generating interval $G = [j_a+1,j_a+l_a]$, write $p_a$ for $p_G$ (as defined above in Section~\ref{sec:GenInt}). Abusing notation slightly, a basis for $\Ib_{\x} H_*(\B(n,k,\Sc)) \Ib_{\y}$ is given by the elements
\[
\phi\left(p\prod_{a=1}^b \bigg(\frac{C_{i_a} p_a}{U_{i_a}}\bigg)^{\varepsilon_a}\right)
\]
where $\varepsilon_a \in \{0,1\}$ is zero if $\Sc \cap [j_a+1,j_a+l_a] = \varnothing$ and $p$ is a monomial in the variables $U_i$ for $i \notin \Sc$, not divisible by $p_G$ for any generating interval $G$ (this condition is only relevant for generating intervals disjoint from $\Sc$).
\end{theorem}

Note that, as discussed in Remark \ref{rem:HIAJ=IHAJ} in the appendix, it does not matter whether we write $\Ib_{\x} H_*(\B(n,k,\Sc)) \Ib_{\y}$ or $H_*(\Ib_{\x} \B(n,k,\Sc) \Ib_{\y})$.

\begin{proof}
By the K{\"u}nneth formula for the homology of a tensor product of chain complexes over $\F_2$, the isomorphism $\psi$ of Corollary~\ref{cor:IBItoTensorProduct} induces an isomorphism
\begin{align*}
\psi \colon \Ib_\x H_*(\B(n,k,\Sc)) \Ib_\y \stackrel{\sim}{\longrightarrow} \F_2[U_i \,|\, i\in\CL{\x,\y}\setminus\Sc] &\otimes H_*(\overline{B}_\circ(l_0,\Sc_0)) \otimes H_*(\overline{B}(l_1,\Sc_1)) \otimes \cdots \\
&\otimes H_*(\overline{B}(l_{b},\Sc_b)) \otimes H_*(\overline{B}_\circ(l_{b+1},\Sc_{b+1})),
\end{align*}
where $\CL{\x,\y}$ is the set of crossed lines from $\x$ to $\y$ as usual.  We can get a basis for $\Ib_\x H_*(\B(n,k,\Sc)) \Ib_\y$ by taking the product of bases for each tensor factor. Using the bases of Lemmas \ref{lem:OSzCLHomology}, \ref{lem:OSzEdgeIntHomology}, and \ref{lem:OSzGenIntHomology}, we get the basis stated in the theorem.
\end{proof}

Each summand $\Ib_{\x} \B_r(n,k,\Sc) \Ib_{\y}$ is a summand of $\B(n,k,\Sc)$ (there are just fewer summands in $\B_r(n,k,\Sc)$), and the same is true for the other truncated algebras. Thus, the homology of the truncated algebras follows from Theorem~\ref{thm:OSzHomology}.

\subsection{Formality and Massey products}\label{sec:Formality}
Let $\A$ be a dg algebra. By homological perturbation theory (see e.g. \cite[Corollary 2.1.18]{LOTBimod} and the references therein), $\A$ is $\A_{\infty}$ homotopy equivalent to $H_*(\A)$ where the multiplication on $H_*(\A)$ is supplemented by certain higher $\A_{\infty}$ actions. If one has such an equivalence when taking the extra $\A_{\infty}$ actions on $H_*(\A)$ to be zero, then $\A$ is called \emph{formal}. Instead of an $\A_{\infty}$ homotopy equivalence, one can ask for a zig-zag pattern of dg quasi-isomorphisms connecting $\A$ and $H_*(\A)$. 

\begin{remark}
The equivalence of these two notions of formality is a standard result at least for $\ring$-algebras defined as rings $\A$ equipped with ring homomorphisms from $\ring$ to $Z(\A)$; see \cite[Corollary 2.9]{Lunts}. Presumably the equivalence also holds in our setting, although we do not need this fact for our results. By \cite[Theorem 9.2.0.4, (b)$\Rightarrow$(a)]{KLH}, a dg quasi-isomorphism is an $\A_{\infty}$ homotopy equivalence under our definitions; the same therefore holds for a more general zig-zag of dg quasi-isomorphisms. When proving algebras are formal, we will always exhibit zig-zags of dg quasi-isomorphisms. When obstructing formality, we will always obstruct $\A_{\infty}$ formality.
\end{remark}

The higher $\A_{\infty}$ actions induced on $H_*(\A)$ are not canonical in general. However, certain higher actions on $H_*(\A)$ (``Massey products'') are canonical and thus obstruct formality. We will use the notion of ``Massey admissible sequences'' described by Lipshitz--Ozsv{\'a}th--Thurston in \cite[Definition 2.1.21]{LOTBimod} to identify Massey products on $H_*(\B(n,k,\Sc))$ for some choices of $n$, $k$, and $\Sc$, and we will show that for the remaining choices $\B(n,k,\Sc)$ is formal. We will consider formality for the truncated algebras $\B_r(n,k,\Sc)$, $\B_l(n,k,\Sc)$, and $\B'(n,k,\Sc)$ in Section~\ref{sec:TruncatedFormality}.

\begin{remark}
Throughout this section, we will be using the quiver description for $\B(n,k,\Sc)$ via the graph $\Gamma(n,k,\Sc)$ freely without reference to the isomorphisms $F$ and $G$ from Sections~\ref{sec:OSz} and \ref{sec:OSzB}.
\end{remark} 

\begin{definition}[see Definition 2.1.21 of \cite{LOTBimod}]\label{def:MasseyAdm}
Let $\A$ be a dg algebra (with a homological grading by $\Z$, and possibly with an additional intrinsic grading) over $\F_2^V$ for $V$ finite and write $\overline{\mu}$ for a set of $\A_{\infty}$ operations on $H_*(\A)$ such that $(H_*(\A),\overline{\mu})$ is $\A_{\infty}$ homotopy equivalent to $\A$. Let $(\alpha_1,\ldots,\alpha_m)$ be a finite sequence of elements of $H_*(\A)$ coming from composable morphisms in $\Cat_{H_*(\A)}$.  The sequence is called \emph{Massey admissible} if the following conditions hold for all $1 \leq i < j \leq m$ with $(i,j) \neq (1,m)$:
\begin{itemize}
\item The higher product $\overline{\mu}_{j-i+1}(\alpha_i, \ldots, \alpha_j)$ is zero.
\item If $\x,\y \in V$ denote the left and right idempotents of $\alpha_i$ and $\alpha_j$ respectively, the summand $\Ib_{\x} H_*(\A) \Ib_{\y}$ of the homology algebra $H_*(\A)$ is zero in degree $j-i+\deg(\alpha_i) + \cdots + \deg(\alpha_j)$.
\end{itemize} 
Integers added to degrees are taken to modify the homological degree while leaving the intrinsic degree unchanged. Note that the Massey admissibility condition for sequences of length three does not depend on the choice of $\overline{\mu}$.
\end{definition}

In our case, the homological grading is the Maslov grading of Definition \ref{def:OSzMaslovDegrees}, while the refined Alexander multi-grading of Definition \ref{def:RefinedAlexOSz} will be treated as the additional intrinsic grading (see Remark~\ref{rem:FormalityWithOtherGradings} below for a brief discussion of the other grading possibilities).

\begin{remark}
Definition~\ref{def:MasseyAdm} is a slight modification of \cite[Definition 2.1.21]{LOTBimod} because it takes the idempotent structure into account; one can check that \cite[Lemma 2.1.22]{LOTBimod} still holds in this setting.
\end{remark}

If $(\alpha_1,\ldots,\alpha_m)$ is a Massey admissible sequence, then $\overline{\mu}_m(\alpha_1,\ldots,\alpha_m)$ can be computed as in \cite[Lemma 2.1.22]{LOTBimod}. In many cases the result will be nonzero, implying that $\A$ cannot be formal.

We will not attempt to characterize all Massey admissible sequences in $H_*(\B(n,k,\Sc))$ or to determine the higher multiplication completely. We will content ourselves with determining the cases in which $\B(n,k,\Sc)$ is formal. Certain Massey products of length three commonly appear, and they will help us obstruct formality in many cases.

\begin{lemma}\label{lem:SomeMasseyProds}
For $1 \leq k \leq n-1$ and $1 \leq i \leq n-1$, we have the following Massey admissible sequences in $H_*(\B(n,k,\Sc))$:
\begin{itemize}
\item If $i \in \Sc$, there is a Massey admissible sequence whose elements have labels 
\[
([L_i],[R_i],[R_{i+1}]),
\]
where the brackets $[\cdot]$ denote the homology class of an element of $\B(n,k,\Sc)$ with the given label.

\item If $i+1 \in \Sc$, there is a Massey admissible sequence whose elements have labels 
\[
([R_{i+1}],[L_{i+1}],[L_i]).
\]

\item If $\Sc \cap \{i,i+1\} = \{i\}$, there is a Massey admissible sequence whose elements have labels 
\[
([L_i],[R_i],[U_{i+1}]).
\]

\item If $\Sc \cap \{i,i+1\} = \{i+1\}$, there is a Massey admissible sequence whose elements have labels 
\[
([R_{i+1}],[L_{i+1}],[U_i]).
\]
\end{itemize}
\end{lemma}

\begin{proof}
Assuming that $i \in \Sc$, let $\x$ be any element of $V(n,k)$ such that $\x \cap [i-1,i+1] = \{i\}$. We have a sequence $(L_i,R_i,R_{i+1})$ of algebra elements such that the left idempotent of $L_i$ is $\Ib_{\x}$. Let $\x_1 = \x$, $\x_2 = (\x \setminus \{i\}) \cup \{i-1\}$, and $\x_3 = (\x \setminus \{i\}) \cup \{i+1\}$.

The products $[L_i] [R_i]$ and $[R_i] [R_{i+1}]$ are both zero (note that $\overline{\mu}_2$ must agree with the usual multiplication induced on $H_*(\B(n,k,\Sc))$, and $[U_i] = 0$ because $U_i = \de(C_i)$). The homological degrees of $L_i$ and $R_i$ are each $-1$. The Alexander multi-degrees of these elements are each $\frac{e_i}{2}$. 

The homology of $\Ib_{\x_1}\B(n,k,\Sc) \Ib_{\x_1}$ is zero in homological degree $-1$ and Alexander multi-degree $e_i$. Indeed, a basis element for $\Ib_{\x_1} H_*\B(n,k,\Sc) \Ib_{\x_1}$ in these degrees would need to be a single edge with label $C_i$, but we have $\de(C_i) \neq 0$ at the idempotent $\Ib_{\x_1}$.  Meanwhile, the homology of $\Ib_{\x_2} \B(n,k,\Sc) \Ib_{\x_3}$ is zero because $\x_2$ and $\x_3$ are far.  Thus, the sequence $([L_i],[R_i],[R_{i+1}])$ is Massey admissible; the case of $([R_{i+1}],[L_{i+1}],[L_i])$ when $i+1 \in \Sc$ is similar.

Now assume that $\Sc \cap \{i,i+1\} = \{i\}$. Again, let $\x$ be any element of $V(n,k)$ such that $\x \cap [i-1,i+1] = \{i\}$. We have a sequence $(L_i,R_i,U_{i+1})$ of algebra elements such that the left idempotent of $L_i$ is $\Ib_{\x}$. Let $\x_1 = \x$ and $\x_2 = (\x \setminus \{i\}) \cup \{i-1\}$.

As before, the products $[L_i] [R_i]$ and $[R_i] [U_{i+1}]$ are zero; the second product even vanishes in $\B(n,k,\Sc)$ because $[i+1]$ is a generating interval from $\x_2$ to $\x_1$. The homological degrees of $L_i$ and $R_i$ are each $-1$; the homological degree of $U_{i+1}$ is $0$. The Alexander multi-degrees of these elements are $\frac{e_i}{2}$, $\frac{e_i}{2}$, and $e_{i+1}$ respectively.

The homology $\Ib_{\x_1} H_*(\B(n,k,\Sc)) \Ib_{\x_1}$ is zero in homological degree $-1$ and Alexander multi-degree $e_i$ for the same reason as above. Since homological degrees are nonpositive and an edge with Alexander multi-degree $\frac{e_i}{2}$ has strictly negative homological degree, the summand $\Ib_{\x_2} \B(n,k,\Sc) \Ib_{\x_1}$ of $\B(n,k,\Sc)$ is zero in homological degree $0$ and Alexander multi-degree $\frac{e_i}{2} + e_{i+1}$. Thus, $\Ib_{\x_2} H_*(\B(n,k,\Sc)) \Ib_{\x_1}$ is also zero in this degree (note that this argument would not work if $i+1 \in \Sc$, since the relevant homological degree would be $-2$ and $\Ib_{\x_2} H_*(\B(n,k,\Sc)) \Ib_{\x_1}$ would have a basis element labeled $[C_{i+1} R_i]$ in this degree).

Thus, the sequence $([L_i],[R_i],[U_{i+1}])$ is Massey admissible; the case of $([R_{i+1}],[L_{i+1}],[U_i])$ when $\Sc \cap \{i,i+1\} = \{i+1\}$ is similar.
\end{proof}

Using Lemma~\ref{lem:SomeMasseyProds}, we can determine when $\B(n,k,\Sc)$ is formal.
\begin{theorem}\label{thm:UntruncatedFormality}
The dg algebra $\B(n,k,\Sc)$ is formal if and only if $\Sc = \varnothing$ or $k \in \{0,n,n+1\}$.
\end{theorem}

\begin{proof}
If $\Sc = \varnothing$ or $k = 0$, then $\B(n,k,\Sc)$ has no differential, so it is formal. If $k = n+1$, then $H_*(\B(n,k,\Sc)) \cong \F_2[U_i \,|\, i \notin \Sc]$ and the inclusion of these polynomials into $\B(n,k,\Sc)$ is a quasi-isomorphism (even a homotopy equivalence).

If $k = n$, then there are no generating intervals from $\x$ to $\y$ for any $\x,\y \in V(n,n)$. Thus, no $C_i$ labels appear in the basis for any summand $\Ib_{\x} H_*(\B(n,n,\Sc)) \Ib_{\y}$ from Theorem~\ref{thm:OSzHomology}.
Define a homomorphism from $\B(n,n,\Sc)$ to $H_*(\B(n,n,\Sc))$ by sending any $C_i$ generator to $0$ and by sending all other generators (which are contained in $\ker \de$) to their homology classes. For each defining relation of $\B(n,n,\Sc)$, either all terms or no terms of the relation involve a $C_i$ generator. Thus, the map is well defined, and it induces an isomorphism on homology by construction.

On the other hand, assume that $\Sc$ is nonempty and that $1 \leq k \leq n-1$. If $\Sc \neq \{n\}$, then Lemma~\ref{lem:SomeMasseyProds} gives us a Massey admissible sequence $([L_i],[R_i],[R_{i+1}])$. In the notation of \cite[Lemma 2.1.22]{LOTBimod}, let $\xi_{01} = L_i$, $\xi_{12} = R_i$, $\xi_{23} = R_{i+1}$, $\xi_{02} = C_i$, and $\xi_{13} = 0$. We have $\overline{\mu}_3([L_i],[R_i],[R_{i+1}]) = [\xi_{02} \xi_{23}] = [C_i R_{i+1}]$ by \cite[Lemma 2.1.22]{LOTBimod}. Since $[C_i R_{i+1}]$ is a basis element of $\Ib_{\x} H_*(\B(n,k,\Sc)) \Ib_{\y}$ and this product exists for any $\A_{\infty}$ structure $\overline{\mu}$ on $H_*(\B(n,k,\Sc))$ that is $\A_{\infty}$ homotopy equivalent to $\B(n,k,\Sc)$, the algebra $\B(n,k,\Sc)$ cannot be formal.

Similarly, if $\Sc \neq \{1\}$, then we have a product $\overline{\mu}_3([R_{i+1}],[L_{i+1}],[L_i]) = [C_{i+1} L_i]$ from Lemma~\ref{lem:SomeMasseyProds}, implying that $\B(n,k,\Sc)$ is not formal. We have already covered all possible cases; for completeness, we note that the remaining items of Lemma~\ref{lem:SomeMasseyProds} give us $\overline{\mu}_3([L_i],[R_i],[U_{i+1}]) = [C_i U_{i+1}]$ and $\overline{\mu}_3([R_{i+1}],[L_{i+1}],[U_i]) = [C_{i+1} U_i]$.
\end{proof}

\subsubsection{Truncated algebras}\label{sec:TruncatedFormality}

We now investigate formality for $\B_r(n,k,\Sc)$, $\B_l(n,k,\Sc)$, and $\B'(n,k,\Sc)$, starting with $\B_r(n,k,\Sc)$. First, we deal with an especially tricky subcase.

\begin{lemma}\label{lem:OneSidedTrickyCase}
If $1\in\Sc$, then the dg algebra $\B_r(n,n-1,\Sc)$ is formal.
\end{lemma}

\begin{proof}
If $n = 1$ the differential is zero, so assume $n \geq 2$. Let $\A$ be the path algebra of the subgraph $\Gamma_{\A}$ of $\Gamma_r(n,n-1,\Sc)$ in which we omit all $U_1$ and $C_1$ loops except for a single $C_1$ loop at the vertex $\x_1 = [2,n]$, modulo the two-sided ideal generated by the elements of $\td{\mc R}_r$ involving only edges in $\Gamma_{\A}$. Explicitly, we take the quotient by the following relations:
\begin{itemize}
\item $U$ central relations not involving $U_1$,
\item loop relations not involving $U_1$,
\item $C$ central relations not involving $C_1$,
\item $C_1$ central relations of the form $C_1C_i=C_iC_1$ and $C_1U_i=U_iC_1$ for $i\neq 1$ (all occurring at vertex $\x_1$), and
\item $C^2$ vanishing relations not involving $C_1$ at a vertex other than $\x_1$.
\end{itemize}
Define a homomorphism of dg algebras $\kappa: \A \to \B_r(n,n-1,\Sc)$ by sending each edge in $\Gamma_{\A}$ to the corresponding generator of $\B_r(n,n-1,\Sc)$. Since $\kappa$ sends the relation set for $\A$ into the relation set for $\B_r(n,n-1,\Sc)$, $\kappa$ is well-defined.

We want to show that $\kappa$ is a quasi-isomorphism; to do this, we compute the homology of $\A$. For $i \in [1,n]$, let $\x_i = [1,n] \setminus \{i\}$. We claim that any path in $\Gamma_{\A}$ from $\x_i$ to $\x_j$ is equal, modulo the relations defining $\A$, to a product of $U_l$ and $C_l$ loops for $l > 1$ with either:
\begin{enumerate}
\item\label{it:FormalityType1Elements} $\gamma_{\x_i,\x_j}$, or
\item\label{it:FormalityType2Elements} $\gamma_{\x_i,\x_1} C_1 \gamma_{\x_1,\x_j}$.
\end{enumerate}

First let $\gamma$ be a path in $\Gamma_{\A}$ from $\x_i$ to $\x_j$ with no edge labeled $C_1$. Since $\A$ has all $U$ central relations and $C$ central relations not involving $U_1$ or $C_1$, $\gamma$ is equal in $\A$ to a product of $U_l$ and $C_l$ loops for $l > 1$ (in any order) with a path $\gamma'$ from $\x_i$ to $\x_j$ having only $R$- or $L$-labeled edges. If $\gamma'$ has exactly $|i-j|$ edges, then $\gamma' = \gamma_{\x_i,\x_j}$. If $\gamma'$ has more than $|i-j|$ edges, then it contains a consecutive pair of edges with labels $(R_k,L_k)$ or $(L_k, R_k)$ for some $k \in [2,n]$. We can use the loop relations of $\A$ to replace $\gamma'$ with an equivalent shorter path while adding a $U_k$ loop to the product for $\gamma$, proving the claim for paths containing no $C_1$ edge by induction on the length of $\gamma'$.

If $\gamma$ is a path in $\Gamma_{\A}$ containing more than one $C_1$ edge, let $\gamma'$ be a subpath of $\gamma$ containing all edges between two consecutive instances of $C_1$ (non-inclusive). The starting and ending vertex of $\gamma'$ are $\x_1$, so by the above argument, $\gamma'$ is equivalent to a product of $U_l$ and $C_l$ loops for $l > 2$ with the identity path $\gamma_{\x_1, \x_1}$ modulo the relations defining $\A$. The $C_1$ central relations existing in $\A$ then imply that $\gamma$ is equal in $\A$ to a path with two consecutive $C_1$ edges at the vertex $\x_1$, which is zero by the one instance of a $C_1^2$ vanishing relation existing in $\A$.

Finally, if $\gamma$ is a path in $\Gamma_{\A}$ from $\x_i$ to $\x_j$ with exactly one $C_1$ edge, write $\gamma = \gamma' C_1 \gamma''$, where $C_1$ stands for the single-edge path from $\x_1$ to itself with label $C_1$. By the argument for paths with no $C_1$ edges, $\gamma'$ and $\gamma''$ are equal in $\A$ to $\gamma_{\x_i,\x_1}$ and $\gamma_{\x_1,\x_j}$ respectively, times some number of $U_l$ and $C_l$ loops for $l > 1$. Thus, $\gamma$ is equal in $\A$ to a product of $\gamma_{\x_i, \x_1} C_1 \gamma_{\x_1,\x_j}$ with $U_l$ and $C_l$ loops for $l > 1$ as desired.

Thus, products of the elements \eqref{it:FormalityType1Elements} and \eqref{it:FormalityType2Elements} above with $U_l$ and $C_l$ loops for $l > 1$ form a spanning set for $\Ib_{\x_i} \A \Ib_{\x_j}$ over $\F_2$. Choose one such product of each type for each pair of a monomial in $U_l$ for $l > 1$ and a square-free monomial in $C_l$ for $l \in [2,n] \cap \Sc$. We want to show that the set $\beta$ of these products is linearly independent over $\F_2$, and thus forms a basis for $\Ib_{\x_i} \A \Ib_{\x_j}$.

Since the algebra homomorphism $\kappa: \A \to \B_r(n,n-1,\Sc)$ is, in particular, an $\F_2$-linear map, it suffices to show that $\kappa$ maps the elements of $\beta$ to distinct elements of a basis for $\Ib_{\x_i} \B_r(n,n-1,\Sc) \Ib_{\x_j}$. In fact, the image under $\kappa$ of each element of $\beta$ is one of the basis elements for $\Ib_{\x_i} \B_r(n,n-1,\Sc) \Ib_{\x_j}$ given in Proposition~\ref{prop:OSzBasis} (recall that we have $R_l \cdots R_2 C_1 L_2 \cdots L_l = C_1 U_2 \cdots U_l$ in $\B_r(n,n-1,\Sc)$), and one can check that $\kappa$ restricted to $\beta$ is injective.

Therefore, we have a basis of $\Ib_{\x_i} \A \Ib_{\x_j}$ over $\F_2$. The subset of basis elements that do not involve $C_1$ (i.e. those from \eqref{it:FormalityType1Elements}) is in differential-preserving bijection with the set of basis elements for $\Ib_{\x_i - 1}\B(n-1,n-1,\Sc-1)\Ib_{\x_j-1}$ from Proposition~\ref{prop:OSzBasis}, where $\Sc - 1 = \{l - 1 \,|\, l \in \Sc \cap [2,n]\}$. Since $\x_1=[2,n]$, we have
\[
\de ( \gamma_{\x_i,\x_1} C_1 \gamma_{\x_1,\x_j} ) = 0,
\]
so the basis elements for $\Ib_{\x_i} \A \Ib_{\x_j}$ that do involve $C_1$ (i.e. those from \eqref{it:FormalityType2Elements}) are also in differential-preserving bijection with basis elements for $\Ib_{\x_i - 1}\B(n-1,n-1,\Sc-1)\Ib_{\x_j-1}$.

Using these bijections, we can deduce a basis for $\Ib_{\x_i} H_*(\A) \Ib_{\x_j}$ from Theorem~\ref{thm:OSzHomology}; note that there are no generating intervals from $\x_i-1$ to $\x_j-1$. The basis elements are $p[\gamma_{\x_i,\x_j}]$ and $p[\gamma_{\x_i,\x_1} C_1 \gamma_{\x_1,\x_j}]$ where $p$ ranges over all monomials in $[U_l]$ for $l \in [2,n] \setminus \Sc$. Using Theorem~\ref{thm:OSzHomology} again, we see that $\kappa$ sends these elements to a basis for $\Ib_{\x_i} H_*(\B_r(n,n-1,\Sc)) \Ib_{\x_j}$ (note that $[1, \min(i,j)]$ is the only generating interval from $\x_i$ to $\x_j$). It follows that $\kappa$ is a quasi-isomorphism.

We can also define a homomorphism of dg algebras $\lambda: \A \to H_*(\B_r(n,n-1,\Sc))$ as follows. Generators of $\A$ labeled $R_i$, $L_i$, and $U_i$ are in the kernel of $\partial$. The generator of $\A$ labeled $C_1$ is a loop at the vertex $\x_1 = [2,n]$, so it is also in the kernel of $\partial$. Let $\lambda$ send each of these generators to the homology class of its image under $\kappa$ and send all other generators of $\A$ (loops labeled $C_l$ for $l > 1$ in $\Sc$) to zero. Any defining relation in $\A$ not involving any such $C_l$ is in the kernel of $\kappa$, so it is in the kernel of $\lambda$.  Meanwhile, any defining relation in $\A$ involving some $C_l$ ($l>1$ in $\Sc$) is in fact wholly divisible by $C_l$, and as such gets mapped to zero as well. Thus, $\lambda$ is well-defined.

To see that $\lambda$ respects the differential, note that $H_*(\B_r(n,n-1,\Sc))$ has zero differential, so it is enough to show that $\lambda\circ\de=0$. Since $\lambda$ is either the zero map or the quotient to homology on each generator of $\A$, it sends boundaries to zero, i.e.~$\lambda \circ \de = 0$.

By Theorem~\ref{thm:OSzHomology}, $\lambda$ sends the basis elements for $\Ib_{\x_i} H_*(\A) \Ib_{\x_j}$ listed above to a basis for $\Ib_{\x_i} H_*(\B_r(n,n-1,\Sc)) \Ib_{\x_j}$, so $\lambda$ is a quasi-isomorphism. In summary, we have a zig-zag of quasi-isomorphisms
\[
\B_r(n,n-1,\Sc) \xleftarrow{\kappa} \A \xrightarrow{\lambda} H_*(\B_r(n,n-1,\Sc)),
\]
so $\B_r(n,n-1,\Sc)$ is formal.
\end{proof}

\begin{remark}
It appears impossible to define a quasi-isomorphism directly from $\B_r(n,n-1,[1,n])$ to its homology or vice-versa.
\end{remark}

\begin{theorem}\label{thm:RightTruncationFormality}
The dg algebra $\B_r(n,k,\Sc)$ is formal if and only if one of the following conditions holds:
\begin{itemize}
\item $\Sc = \varnothing$, $\Sc = \{1\}$, $k = 0$, or $k = n$;
\item $k = n-1$ and $1\in\Sc$.
\end{itemize}
\end{theorem}

\begin{proof}
The cases $\Sc = \varnothing$, $k = 0$, and $k = n$ follow as in Theorem~\ref{thm:UntruncatedFormality}. The case $k = n-1$ and $1 \in \Sc$ follows from Lemma~\ref{lem:OneSidedTrickyCase}.

If $\Sc = \{1\}$ and $k < n$, let $\x,\y \in V_r(n,k)$ (following the notation of Definition~\ref{def:TruncatedOSzAlgs}). By Theorem~\ref{thm:OSzHomology}, we have a basis for $\Ib_{\x} H_*(\B_r(n,k,\Sc)) \Ib_{\y}$ consisting of classes $[p]$ where $p$ is a monomial in the $U_i$ variables not divisible by $U_1$ or the monomial for any generating interval, as well as classes $[C_1 \cdots U_l p]$ where $[1,l]$ is the first generating interval from $\x$ to $\y$. Note that a generating interval $[1, l]$ exists because $k < n$. Define a map from $\Ib_{\x} H_*(\B_r(n,k,\Sc)) \Ib_{\y}$ to $\Ib_{\x} \B_r(n,k,\Sc) \Ib_{\y}$ by sending $[p]$ to $p$ and $[C_1 \cdots U_l p]$ to $C_1 \cdots U_l p$. By definition, this map induces an isomorphism on homology; one can check that it is also compatible with the multiplication on $H_*(\B_r(n,k,\Sc))$ and $\B_r(n,k,\Sc)$. Thus, $\B_r(n,k,\Sc)$ is quasi-isomorphic to its homology in this case.

On the other hand, assume $\Sc \cap [2,n] \neq \varnothing$ and $1 \leq k \leq n-2$. As in the proof of Theorem~\ref{thm:UntruncatedFormality}, Lemma~\ref{lem:SomeMasseyProds} gives us a canonical higher multiplication $\overline{\mu}_3([L_i],[R_i],[R_{i+1}]) = [C_i R_{i+1}]$ or $\overline{\mu}_3([R_{i+1}],[L_{i+1}],[L_i]]) = [C_{i+1} L_i]$ for some $i$, so $\B_r(n,k,\Sc)$ is not formal. 

Lastly, let $k=n-1$. Assuming that $\Sc$ is nonempty and $1 \notin \Sc$, we show that there exists a nonzero triple Massey product.
Let $l = \min \Sc$; we must have $l>1$ because $1 \notin \Sc$. We claim that there is a Massey admissible sequence labeled
\begin{equation}
\label{eq:ExtraMasseySequence}
([R_l \cdots R_2], [L_2 \cdots L_l], [U_1])
\end{equation}
such that the left idempotent of $[R_l \cdots R_2]$ is $\x_l = [1,n]\sm\set{l}$. The right idempotent of $[L_2 \cdots L_l]$ and both the left and right idempotents of $[U_1]$ are also $\x_l$. The right idempotent of $[R_l \cdots R_2]$ and the left idempotent of $[L_2 \cdots L_l]$ are $\x_1 = [1,n]\sm\set1$.
The homology classes $[R_l \cdots R_2]$ and $[L_2 \cdots L_l]$ both have homological degree $-1$ and Alexander multi-degree $\frac12(e_2 + \cdots + e_l)$, while $[U_1]$ has homological degree $0$ and Alexander multi-degree $e_1$. By Theorem \ref{thm:OSzHomology}, $\Ib_{\x_l} H_*(\B_r(n,k,\Sc)) \Ib_{\x_l}$ vanishes in homological degree $-1$ and Alexander multi-degree $e_2+\cdots+e_l$, and $\Ib_{\x_1} H_*(\B_r(n,k,\Sc)) \Ib_{\x_l}$ vanishes in homological degree $0$ and Alexander multi-degree $e_1 + \frac12(e_2+\cdots+e_l)$.
Moreover,
\[
\overline{\mu}_2([R_l \cdots R_2], [L_2 \cdots L_l]) = [U_2 \cdots U_l] = [\de(U_2 \cdots U_{l-1} C_l)] = 0,
\]
and $\overline{\mu}_2([L_2 \cdots L_l], [U_1]) = 0$ because $[U_1]$ can be commuted to the left but $U_1 = 0$ at the vertex $\x_1$ even before taking homology.
Thus, the sequence in \eqref{eq:ExtraMasseySequence} is Massey admissible.
The triple Massey product is computed following \cite[Lemma 2.1.22]{LOTBimod}: let $\xi_{01} = R_l \cdots R_2$, $\xi_{12} = L_2 \cdots L_l$, $\xi_{23} = U_1$, $\xi_{02} = U_2 \cdots U_{l-1}C_l$, and $\xi_{13} = 0$. We have
\[
\overline{\mu}_3([R_l \cdots R_2], [L_2 \cdots L_l], [U_1]) = [\xi_{02}\xi_{23}] = [U_1 \cdots U_{l-1}C_l].
\]
Note that this element is nonzero in $\Ib_{\x_l} H_*(\B_r(n,k,\Sc)) \Ib_{\x_l}$ by Theorem \ref{thm:OSzHomology}, since $[1,l]$ is a generating interval from $\x_l$ to itself.
\end{proof}

Using the symmetry $\rho: \B_l(n,k,\Sc) \xrightarrow{\cong} \B_r(n,k,\rho(\Sc))$, we can deduce the following theorem from Theorem \ref{thm:RightTruncationFormality}.
\begin{theorem}\label{thm:LeftTruncationFormality}
The dg algebra $\B_l(n,k,\Sc)$ is formal if and only if one of the following conditions holds:
\begin{itemize}
\item $\Sc = \varnothing$, $\Sc = \{n\}$, $k = 0$, or $k = n$;
\item $k = n-1$ and $n \in \Sc$.
\end{itemize}
\end{theorem}

For the doubly-truncated algebra $\B'(n,k,\Sc)$, we again consider an interesting subcase.

\begin{lemma}\label{lem:TwoSidedTrickyCase}
If both $1\in\Sc$ and $n\in\Sc$, the algebra $\B'(n,n-2,\Sc)$ is formal.
\end{lemma}

\begin{proof}
The algebra is zero if $n = 1$ and has no differential if $n = 2$, so assume $n \geq 3$.  Let $\x_i$ denote the vertex $[1,n-1]\setminus\{i\}$, and let $\gammahole{i}{j}$ denote the path $\gamma_{\x_i,\x_j}$.  In parallel fashion to the proof of Lemma~\ref{lem:OneSidedTrickyCase}, we consider the subgraph $\Gamma_\Cc$ of $\Gamma'(n,n-2,\Sc)$ in which we omit all loops with labels $U_1,U_n,C_1,C_n$ except for a single $C_1$ loop at $\x_1$ and a single $C_n$ loop at $\x_{n-1}$.  We will let $\Cc$ be the path algebra of $\Gamma_\Cc$ modulo the following relations:
\begin{itemize}
\item $U$ central relations not involving $U_1$ or $U_n$,
\item loop relations not involving $U_1$ or $U_n$,
\item $C$ central relations not involving $C_1$ or $C_n$,
\item $C_1$ central relations of the form $C_1C_i=C_iC_1$ and $C_1U_i=U_iC_1$ for $i\notin\{1,n\}$ (all occurring at vertex $\x_1$),
\item $C_n$ central relations of the form $C_nC_i=C_iC_n$ and $C_nU_i=U_iC_n$ for $i\notin\{1,n\}$ (all occurring at vertex $\x_{n-1}$),
\item $C^2$ vanishing relations not involving $C_1$ (respectively $C_n$) at a vertex other than $\x_1$ (respectively $\x_{n-1}$), and
\item the ``Pong" relations $\gammahole{i}{1} C_1 \gammahole{1}{n-1} C_n \gammahole{n-1}{i} = \gammahole{i}{n-1} C_n \gammahole{n-1}{1} C_1 \gammahole{1}{i}$ for $i\in[1,n-1]$.
\end{itemize}
Unlike in Lemma~\ref{lem:OneSidedTrickyCase}, the Pong relations are not in the relation set $\td{\mc R}'$ defining $\B'(n,n-2,\Sc)$.

As in Lemma \ref{lem:OneSidedTrickyCase}, we have a well-defined dg algebra homomorphism $\kappa':\Cc\rightarrow\B'(n,n-2,\Sc)$ that sends edges of $\Gamma_\Cc$ to corresponding generators in $\B'(n,n-2,\Sc)$. While the Pong relations are not in $\td{\mc R}'$, they are in the two-sided ideal generated by $\td{\mc R}'$; the elements $\gammahole{i}{1} C_1 \gammahole{1}{n-1} C_n \gammahole{n-1}{i}$ and $\gammahole{i}{n-1} C_n \gammahole{n-1}{1} C_1 \gammahole{1}{i}$ are both equal to $C_1 U_2 \cdots U_{n-1} C_n$ in $\B'(n,n-2,\Sc)$.

We again seek a basis for $\Ib_{\x_i}\Cc\Ib_{\x_j}$.  This time around, we consider the following elements:
\begin{enumerate}
\item\label{it:Formality' basis no C} $\gammahole{i}{j}$,
\item\label{it:Formality' basis C1} $\gammahole{i}{1} C_1 \gammahole{1}{j}$,
\item\label{it:Formality' basis Cn} $\gammahole{i}{n-1} C_n \gammahole{n-1}{j}$,
\item\label{it:Formality' basis C1Cn} $\gammahole{i}{1} C_1 \gammahole{1}{n-1} C_n \gammahole{n-1}{j}$ for $i\leq j$, or
\item\label{it:Formality' basis CnC1} $\gammahole{i}{n-1} C_n \gammahole{n-1}{1} C_1 \gammahole{1}{j}$ for $j<i$.
\end{enumerate}
We define a set $\beta'$ analogous to $\beta$ in Lemma \ref{lem:OneSidedTrickyCase} by choosing, for each monomial in $U_l$ and $C_l$ ($l\neq 1,n$) not divisible by any $C_l^2$, a corresponding path of each type above.  To prove that $\beta'$ spans $\Ib_{\x_i}\Cc\Ib_{\x_j}$, we consider various cases.  If we have a path $\gamma\in\Ib_{\x_i}\Path(\Gamma_\Cc)\Ib_{\x_j}$ that contains no $C_1$ or $C_n$ loops, or only one such loop, then it is equivalent to a $U_l$- and $C_l$-multiple of element \eqref{it:Formality' basis no C}, \eqref{it:Formality' basis C1}, or \eqref{it:Formality' basis Cn} as in the proof of Lemma \ref{lem:OneSidedTrickyCase}.  Similarly, if $\gamma\in\Ib_{\x_i}\Path(\Gamma_\Cc)\Ib_{\x_j}$ has two or more $C_1$ loops with no $C_n$ loops occurring between them (or $C_n$ loops with no $C_1$ loops between them), then $\gamma=0$ in $\Ib_{\x_i}\Cc\Ib_{\x_j}$ by arguments like those in the proof of Lemma \ref{lem:OneSidedTrickyCase}.

In the genuinely new case where $\gamma\in\Ib_{\x_i}\Path(\Gamma_\Cc)\Ib_{\x_j}$ contains a single $C_1$ loop followed by a single $C_n$ loop (and no other $C_1$ or $C_n$ loops), the arguments in the proof of Lemma \ref{lem:OneSidedTrickyCase} allow us to write $\gamma$ as a $U_l$- and $C_l$-multiple of $\gammahole{i}{1} C_1 \gammahole{1}{n-1} C_n \gammahole{n-1}{j}$.  If $i\leq j$, we are done ($\gamma$ is a multiple of element \eqref{it:Formality' basis C1Cn}).  If $i>j$, we use the further relations
\begin{align*}
\gammahole{i}{1} C_1 \gammahole{1}{n-1} C_n \gammahole{n-1}{j} &= \gammahole{i}{j}\gammahole{j}{1} C_1 \gammahole{1}{n-1} C_n \gammahole{n-1}{j}\\
&= \gammahole{i}{j} \gammahole{j}{n-1} C_n \gammahole{n-1}{1} C_1 \gammahole{1}{j}\\
&= U_{j+1}\cdots U_i \gammahole{i}{n-1} C_n \gammahole{n-1}{1} C_1 \gammahole{1}{j}
\end{align*}
where the second equality follows from the Pong relations, and the others follow directly from the definitions and loop relations.  We see that $\gamma$ is a $U_l$- and $C_l$-multiple of element \eqref{it:Formality' basis CnC1}.  The case where $\gamma$ contains a single $C_n$ loop followed by a single $C_1$ loop is handled similarly.

Finally, in the case that $\gamma$ contains two $C_1$ loops separated by a $C_n$ loop (or vice versa), the Pong relations can be used to swap the order of a $C_1$ with a $C_n$. We get two consecutive $C_1$ or $C_n$ edges, causing $\gamma$ to be zero in $\Cc$ as above.

Thus $\beta'$ is a spanning set. The linear independence of the elements in $\beta'$ follows as in Lemma~\ref{lem:OneSidedTrickyCase}, since $\kappa'$ maps each element of $\beta'$ to a distinct basis element in $\B'(n,n-2,\Sc)$. We conclude that $\beta'$ is a basis for $\Ib_{\x_i}\Cc\Ib_{\x_j}$.

Note that, after fixing $\x_i$ and $\x_j$, we consider only one of the two elements \eqref{it:Formality' basis C1Cn} and \eqref{it:Formality' basis CnC1}; in either case, we find the image under $\kappa'$ to be a $U_l$- and $C_l$-multiple of an element of the form $(C_1U_2\cdots U_{\min(i,j)})(U_{\max(i,j)+1}\cdots U_{n-1}C_n) \gammahole{i}{j}$ in $\B'(n,n-2,\Sc)$.  The intervals $[1,\min(i,j)]$ and $[\max(i,j)+1, n]$ are precisely the generating intervals in $\B'(n,n-2,\Sc)$ from $\x_i$ to $\x_j$.

Using these facts together with Proposition \ref{prop:OSzBasis} and Theorem \ref{thm:OSzHomology}, one can show that $\kappa'$ is a quasi-isomorphism as in the proof of Lemma \ref{lem:OneSidedTrickyCase}.  This time the basis of $\Ib_{\x_i} H_*(\Cc) \Ib_{\x_j}$ is deduced by comparing $\Ib_{\x_i} \Cc \Ib_{\x_j}$ to four copies of $\Ib_{\x_i-1} \B(n-2,n-2,\Sc-1) \Ib_{\x_j-1}$; each copy of the latter complex corresponds to a type of basis element of $\Ib_{\x_i} \Cc \Ib_{\x_j}$ listed above (namely types \eqref{it:Formality' basis no C}, \eqref{it:Formality' basis C1}, \eqref{it:Formality' basis Cn}, or \eqref{it:Formality' basis C1Cn}/\eqref{it:Formality' basis CnC1}).

The remainder of the proof follows along the lines of the proof of Lemma \ref{lem:OneSidedTrickyCase}.  We have a well-defined dg algebra homomorphism $\lambda':\Cc \rightarrow H_*(\B'(n,n-2,\Sc))$ that sends edges labeled $C_l$ for $l\neq 1,n$ to zero and sends all other edges (which are in $\ker(\de)$) to the homology classes of their images under $\kappa'$.  The map $\lambda'$ is well-defined and respects the differential by the same arguments as in the proof of Lemma \ref{lem:OneSidedTrickyCase}. Theorem \ref{thm:OSzHomology} is used in a similar fashion to show that $\lambda'$ maps a basis for $\Ib_{\x_i}H_*(\Cc)\Ib_{\x_j}$ to a basis for $\Ib_{\x_i}H_*(\B'(n,n-2,\Sc))\Ib_{\x_j}$.  

Thus, we again have a zig-zag of quasi-isomorphisms
\[\B'(n,n-2,\Sc) \xleftarrow{\kappa'} \Cc \xrightarrow{\lambda'} H_*(\B'(n,n-2,\Sc)),\]
showing that $\B'(n,n-2,\Sc)$ is formal.
\end{proof}

\begin{remark}
Our terminology in the proof of Lemma~\ref{lem:TwoSidedTrickyCase} follows Ozsv{\'a}th--Szab{\'o}. In their more general ``Pong algebra,'' to be introduced in \cite{OSzPong}, the above Pong relations are not zero. Instead, the Pong algebra is curved, and these relations form the curvature in one summand of the algebra.
\end{remark}

\begin{theorem}\label{thm:DoubleTruncationFormality}
The dg algebra $\B'(n,k,\Sc)$ is formal if and only if one of the following conditions holds:
\begin{itemize}
\item $\Sc = \varnothing, \{1\}, \{n\}$, or $\{1,n\}$;
\item $k = 0$ or $k = n-1$;
\item $k = n-2$ and $\set{1,n} \subset \Sc$. 
\end{itemize}
\end{theorem}

\begin{proof}
The cases $\Sc = \varnothing$ and $k = 0$ follow as in Theorem~\ref{thm:UntruncatedFormality}.

When $\Sc$ is nonempty and $k = n-1$, a basis for $H_*(\B'(n,n-1,\Sc))$ is given by elements $\left[p\bigg(\frac{C_j U_1 \cdots U_n}{U_j}\bigg)^{\varepsilon}\right]$ where $p$ is a monomial in the $U_i$ variables for $i \notin \Sc$, $\varepsilon \in \{0,1\}$, and $j$ is any element of $\Sc$ (the resulting basis element for the homology is independent of the choice of $j$). Choosing one particular $j \in \Sc$, the map sending
\[
\left[p\bigg(\frac{C_j U_1 \cdots U_n}{U_j}\bigg)^{\varepsilon}\right] \mapsto p\bigg(\frac{C_j U_1 \cdots U_n}{U_j}\bigg)^{\varepsilon}
\]
is a quasi-isomorphism from $H_*(\B'(n,n-1,\Sc))$ to $\B'(n,n-1,\Sc)$.

The case $\Sc = \{1\}$ is treated as in Theorem~\ref{thm:RightTruncationFormality}; the case $\Sc = \{n\}$ follows by applying the symmetry $\rho$. For the case $\Sc = \{1,n\}$, only a slight modification is needed: the quasi-isomorphism sends
\begin{itemize}
\item $[p] \mapsto p$,
\item $[C_1 \cdots U_l p] \mapsto C_1 \cdots U_l p$,
\item $[U_{l'} \cdots C_n p] \mapsto U_{l'} \cdots C_n p$, and
\item $[(C_1 \cdots U_l)(U_{l'} \cdots C_n)p] \mapsto (C_1 \cdots U_l)(U_{l'} \cdots C_n)p$.
\end{itemize}

The case $k = n-2$ and $\set{1,l} \subset \Sc$ follows from Lemma~\ref{lem:TwoSidedTrickyCase}. Conversely, we show that in any of the remaining cases there is a canonical non-vanishing triple Massey product.

If $\Sc \cap [2,n-1] \neq \varnothing$ and $1 \leq k \leq n-3$, then Lemma~\ref{lem:SomeMasseyProds} gives us a canonical higher multiplication $\overline{\mu}_3([L_i],[R_i],[R_{i+1}]) = [C_i R_{i+1}]$ or $\overline{\mu}_3([R_{i+1}],[L_{i+1}],[L_i]]) = [C_{i+1} L_i]$ for some $i$, so $\B'(n,k,\Sc)$ is not formal.

If $k = n-2$, $1 \notin \Sc$, and $\Sc \cap [2,n-1]$ is nonempty, then as in the proof of Theorem \ref{thm:RightTruncationFormality} we have a Massey admissible sequence
\[
([R_{\min{\Sc}} \cdots R_2], [L_2 \cdots L_{\min{\Sc}}], [U_1])
\]
inducing a non-vanishing triple product.
Lastly, if $k = n-2$, $n \notin \Sc$, and $ \Sc \cap [2,n-1]$ is nonempty, a symmetric argument shows that
\[
([L_{\max{\Sc}} \cdots L_{n-1}], [R_{n-1} \cdots R_{\max\Sc}], [U_n])
\]
is a Massey admissible sequence with non-vanishing triple product.
\end{proof}

\begin{remark}\label{rem:FormalityWithOtherGradings}
We have used the refined Alexander grading when discussing formality. A priori, formality of $\B(n,k,\Sc)$ when given the unrefined Alexander grading could be stronger than formality given the refined Alexander grading, which could be stronger than formality given the single Alexander grading. In fact, one can check that in all cases where formality is proven above, the results hold even for the unrefined Alexander grading; in particular, all of our quasi-isomorphisms respect the unrefined grading. However, our arguments for Massey admissibility break down in the singly graded case, so that non-formality results proved above do not necessarily hold. It would be interesting to determine whether $\B(n,k,\Sc)$ is sometimes formal when given the single Alexander grading but not formal when given multiple Alexander gradings.
\end{remark}

\appendix
\section{Algebraic background}\label{app:Algebra}

In this section, we review some useful algebraic and category-theoretic definitions for readers who may be unfamiliar with them. For some motivation, consider the familiar ``torus algebra'' $\A(\mc Z, 1)$ from bordered Floer homology, which can be described as the $\mathbb{F}_2$-algebra of paths in the directed graph of Figure~\ref{fig:TorusAlg} modulo the relations $\rho_2 \rho_1 = \rho_3 \rho_2 = 0$ (path algebras with relations are discussed more formally in Section~\ref{sec:QuiverAlgs}).

\begin{figure}
\includegraphics[scale=0.5]{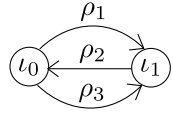}
\caption{The torus algebra as described in \cite[Figure 4]{HRW}.}
\label{fig:TorusAlg}
\end{figure}

It is often important to remember that $\A(\mc Z, 1)$ comes with a distinguished pair of orthogonal idempotents, corresponding to the constant paths at the vertices $\iota_0$ and $\iota_1$ in Figure~\ref{fig:TorusAlg}. One can remember this fact by viewing $\A(\mc Z, 1)$ as an algebra over $\mathbb{F}_2^{\times 2}$, but an equivalent and sometimes more natural way to remember the idempotent data is to interpret $\A(\mc Z,1)$ as a category with two objects (corresponding to the basic idempotents). One wants to be able to recover $\A(\mc Z, 1)$ as the direct sum of the (four) morphism spaces in the category. Because $\A(\mc Z,1)$ is defined over the ground field $\mathbb{F}_2$, these morphism spaces should be $\mathbb{F}_2$-vector spaces rather than just sets. 

Thus, the category corresponding to $\A(\mc Z, 1)$ should be an $\mathbb{F}_2$-linear category. An $\mathbb{F}_2$-linear category is informally a category whose morphism spaces are $\mathbb{F}_2$-vector spaces, rather than sets, and whose composition maps are $\mathbb{F}_2$-linear maps. If we had a more general dg algebra instead of $\A(\mc Z,1)$, the corresponding category should be a differential graded, or dg, category over $\mathbb{F}_2$, which is informally a category whose morphism spaces are chain complexes of $\mathbb{F}_2$-vector spaces and whose composition maps are chain maps. We define linear categories and dg categories in Section~\ref{sec:LinearDGDefs}; we review our intended relationship between algebras and categories in Section~\ref{sec:AlgsAndCats}.

In Section~\ref{sec:B0Section}, we find it useful to work with algebras and categories defined over a polynomial ring $\mathbb{F}_2[U_1,\ldots,U_n]$. Thus we give the below definitions over a general commutative ring $\ring$ of characteristic $2$.

\subsection{Algebras}\label{sec:AppendixAlgebras}

Let $\ring$ be a commutative ring of characteristic $2$.
A \emph{$\ring$-algebra} is a ring $\mc A$ equipped with a ring homomorphism $\ring \to \mc A$. 
\begin{remark}
\label{rem:SAS}
Equivalently, a $\ring$-algebra is a monoid in the monoidal category of $(\ring,\ring)$-bimodules; see below for some basics on monoidal categories. Note that our definition differs slightly from another standard definition taking a $\ring$-algebra to be a monoid in the monoidal category of $\ring$-modules, or equivalently a ring $\mc A$ equipped with a ring homomorphism from $\ring$ to $Z(\mc A)$.
\end{remark}

A \emph{differential algebra} over $\ring$ is a $\ring$-algebra $\mc A$ equipped with a $\ring$-linear map $\partial: \mc A \to \mc A$ satisfying $\partial^2 = 0$ and $\partial(ab) = (\partial a)b + a(\partial b)$ for all $a,b \in \mc A$. If $\A$ is a differential $\ring$-algebra, we can take its homology to get a $\ring$-algebra $H_*(\A)$. 

Let $G$ be a group and let $\lambda$ be an element of the center of $G$. A $(G,\lambda)$-graded \emph{differential graded algebra} or \emph{dg algebra} over $\ring$ is a differential $\ring$-algebra $\A$ equipped with a decomposition $\A = \oplus_{g \in G} \A_g$ as $\ring$-modules, such that $\A_g \cdot \A_{g'} \subset \A_{gg'}$ and $\partial(\A_g) \subset \A_{\lambda^{-1} g}$ for all $g, g' \in G$. If differentials are not present, we can define $G$-graded $\ring$-algebras similarly. If $\A$ is a $(G,\lambda)$-graded dg $\ring$-algebra, its homology $H_*(\A)$ is a $G$-graded $\ring$-algebra.

If $\mc A$ and $\mc A'$ are differential $\ring$-algebras, a homomorphism of differential algebras from $\mc A$ to $\mc A'$ is a homomorphism $f: \mc A \to \mc A'$ of $\ring$-algebras satisfying $\partial' \circ f = f \circ \partial$. If $\mc A$ and $\mc A'$ are $(G,\lambda)$-graded dg $\ring$-algebras, a homomorphism of dg algebras from $\mc A$ to $\mc A'$ is a homomorphism $f$ of differential algebras such that $f(A_g) \subset A'_g$ for all $g \in G$. Homomorphisms of $G$-graded $\ring$-algebras (without differentials) are defined similarly. We can consider the following categories:
\begin{itemize}
\item ($\ring$-algebras, homomorphisms)
\item ($G$-graded $\ring$-algebras, homomorphisms)
\item (differential $\ring$-algebras, homomorphisms)
\item ($(G,\lambda)$-graded dg $\ring$-algebras, homomorphisms).
\end{itemize}

A homomorphism $f: \mc A \to \mc A'$ of differential or dg algebras induces a map $H_*(f)$ from $H_*(\mc A)$ to $H_*(\mc A')$. Indeed, homology gives us functors $H_*$ from (differential $\ring$-algebras, homomorphisms) to ($\ring$-algebras, homomorphisms) and from ($(G,\lambda)$-graded dg $\ring$-algebras, homomorphisms) to ($G$-graded $\ring$-algebras, homomorphisms).

\begin{definition}
Let $f: \mc A \to \mc A'$ be a homomorphism of differential or $(G,\lambda)$-graded dg algebras over $\ring$. If the induced map $H_*(f): H_*(\mc A) \to H_*(\mc A')$ is an isomorphism, $f$ is called a \emph{quasi-isomorphism}.
\end{definition}

\subsection{Linear and dg categories}\label{sec:LinearDGDefs}

As mentioned above, it is often convenient to have an analogue of Section~\ref{sec:AppendixAlgebras} for categories. In this section we will review the theory of $\ring$-linear categories, $G$-graded $\ring$-linear categories, differential categories, and $(G,\lambda)$-graded dg categories. All four of these constructions can be treated in a unified way by viewing them as categories enriched in four different monoidal categories. We will review what we need of enriched category theory and apply it to the four examples of interest. For a more detailed overview, see \cite[Chapters 2.1 and 2.6]{EilenbergKelly}. 

\subsubsection{Monoidal categories}

Recall that the Cartesian product $M_1 \times M_2$ of two categories $M_1$ and $M_2$ has objects and morphisms given by pairs of objects and morphisms of the two factors, with composition defined componentwise. A monoidal category is a category $M$ equipped with a functor $\otimes: M \times M \to M$, an object $I \in M$, a natural isomorphism $\alpha: (- \otimes -) \otimes - \to - \otimes (- \otimes -)$ called the associator, and natural isomorphisms $\lambda: I \otimes - \to -$ and $\rho: - \otimes I \to -$ called the left and right unitor respectively, satisfying:
\[
\alpha_{w,x,y \otimes z} \circ \alpha_{w \otimes x, y, z} = (\id_w \otimes \alpha_{x,y,z}) \circ \alpha_{w, x \otimes y, z} \circ (\alpha_{w,x,y} \otimes \id_z)
\]
as morphisms from $((w \otimes x) \otimes y) \otimes z$ to $w \otimes (x \otimes (y \otimes z))$ for all objects $w,x,y,z$ of $M$ (the pentagon identity) as well as
\[
\rho_x \otimes \id_y = (\id_x \otimes \lambda_y) \circ \alpha_{x,I,y}
\]
as morphisms from $(x \otimes I) \otimes y$ to $x \otimes y$ for all objects $x,y$ of $M$ (the triangle identity).

Recall that $\ring$ denotes a commutative ring of characteristic $2$.
\begin{example}\label{ex:MonoidalCats}
As with algebras, we will consider four examples:
\begin{itemize}
\item
The category of $\ring$-modules and $\ring$-linear maps has a monoidal structure with $\otimes$ given by the tensor product of $\ring$-modules and $I$ given by $\ring$ as a module over itself. 

\item
For a group $G$, the category of $G$-graded $\ring$-modules and degree-preserving $\ring$-linear maps has a similar monoidal structure: we have 
\[
(M \otimes N)_g := \bigoplus_{g_1, g_2 \in G: g_1 g_2 = g} M_{g_1} \otimes N_{g_2}.
\]
The monoidal unit $I$ is $\ring$ concentrated in the identity degree.

\item
The category of differential $\ring$-modules and homomorphisms of differential modules (i.e. $\ring$-module maps $f$ with $\partial \circ f = f \circ \partial$) also has a monoidal structure: $\otimes$ is given by the tensor product of differential $\ring$-modules, and $I$ is given by $\ring$ as a module over itself with zero differential. 

\item
The category of $(G,\lambda)$-graded dg $\ring$-modules and homomorphisms of dg modules (i.e. degree-preserving homomorphisms of differential modules) has a similar monoidal structure, defined as in the above two items. The monoidal unit $I$ is $\ring$ with zero differential and concentrated in the identity degree.
\end{itemize}
\end{example}

\subsubsection{Monoidal functors}

One can also consider functors between monoidal categories.
\begin{definition}
Let $(M,\otimes,I,\alpha,\rho,\lambda)$ and $(M',\otimes',I',\alpha',\rho',\lambda')$ be monoidal categories. A (lax) \emph{monoidal functor} from $M$ to $M'$ is a triple $(H,\mu,\epsilon)$ where $H: M \to M'$ is a functor, $\mu: H(-) \otimes' H(-) \to H(- \otimes -)$ is a natural transformation of functors from $M \times M$ to $M'$,  and $\epsilon: I' \to H(I)$ is a morphism in $M'$, such that
\[
H(\alpha_{x,y,z}) \circ \mu_{x \otimes y, z} \circ (\mu_{x,y} \otimes' \id_{H(z)}) = \mu_{x, y \otimes z} \circ (\id_{H(x)} \otimes' \mu_{y,z}) \circ \alpha'_{H(x),H(y),H(z)}
\]
as morphisms from $(H(x) \otimes' H(y)) \otimes ' H(z)$ to $H(x \otimes (y \otimes z))$ for all objects $x,y,z$ of $M$,
\[
\lambda'_{H(x)} = H(\lambda_x) \circ \mu_{I,x} \circ (\epsilon \otimes' \id_{H(x)})
\]
as morphisms from $I' \otimes' H(x)$ to $H(x)$ for all objects $x$ of $M$, and
\[
\rho'_{H(x)} = H(\rho_x) \circ \mu_{x,I} \circ (\id_{H(x)} \otimes' \epsilon)
\]
as morphisms from $H(x) \otimes' I'$ to $H(x)$ for all objects $x$ of $M$.
\end{definition}

\begin{example}\label{ex:HomologyMonoidalFunctors}
Homology gives us a monoidal functor from (differential $\ring$-modules, homomorphisms) to ($\ring$-modules, $\ring$-linear maps) sending a differential $\ring$-module $N$ to $H_*(N)$ and sending a homomorphism of differential modules to the induced map on homology. The natural transformation $\mu$ is given by the map from $H_*(N_1) \otimes H_*(N_2)$ to $H_*(N_1 \otimes N_2)$ sending $[n_1] \otimes [n_2]$ to $[n_1 \otimes n_2]$ for each pair $(N_1,N_2)$ of differential $\ring$-modules.

We have a similar monoidal functor from ($(G,\lambda)$-graded dg $\ring$-modules, homomorphisms) to ($G$-graded $\ring$-modules, degree-preserving $\ring$-linear maps).
\end{example}

\subsubsection{Categories enriched in a monoidal category}

\begin{definition}
If $(M,\otimes,I,\alpha,\lambda,\rho)$ is a monoidal category, then a \emph{category} $C$ \emph{enriched in} $M$ consists of the following data:
\begin{itemize}
\item A class of objects $\Ob C$;
\item For $X,Y \in \Ob C$, an object $\Hom(X,Y)$ of $M$;
\item For $X \in \Ob C$, a morphism $j_X: I \to \Hom(X,X)$ in $M$;
\item For $X, Y, Z \in \Ob C$, a morphism $\circ_{X,Y,Z}: \Hom(Y,Z) \otimes \Hom(X,Y) \to \Hom(X,Z)$ in $M$.
\end{itemize}
We require that, for $W,X,Y,Z \in \Ob C$, we have
\[
\circ_{W,X,Z} \circ (\circ_{X,Y,Z} \otimes \id_{\Hom(W,X)}) = \circ_{W,Y,Z} \circ (\id_{\Hom(Y,Z)} \otimes \circ_{W,X,Y}) \circ \alpha_{\Hom(Y,Z),\Hom(X,Y),\Hom(W,X)}
\]
as morphisms in $M$ from $(\Hom(Y,Z) \otimes \Hom(X,Y)) \otimes \Hom(W,X)$ to $\Hom(W,Z)$. We also require that, for $X,Y \in \Ob C$, we have 
\[
\lambda_{\Hom(X,Y)} = \circ_{X,Y,Y} \circ (j_Y \otimes \id_{\Hom(X,Y)})
\]
as morphisms in $M$ from $I \otimes \Hom(X,Y)$ to $\Hom(X,Y)$ and
\[
\rho_{\Hom(X,Y)} = \circ_{X,X,Y} \circ (\id_{\Hom(X,Y)} \otimes j_X)
\]
as morphisms in $M$ from $\Hom(X,Y) \otimes I$ to $\Hom(X,Y)$.
\end{definition}

We will consider four types of enriched category; recall that homomorphisms of differential and dg $\ring$-modules were defined in Example~\ref{ex:MonoidalCats}.
\begin{example}
A $\ring$-linear category is a category enriched in ($\ring$-modules, $\ring$-linear maps). A $G$-graded $\ring$-linear category is a category enriched in ($G$-graded $\ring$-modules, degree-preserving $\ring$-linear maps).

A differential category over $\ring$ is a category enriched in (differential $\ring$-modules, homomorphisms). A $(G,\lambda)$-graded dg category over $\ring$ (a.k.a.~$\ring$-linear dg category) is a category enriched in ($(G,\lambda)$-graded dg $\ring$-modules, homomorphisms).
\end{example}

\subsubsection{Change of enrichment}

We can change enrichment via monoidal functors as follows.
\begin{definition}\label{def:ChangeOfEnrichment}
If $C$ is a category enriched in the monoidal category $M$ and $(H,\mu,\epsilon): M \to M'$ is a monoidal functor, we obtain a $M'$-enriched category $C'$ with the same objects as $C$ by defining $\Hom_{C'}(X,Y) = H(\Hom_C(X,Y))$ for objects $X,Y$ of $C$. The identity morphisms $j'_X: I_{M'} \to \Hom_{C'}(X,X)$ are defined to be
\[
I_{M'} \xrightarrow{\epsilon} H(I_M) \xrightarrow{H(j_X)} H(\Hom_C(X,X)) = \Hom_{C'}(X,X).
\]
The composition morphisms are defined as
\[
\Hom_{C'}(Y,Z) \otimes' \Hom_{C'}(X,Y) \xrightarrow{\mu} H(\Hom_C(Y,Z) \otimes \Hom_C(X,Y)) \xrightarrow{H(\circ)} \Hom_{C'}(X,Z),
\]
where $\mu = \mu_{\Hom_C(Y,Z), \Hom_C(X,Y)}$.
\end{definition}

\begin{example}
Given a differential category $C$ over $\ring$, we can get a $\ring$-linear category $H_*(C)$ by applying Definition~\ref{def:ChangeOfEnrichment} to the homology monoidal functor from Example~\ref{ex:HomologyMonoidalFunctors}. Similarly, from a $(G,\lambda)$-graded dg category $C$ over $\ring$, we can get a $G$-graded $\ring$-linear category $H_*(C)$. 
\end{example}

\begin{remark}
Concretely, $H_*(C)$ is obtained from $C$ by taking the homology of the morphism spaces of $C$. The theory of monoidal functors and change of enrichment ensures that the resulting $\ring$-linear category is well-defined, although one could also check this fact directly.
\end{remark}

\subsubsection{Enriched functors}

We can also consider enriched functors between enriched categories. 
\begin{definition}\label{def:EnrichedFunctor}
If $C$ and $C'$ are categories enriched in a monoidal category $M$, an \emph{enriched functor} $F$ from $C$ to $C'$ is the data of an object $F(X)$ of $C'$ for every object $X$ of $C$, as well as a morphism $F_{X,Y}: \Hom_C(X,Y) \to \Hom_{C'}(F(X),F(Y))$ in $M$ for all pairs of objects $(X,Y)$ of $C$, such that the diagrams
\[
\xymatrix{\Hom_C(Y,Z) \otimes \Hom_C(X,Y) \ar[rr]^-{\circ_{X,Y,Z}} \ar[d]_-{F_{Y,Z} \otimes F_{X,Y}} & & \Hom_C(X,Z) \ar[d]^-{F_{X,Z}} \\ 
\Hom_{C'}(F(Y),F(Z)) \otimes \Hom_{C'}(F(X),F(Y)) \ar[rr]_-{\circ_{F(X),F(Y),F(Z)}} & & \Hom_{C'}(F(X),F(Z))}
\]
and
\[
\xymatrix{& I \ar[dl]_-{j_X} \ar[dr]^-{j_{F(X)}} & \\
\Hom_C(X,X) \ar[rr]^-{F_{X,X}} & & \Hom_{C'}(F(X),F(X))}
\]
commute for all objects $X,Y,Z$ of $C$.
\end{definition}

\begin{example}
We define \emph{$\ring$-linear functors}, \emph{$G$-graded $\ring$-linear functors}, \emph{differential functors}, and \emph{$(G,\lambda)$-graded dg functors} using the general notion of enriched functor from Definition~\ref{def:EnrichedFunctor}.
\end{example}

One can define composition of enriched functors as well as identity enriched functors in a natural way. Thus, for a given monoidal category $M$, we have a category whose objects are $M$-enriched categories and whose morphisms are enriched functors. In particular, we have the following four categories:
\begin{itemize}
\item ($\ring$-linear categories, $\ring$-linear functors)
\item ($G$-graded $\ring$-linear categories, $G$-graded $\ring$-linear functors)
\item (differential categories over $\ring$, differential functors)
\item ($(G,\lambda)$-graded dg categories over $\ring$, dg functors).
\end{itemize}

\subsubsection{Changing enrichment on functors}

If we change enrichment on categories, we can change enrichment on functors correspondingly.
\begin{definition}\label{def:ChangeEnrichmentOfFunctor}
Suppose $C$ and $C'$ are $M$-enriched categories and $F: C \to C'$ is an enriched functor. Let $H: M \to M'$ be a monoidal functor; write $H(C)$ and $H(C')$ for the $M'$-enriched categories obtained from $C$ and $C'$ by change of enrichment. Define an enriched functor $H(F): H(C) \to H(C')$ by letting $H(F) := F$ on objects, and defining 
\[
H(F)_{X,Y}: H(\Hom_C(X,Y)) \to H(\Hom_{C'}(F(X),F(Y)))
\]
to be $H(F)_{X,Y} := H(F_{X,Y})$. We leave it to the reader to verify commutativity of the required diagrams.
\end{definition}

Change of enrichment is compatible with composition of enriched functors and identity enriched functors, so $H$ gives us a functor from ($M$-enriched categories, enriched functors) to ($M'$-enriched categories, enriched functors). 

\begin{example}\label{ex:RLinearFunctorsEtc}
Given a differential functor $F: C \to C'$ where $C$ and $C'$ are differential categories over $\ring$, Definition~\ref{def:ChangeEnrichmentOfFunctor} gives us a $\ring$-linear functor $H_*(F): H_*(C) \to H_*(C')$. Similarly, if $F$ is a functor between $(G,\lambda)$-graded dg categories over $\ring$, we get a $G$-graded $\ring$-linear functor $H_*(F)$.
\end{example}

To summarize, homology gives us a functor from (differential categories over $\ring$, differential functors) to ($\ring$-linear categories, $\ring$-linear functors). In the graded case, it gives us a functor from ($(G,\lambda)$-graded dg categories over $\ring$, dg functors) to ($G$-graded $\ring$-linear categories, $G$-graded $\ring$-linear functors).

\subsection{Algebras and categories}\label{sec:AlgsAndCats}

Let $V$ be a finite set. Following \cite[Appendix C.2]{DGQuotients}, if $\mc C$ is any $\ring$-linear category with object set $V$, let 
\[
\Alg_{\mc C} := \bigoplus_{v_1, v_2 \in V} \Hom_{\mc C}(v_2,v_1).
\]
We may view $\Alg_{\mc C}$ as an algebra over $\ring$; multiplication is induced by the composition map
\[
\Hom_{\mc C}(v_2,v_1) \otimes \Hom_{\mc C}(v_3,v_2) \to \Hom_{\mc C}(v_3,v_1)
\]
and is zero when two morphisms are not composable.  Since $V$ is finite, $\Alg_{\mc C}$ is unital; the unit is $\sum_{v \in V} \Ib_v$ where $\Ib_v \in \Hom_{\mc C}(v,v)$ is the identity morphism. These elements $\Ib_v$ are idempotents in $\Alg_{\mc C}$ that are pairwise orthogonal: we have $\Ib_{v_1} \Ib_{v_2} = 0$ if $v_1 \neq v_2$. Also note that the natural map $\ring \to \Alg_{\mc C}$ has image in the center of $\Alg_{\mc C}$, since the morphism spaces in $\mc C$ are modules over $\ring$ (not just bimodules).

It can be even more useful to view $\Alg_{\mc C}$ as an algebra over $\IdemRing$, where $\IdemRing = \ring^V$ is the ring of functions from $V$ into $\ring$. By slight abuse of notation, denote the basis element of $\IdemRing$ corresponding to $v \in V$ by $\Ib_v$ as well. We have a $\ring$-algebra homomorphism from $\IdemRing$ to $\Alg_{\mc C}$ sending $\Ib_v$ to $\Ib_v$. Thus, we can view $\Alg_{\mc C}$ as an $\IdemRing$-algebra; the natural map $\ring \to \IdemRing \to \Alg_{\mc C}$ is the structure map of $\Alg_{\mc C}$ as a $\ring$-algebra, so it has image in $Z(\Alg_{\mc C})$. 

Conversely, if $\mc A$ is a $\IdemRing$-algebra such that the natural map $\ring \to \IdemRing \to \mc A$ has image in $Z(\mc A)$, one can define a $\ring$-linear category $\Cat_{\mc A}$ with set of objects $V$ by defining $\Hom_{\Cat_{\mc A}}(v_2,v_1) := \Ib_{v_1}  \mc A \Ib_{v_2}$. The condition on the map $\ring \to \IdemRing \to \mc A$ ensures that composition is a $\ring$-linear map whose domain is a tensor product of $\ring$-modules (not just of $\ring$-bimodules).

The operations $\mc C \mapsto \Alg_{\mc C}$ and $\mc A \mapsto \Cat_{\mc A}$ extend to inverse equivalences of categories between ($\ring$-linear categories with object set $V$, $\ring$-linear functors that are the identity on objects) and ($\IdemRing$-algebras, homomorphisms), where we assume as usual that $\IdemRing$-algebras $\mc A$ satisfy the standard condition on $\ring \to \IdemRing \to \mc A$. Similarly, we have inverse equivalences of categories between ($G$-graded $\ring$-linear categories with object set $V$, $G$-graded $\ring$-linear functors that are the identity on objects) and ($G$-graded $\IdemRing$-algebras, homomorphisms).

Adding in differentials, the same constructions give inverse equivalences of categories between (differential categories with object set $V$, differential functors that are the identity on objects) and (differential algebras over $\IdemRing$, homomorphisms) as well as between ($(G,\lambda)$-graded dg categories with object set $V$, dg functors that are the identity on objects) and (dg algebras over $\IdemRing$, homomorphisms). 

\begin{remark}
Explicitly, if $F: {\mc C}_1 \to {\mc C}_2$ is a dg functor that is the identity on objects, then $F$ gives us a chain map from $\Hom_{{\mc C}_1}(v_2,v_1)$ to $\Hom_{{\mc C}_2}(v_2,v_1)$ for all $v_1, v_2 \in V$, compatible with composition. These chain maps assemble into a dg $\IdemRing$-algebra homomorphism from $\Alg_{{\mc C}_1}$ to $\Alg_{{\mc C}_2}$. In the other direction, if $f: {\mc A}_1 \to {\mc A}_2$ is a dg $\IdemRing$-algebra homomorphism, then $f$ restricts to a chain map from $\Ib_{v_1} {\mc A}_1 \Ib_{v_2}$ to $\Ib_{v_1} {\mc A}_2 \Ib_{v_2}$ for all $v_1, v_2 \in V$, giving us a dg functor from $\Cat({\mc A}_1)$ to $\Cat({\mc A}_2)$ that is the identity on objects.
\end{remark}

The equivalences $\Alg$ and $\Cat$ interact with homology as follows.
\begin{proposition}\label{prop:CatsAlgsHomology}
The diagram of functors
\[
\xymatrix{
\textrm{(diff cats with objs }V \textrm{, diff functors id.~on objs)} \ar@<1ex>[r]^-{\Alg} \ar[d]_-{H_*} & \textrm{(diff }\IdemRing \textrm{-algs, homomorphisms)} \ar[l]^-{\Cat} \ar[d]^-{H_*} \\
(\ring\textrm{-linear cats with objs }V \textrm{, }\ring\textrm{-linear functors id.~on objs)} \ar@<1ex>[r]^-{\Alg} & (\IdemRing\textrm{-algs, homomorphisms) \ar[l]^-{\Cat}}
}
\]
is commutative up to natural isomorphism of functors. Analogous statements hold when gradings are present. 
\end{proposition}

\begin{proof}
Left to reader.
\end{proof}

Thus, for a dg category $C$, we have $H_*(\Alg_C) \cong \Alg_{H_*(C)}$, and for a dg algebra $A$, we have $H_*(\Cat_A) \cong \Cat_{H_*(A)}$. 

\subsubsection{Orthogonal idempotents}

Let $\A$ be a $\ring$-algebra whose structure map $\ring \to \A$ has image in $Z(\A)$; let $O$ be a subset of $\A$ consisting of pairwise orthogonal idempotents satisfying $1_\A = \sum_{\Ib \in O} \Ib$. We can view $\A$ as a $\ring^O$-algebra via the $\ring$-algebra homomorphism from $\ring^O$ to $\A$ sending the basis element of $\ring^O$ corresponding to $\Ib \in O$ to the actual element $\Ib \in \A$. The fact that the elements of $O$ are orthogonal idempotents ensures that this $\ring$-linear map respects algebra multiplication; the fact that the elements of $O$ sum to $1 \in \A$ ensures that it respects the units of the algebras. The natural map $\ring \to \ring^O \to \A$ is the original structure map of $\A$, so it has image in $Z(\A)$.

Similar statements hold for $G$-graded, differential, and $(G,\lambda)$-graded dg algebras over $\ring$. In the presence of a differential, we require that $\partial(\Ib) = 0$ for all $\Ib \in O$. In this case, we have a homomorphism of differential $\ring$-algebras from $\ring^O$ to $\A$, i.e. a differential $\ring^O$-algebra structure on $\A$.

Corresponding to the $\ring^O$-algebra $\A$, by above we have a $\ring$-linear category $\Cat_{\A}$ such that $\A \cong \Alg_{\Cat_{\A}}$. The isomorphism from $\A$ to $\Alg_{\Cat_{\A}}$ sends $a \in A$ to $(\Ib a \Jb)_{\Ib, \Jb}$. In particular, we have the following lemma.
\begin{lemma}\label{lem:OrthogonalIdempotents}
Let $\A$ be a $\ring$-algebra with a finite subset $O$ consisting of pairwise orthogonal idempotents satisfying $1_\A = \sum_{\Ib \in O} \Ib$.
Writing $\Alg_{\Cat_{\A}}$ as $\bigoplus_{\Ib, \Jb} \Ib \A \Jb$, the map
\begin{align*}
\A &\to \bigoplus_{\Ib, \Jb} \Ib \A \Jb \\
a & \mapsto (\Ib a \Jb)_{\Ib, \Jb}
\end{align*}
is an isomorphism of $\ring^O$-algebras. Similar statements hold for $G$-graded, differential, and $(G,\lambda)$-graded dg algebras over $\ring$.
\end{lemma}

If $\A$ has a differential, the homology of $\A$ admits a similar decomposition. From the above isomorphism $\A \xrightarrow{\cong} \Alg_{\Cat_{\A}}$, we get isomorphisms 
\[
H_*(\A) \xrightarrow{\cong} H_*(\Alg_{\Cat_{\A}}) \xrightarrow{\cong} \Alg_{H_*(\Cat_{\A})} \xrightarrow{\cong} \Alg_{\Cat_{H_*(\A)}}
\]
using Proposition~\ref{prop:CatsAlgsHomology}. The first map sends $[a]$ to $[(\Ib a \Jb)_{\Ib,\Jb}]$. The second map sends this to $([\Ib a \Jb])_{\Ib,\Jb}$. The third map sends this to $(\Ib [a] \Jb)_{\Ib,\Jb}$. We get the following analogue of Lemma~\ref{lem:OrthogonalIdempotents}.
\begin{lemma}\label{lem:OrthogonalIdempotentsHomology}
Let $\A$ be a differential algebra over $\ring$ with a finite subset $O$ consisting of pairwise orthogonal idempotents satisfying $1_\A = \sum_{\Ib \in O} \Ib$.
Writing $\Alg_{\Cat_{H_*(\A)}}$ as $\bigoplus_{\Ib, \Jb} \Ib H_*(\A) \Jb$, the map
\begin{align*}
H_*(\A) &\to \bigoplus_{\Ib, \Jb} \Ib H_*(\A) \Jb \\
[a] & \mapsto (\Ib [a] \Jb)_{\Ib, \Jb}
\end{align*}
is an isomorphism of $\ring^O$-algebras. A similar statement holds for $(G,\lambda)$-graded dg algebras over $\ring$.
\end{lemma}

\begin{remark}
\label{rem:HIAJ=IHAJ}
Under the hypotheses of Lemma \ref{lem:OrthogonalIdempotentsHomology}, let $\Ib$ and $\Jb$ in $O$ be two orthogonal idempotents. Then, the map above canonically identifies $H_*(\Ib \A \Jb)$ with $\Ib H_*(\A)\Jb$, so we can use these expressions interchangeably.
\end{remark}

\subsubsection{Quasi-isomorphisms}

We will not need the general notion of quasi-equivalence of dg categories; the following restricted definition will suffice.
\begin{definition}\label{def:DGCatQI}
Let ${\mc C}_1$ and ${\mc C}_2$ be differential categories over $\ring$ with object set $V$. Let $F$ be a differential functor from ${\mc C}_1$ to ${\mc C}_2$ that is the identity on objects of ${\mc C}_1$. We call $F$ a \emph{quasi-isomorphism} if the induced functor $H_*(F): H_*({\mc C}_1) \to H_*({\mc C}_2)$ (see Example~\ref{ex:RLinearFunctorsEtc} above) is an isomorphism of $\ring$-linear categories. We define quasi-isomorphisms of $(G,\lambda)$-graded dg categories similarly, again restricting our attention to functors that are the identity on objects.
\end{definition}

The inverse equivalences between (differential categories over $\ring$ with object set $V$, differential functors that are the identity on objects) and (differential algebras over $\IdemRing$, homomorphisms) from Section~\ref{sec:AlgsAndCats} send quasi-isomorphisms of differential categories over $\ring$ to quasi-isomorphisms of differential algebras over $\IdemRing$ and vice-versa, by Proposition~\ref{prop:CatsAlgsHomology}. Similar statements hold in the dg setting.

\bibliographystyle{alpha}
\bibliography{biblio}

\end{document}